\documentclass[10pt,letterpaper]{article}
\usepackage[margin = 1in]{geometry}
\usepackage{amssymb,amsmath,amsthm,amsfonts}
\usepackage{times}
\usepackage[title]{appendix}

\newtheorem{assumption}{Assumption}
\newtheorem{proposition}{Proposition}

\newtheorem{remark}{Remark}
\usepackage[]{pdfpages}
\newtheorem{theorem}{Theorem}
\newtheorem*{theorem*}{Theorem}
\newtheorem*{proposition*}{Proposition}
\newtheorem{lemma}{Lemma}
\usepackage[ruled,vlined]{algorithm2e}
\usepackage{mathrsfs}
\newtheorem{corollary}{Corollary}
\newcommand\independent{\protect\mathpalette{\protect\independenT}{\perp}}
\def\independenT#1#2{\mathrel{\rlap{$#1#2$}\mkern2mu{#1#2}}}
\LinesNumbered

\usepackage{graphicx}
\usepackage[inline]{enumitem}
\usepackage[colorlinks,
linkcolor = black,
            urlcolor  = blue,
            citecolor = black]{hyperref}
\usepackage{natbib}
\date{\today}
\title{Parametrization, Prior Independence, and the Semiparametric Bernstein-von Mises Theorem for the Partially Linear Model}
\author{Christopher D. Walker\footnote{Harvard University, Department of Economics: \href{mailto:cwalker@g.harvard.edu}{cwalker@g.harvard.edu}. I thank the Editor, an Associate Editor, and three anonymous referees for providing comments that have significantly improved my paper. I also thank Isaiah Andrews, Nick Cao, Kevin Chen, Anna Mikusheva, Bas Sanders, Neil Shephard, James Stratton, Rahul Singh, and Elie Tamer for their helpful comments. All errors are mine.}}
\begin{document}
\maketitle

\begin{abstract}
I prove a semiparametric Bernstein-von Mises theorem for a partially linear regression model with independent priors for the low-dimensional parameter of interest and the infinite-dimensional nuisance parameters. My result avoids a challenging prior invariance condition that arises from a loss of information associated with not knowing the nuisance parameter. The key idea is to employ a feasible reparametrization of the partially linear regression model that reflects the semiparametric structure of the model. This allows a researcher to assume independent priors for the model parameters while automatically accounting for the loss of information associated with not knowing the nuisance parameters. The theorem is verified for uniform wavelet series priors and Mat\'{e}rn Gaussian process priors.
\end{abstract}

\section{Introduction}
\subsection{Overview}
This paper is concerned with Bayesian inference for a partially linear regression model. The partially linear model states that the outcome $Y \in \mathbb{R}$ is related to covariates $X \in \mathbb{R}$ and $W \in \mathbb{R}^{d_{w}}$ via the regression model
\begin{align}\label{eq:original}
    Y = X\beta_{0} + \eta_{0}(W) + U,
\end{align}
where $\beta_{0}$ is a scalar parameter, $\eta_{0}$ is a function-valued parameter, and $U$ is a scalar unobservable that satisfies $E[U|X,W] = 0$ and $Var[U|X,W] = \sigma_{01}^{2}$ for some scalar $\sigma_{01}^{2}\in (0,\infty)$. Partially linear models appear in a variety of applications including modeling the relationship between temperature and electricity demand \citep{engle1986semiparametric}, economic models of production with endogenous input choices \citep{olley1996dynamics,levinsohn2003estimating,wooldridge2009estimating,ackerberg2015identification}, and sample selection models \citep{ahn1993semiparametric,das2003nonparametric}. Moreover, the model has also gained recent theoretical attention as a leading example for which debiased machine learning methods are applied (see, for example, \cite{chernozhukov2018double}).

My analysis views $\beta_{0}$ as the parameter of interest while treating $\eta_{0}$ as an infinite-dimensional nuisance parameter. The goal is to derive conditions under which the \textit{Bernstein-von Mises theorem} holds (i.e., the marginal posterior for $\beta$ is asymptotically normal and matches an efficient frequentist estimator). A key challenge with Bernstein-von Mises theory for the partially linear model is that there is loss of information associated with not knowing the nuisance parameter $\eta_{0}$, and, as a result, the natural assumption that $\beta$ and $\eta$ are independent a priori leads to \textit{prior invariance}. Prior invariance refers to constructing a change of variables that respects the model's semiparametric structure and leaves the prior for the nuisance function $\eta$ roughly unchanged. It was first emphasized as a general condition for semiparametric Bernstein-von Mises theorems in the important papers of \cite{castillo2012semiparametric,castillo2012semiparametricBVM}, and it introduces both practical and technical issues.\footnote{Prior invariance is also encountered in statistical functional estimation \citep{rivoirard2012bernstein,castillo2015bernstein,ray2020semiparametric}. Moreover, the condition has been shown to be an almost necessary condition for the Bernstein-von Mises theorem in some models (e.g., the curve alignment model in \cite{castillo2012semiparametricBVM}).} From a practical standpoint, prior invariance stipulates that the prior for $\eta$ must be adequate for estimating of the true control function $\eta_{0}$ \textit{and} the conditional expectation $m_{02}$ of $X$ given $W$, a condition that restricts the smoothness $m_{02}$ relative to $\eta_{0}$, and, as a result, limits the data-generating processes for which the Bernstein-von Mises theorem can be applied. From a technical perspective, showing the prior for $\eta$ is roughly unchanged after the change of variables can be difficult to verify for some priors because it involves performing an infinite-dimensional change of measure (e.g., this is a key reason that \cite{castillo2012semiparametric,castillo2012semiparametricBVM} focuses on Gaussian priors).

The main contribution of this paper is a Bernstein-von Mises theorem for the partially linear model that bypasses prior invariance. Theorem \ref{thm:BVM} in Section \ref{sec:generalbvm} establishes this result, and the theorem is verified for two examples in Section \ref{sec:examples}. A common theme in the examples is that there are no cross-restrictions on the smoothness of nuisance functions, which demonstrates that my approach can eliminate the aforementioned practical challenge of requiring a prior for a \textit{single} nuisance parameter (i.e., $\eta$) be suitable for estimation of \textit{multiple} quantities (i.e., $\eta_{0}$ and $m_{02}$). To make things concrete, my proposal is to embed an interest-respecting reparametrization of (\ref{eq:original}) given by
\begin{align}\label{eq:struc}
Y = m_{01}(W) + [X-m_{02}(W)]\beta_{0} + U,  
\end{align}
where $m_{01}(W)=E_{P_{0}}[Y|W]$ and $m_{02}(W)=E_{P_{0}}[Y|W]$, within a (quasi-)likelihood model that is suitable for identification of $(\beta_{0},m_{01},m_{02})$, and then perform Bayesian inference by placing independent priors on $\beta$ and $m=(m_{1},m_{2})$. The transformed regression model (\ref{eq:struc}) is commonly referred to as the \textit{\cite{robinson1988root} transformation}, and is observationally equivalent to (\ref{eq:original}) because $\eta_{0}$ is identified as $m_{01}-\beta_{0}m_{02}$. A key property of the (quasi-)likelihood model is that there is no loss of information associated with not knowing the nuisance parameter $m_{0}=(m_{01},m_{02})$ (i.e., it is \textit{adaptive} in the sense of \cite{bickel1982adaptive}). Consequently, there is no need to change variables because ordinary/parametric local asymptotic normality expansions respect the semiparametric structure of the model. For terminology, I refer to (\ref{eq:original}) as the \textit{$(\beta,\eta)$-parametrization} and (\ref{eq:struc}) as the \textit{$(\beta,m)$-parametrization} throughout.

To my knowledge, Theorem \ref{thm:BVM} is the first Bernstein-von Mises theorem for the $(\beta,m)$-parametrization of the partially linear model. It has several important takeaways. The first takeaway is that the findings demonstrate that the choice of parametrization can influence the conditions under which the semiparametric Bernstein-von Mises theorem holds. This is demonstrated explicitly in Section \ref{sec:comparison} in which the $(\beta,\eta)$- and $(\beta,m)$-parametrizations are compared for Mat\'{e}rn Gaussian process priors, a class of Gaussian process priors popular in applications (see, for example, \cite{williams2006gaussian}). Since the $(\beta,m)$-parametrization is orthogonal (and the $(\beta,\eta)$-parametrization is not), this takeaway is similar to the recent frequentist debiased machine learning literature \citep{chernozhukov2018double} and is reminiscent of classical results on improving parametric likelihood estimators via orthogonal parametrizations \citep{cox1973parameter}. The second takeaway is that the use of reparametrization to eliminate prior invariance conditions avoids data-driven prior/posterior adjustments. For this reason, my proposal offers a conceptually distinct approach to obtaining debiased Bernstein-von Mises theorems relative to those encountered in a growing literature on data-driven prior/posterior adjustments \citep{yang2015semiparametric,ray2019debiased,ray2020semiparametric,breunig2022double,yiu2025semiparametric}. Determining whether the approach of this paper is more generally applicable is an important area for future research. A final takeaway is that the Bernstein-von Mises theorem can be valid even if the (quasi-)likelihood model is misspecified. This reflects classical results on quasi-maximum likelihood estimation \citep{gourieroux1984pseudo,white1994estimation}, and Section \ref{sec:waveletex} makes the claim precise by verifying Theorem \ref{thm:BVM} for uniform wavelet series priors without requiring the sampling model be correctly specified.
\subsection{Related Literature}
This paper connects to several active literatures in statistics and econometrics. First, it is related to the literature on semiparametric Bernstein-von Mises theorems. A Bernstein-von Mises theorem for the $(\beta,\eta)$-parametrization with independent priors for $\beta$ and $\eta$ is studied in \cite{shen2002asymptotic}, \cite{bickel2012semiparametric}, \cite{yang2015semiparametric}, and \cite{xie2020adaptive}. \cite{castillo2012semiparametric,castillo2012semiparametricBVM} are also relevant to my paper because these important papers emphasize the role of the prior invariance condition in separated semiparametric models with information loss. Relatedly, \cite{rivoirard2012bernstein} and \cite{castillo2015bernstein} demonstrated the importance of prior invariance for general smooth functionals of nonparametric models (with the latter treating separated semiparametric models as a special case). \cite{yang2015semiparametric,ray2020semiparametric,breunig2022double}, and \cite{yiu2025semiparametric} propose data-driven prior/posterior corrections to prove Bernstein-von Mises theorems under weaker conditions than their unmodified counterparts. A key difference between these papers and my proposal is that I do not require any data-driven adjustment because the lack of prior invariance follows from the adaptivity of the $(\beta,m)$-parametrization. \cite{yang2019posterior} derives Bernstein-von Mises theorem for a single coordinate in a high-dimensional linear regression model under sparsity using a different type of reparametrization and a prior correction. Although I do not explicitly study high-dimensional linear regression, Section \ref{sec:conclusion} outlines how my approach can be extended to this setting and an interesting subject for future research would be determining whether my approach leads to results similar to \cite{yang2019posterior} in this setting. Other papers on semiparametric Bernstein-von Mises theorems include \cite{kim2006bernstein}, \cite{dejonge2013semiparametric}, \cite{norets2015bayesian}, \cite{chae2019semi}, \cite{nickl2019bernstein}, \cite{nickl2020bernstein}, \cite{monard2021statistical}, and \cite{l2023semiparametric}. 

Second, some of the auxiliary results in Section \ref{sec:examples} connect this paper to the literature on posterior contraction rates in nonparametric models. Specifically, Propositions \ref{prop:suffgeneral} and \ref{prop:suffgeneralsieve} provide a general approach to deriving posterior contraction rates for nuisance functions $m_{01}$ and $m_{02}$. These propositions are based on empirical process conditions that arise from direct expansions of the (quasi-)likelihood ratio process, and, for this reason, are conceptually similar to the convergence rate framework of \cite{shen2001rates} in which empirical process-type conditions for likelihood ratios (e.g., bracketing integrals) are used to derive posterior contraction rates. Since Propositions \ref{prop:suffgeneral} and \ref{prop:suffgeneralsieve} concern posterior contraction rates in (possibly) misspecified Gaussian nonparametric regression, the auxiliary results in Section \ref{sec:examples} are also related to the theory presented in Section 4 of \cite{kleijn2006misspecification}. Other papers on posterior contraction rates in general nonparametric models include \cite{ghosal2000convergence,ghosal2007convergence,castillo2008lower,vaart2008rates,10.1214/08-AOS678,van2011information,gine2011rates,dejonge2012adaptive,castillo2014supremum}, and \cite{shen2015adaptive}.

Finally, since the Bernstein-von Mises theorem establishes asymptotic equivalence between Bayesian and frequentist estimators, my paper also dovetails with the literature on frequentist semiparametric estimation. The $(\beta,m)$-parametrization was first utilized in \cite{robinson1988root} to derive asymptotically normal least squares estimators of $\beta_{0}$ based on kernel estimators of $m_{01}$ and $m_{02}$. \cite{DONALD199430} also apply the partialling out technique of \cite{robinson1988root} to show that asymptotic normality of the least squares estimators of $\beta_{0}$ can be achieved under very mild conditions when series regression methods are used to estimate the nuisance functions. My paper provides a Bayesian analogue to these papers because it shows the $(\beta,m)$-parametrization can be used to derive a Bernstein-von Mises theorem that bypasses prior invariance. The proposed Bayesian inference framework is also related to the literature on profile likelihood estimation \citep{severini1992profile, murphy2000profile} because the transformed regression model (\ref{eq:struc}) mimics the Gaussian profile maximum likelihood problem for (\ref{eq:original}). Since $(\beta,m)$-parametrization is orthogonal, my paper is also related to the literature on two-step semiparametric estimation based on Neyman orthogonal moment conditions (e.g., \cite{andrews1994asymptotics, newey1994asymptotic,chernozhukov2018double}). Using a (quasi-)likelihood to perform Bayesian inference a low-dimensional parameter of interest defined by conditional moment restrictions also connects my proposal to the quasi-maximum likelihood estimation literature \citep{gourieroux1984pseudo,white1994estimation,KOMUNJER2005137}.
\subsection{Outline of Paper}
The rest of the paper is organized as follows. Section \ref{sec:DGP} formalizes the data-generating process and puts forward the Bayesian inference framework. Section \ref{sec:BVM} derives conditions under which the marginal posterior for $\beta$ satisfies a Bernstein-von Mises theorem and verifies the conditions for two important classes of priors. Section \ref{sec:discuss} provides an extensive discussion the findings of Section \ref{sec:BVM}. Section \ref{sec:conclusion} concludes and proposes some directions for future research. All proofs are in the Appendix and notation is introduced when appropriate.
\section{Data Generating Process and Bayesian Inference}\label{sec:DGP}
\subsection{Data Generating Process}
Let $Y \in \mathcal{Y} \subseteq \mathbb{R}$ be an outcome variable, let $X \in \mathcal{X} \subseteq \mathbb{R}$ be a covariate of interest, let $W \in \mathcal{W} \subseteq \mathbb{R}^{d_{w}}$, $d_{w} < \infty$, be some control variables, and let $P_{0}$ be the joint distribution of $(Y,X,W')'$. The observed data $\{(Y_{i},X_{i},W_{i}')'\}_{i=1}^{n}$ consists of the first $n$ elements of a sequence $\{(Y_{i},X_{i},W_{i}')'\}_{i \geq 1}$ of independent and identically distributed (i.i.d) random vectors drawn from $P_{0}$. Assumption \ref{as:dgp} states the restrictions on the data distribution $P_{0}$.
\begin{assumption}\label{as:dgp} 
The distribution $P_{0}$ of $(Y,X,W')'$ satisfies the following restrictions:
\begin{enumerate}
\item The distribution of $Y$ given $X$ and $W$ satisfies $E_{P_{0}}[Y|X,W]= X\beta_{0}+\eta_{0}(W)$ and $Var_{P_{0}}[Y|X,W] = \sigma_{01}^{2}$ for some $(\beta_{0},\eta_{0}, \sigma_{01}^{2}) \in \mathcal{B} \times \mathcal{H} \times \mathbb{R}_{++}$ with $\mathcal{B} \subseteq \mathbb{R}$ and $\mathcal{H} \subseteq L^{2}(\mathcal{W})$, where $f \in L^{2}(\mathcal{W})$ iff $E_{P_{0}}[f^{2}(W)] < \infty$.
\item There exists constants $0<\underline{c}< \overline{c}<\infty$ such that $E_{P_{0}}[(X-E_{P_{0}}[X|W])^{2}]\in [\underline{c},\overline{c}]$.
\item The conditional distribution $P_{0,YX|W}$ of $(Y,X)'$ given $W$ has a density $p_{0,YX|W}$ with respect to the Lebesgue measure and $E_{P_{0}}|\log p_{0,YX|W}(Y,X|W)| \in (0,\infty)$
\item The projection errors $Y- E_{P_{0}}[Y|W]$ and $X-E_{P_{0}}[X|W]$ satisfy $E_{P_{0}}[(Y-E_{P_{0}}[Y|W])^{2}] \in (0,\infty)$ and $E_{P_{0}}[(X-E_{P_{0}}[X|W])^{4}]\in (0,\infty)$.
\end{enumerate} 
\end{assumption}
I briefly discuss Assumption \ref{as:dgp}. Part 1 of Assumption \ref{as:dgp} imposes that the conditional moment restriction (\ref{eq:original}) is satisfied at $P_{0}$ and that the projection error $Y - X\beta_{0}-\eta_{0}(W)$ is homoskedastic. I treat $\sigma_{01}^{2}$ as known for exposition, however, Appendix \ref{ap:multivariate} extends the main results to accommodate unknown $\sigma_{01}^{2}$ and multivariate $X$. Part 2 of Assumption \ref{as:dgp} is a strong identification condition that is necessary and sufficient for regular estimation of $\beta_{0}$ \citep{robinson1988root}. The remaining parts are technical conditions that enable the application of stochastic limit theorems and ensure certain population objective functions are well-defined.
\subsection{Bayesian Inference}
Bayesian inference starts with a conditional model for the data (\textit{sampling model}) and a distribution over the model parameters (\textit{prior}). I describe each of these in turn. The sampling model is
\begin{align}
    &Y_{i}|\{(X_{i},W_{i}')'\}_{i=1}^{n},\beta,m \overset{ind}{\sim} \mathcal{N}(m_{1}(W_{i})+ \beta(X_{i}-m_{2}(W_{i})),\sigma_{01}^{2}) \label{eq:samplingmodelY} \\
    &X_{i}|\{W_{i}\}_{i=1}^{n},\beta,m\overset{ind}{\sim} \mathcal{N}(m_{2}(W_{i}),\sigma_{02}^{2})\label{eq:samplingmodelX}
\end{align}
for $i=1,...,n$, where $\beta \in \mathcal{B}$, $m = (m_{1},m_{2}) \in \mathcal{M}$ with $\mathcal{M} = \mathcal{M}_{1} \times \mathcal{M}_{2}$ and $\mathcal{M}_{j} \subseteq L^{2}(\mathcal{W})$ for $j \in \{1,2\}$, and $\sigma_{01}^{2},\sigma_{02}^{2} \in (0,\infty)$ are known scalars. Display (\ref{eq:samplingmodelY}) is a model for the conditional distribution of $Y$ given $X$ and $W$ that imposes the \cite{robinson1988root} transformation of the partially linear model, a result that establishes the observational equivalence between (\ref{eq:original}) and the ‘partialed out' regression model $Y = m_{01}(W) + (X-m_{02}(W))\beta_{0} + U$, where $m_{01}(W) = E_{P_{0}}[Y|W]$, $m_{02}(W) = E_{P_{0}}[X|W]$, and $U$ satisfies $E[U|X,W] = 0$ and $Var[U|X,W] = \sigma_{01}^{2}$. Display (\ref{eq:samplingmodelX}) is a model for the conditional distribution of $X$ given $W$ that is included to account for the fact that $m_{01}$ and $m_{02}$ are not separately identified from (\ref{eq:samplingmodelY}).\footnote{Since (\ref{eq:samplingmodelX}) is for identification of $m_{02}$ only, $\sigma_{02}^{2}$ has no intrinsic meaning and can be set at an arbitrary fixed value (e.g., $\sigma_{02}^{2}=1$).} Together, (\ref{eq:samplingmodelY}) and (\ref{eq:samplingmodelX}) imply a (quasi-)likelihood $L_{n}(\beta,m)$ given by
\begin{align*}
L_{n}(\beta,m) = \prod_{i=1}^{n}p_{\beta,m}(Y_{i},X_{i}|W_{i}),    
\end{align*} 
where $p_{\beta,m}(\cdot|w)$ is the probability density function of a bivariate Gaussian distribution with mean vector $m(w)=(m_{1}(w),m_{2}(w))'$ and covariance matrix
\begin{align*}
    V(\beta)  = \begin{pmatrix} \sigma_{01}^{2}+\sigma_{02}^{2}\beta^{2} & \beta \sigma_{02}^{2} \\
    \beta \sigma_{02}^{2} & \sigma_{02}^{2}
    \end{pmatrix}.
\end{align*}
The term `(quasi-)likelihood' is used because the set of conditional distributions for $(Y,X)$ given $W$ compatible with Assumption \ref{as:dgp} is much larger than those that are compatible with (\ref{eq:samplingmodelY})--(\ref{eq:samplingmodelX}). Sections \ref{sec:BVM} and \ref{sec:discuss} elaborate on the sampling model's suitability for estimating $\beta_{0}$ despite it being much more restrictive than Assumption \ref{as:dgp}.

Since the model parameters are $(\beta,m)$, the prior $\Pi$ is a probability distribution over $\mathcal{B} \times \mathcal{M}$. Throughout, I maintain that $\beta \sim \Pi_{\mathcal{B}}$, $m \sim \Pi_{\mathcal{M}}$, and $\beta \independent m$, where $\Pi_{\mathcal{B}}$ and $\Pi_{\mathcal{M}}$ are probability measures over $\mathcal{B}$ and $\mathcal{M}$, respectively. This means that $\Pi = \Pi_{\mathcal{B}} \otimes \Pi_{\mathcal{M}}$. Prior independence between the low-dimensional target parameter and the infinite-dimensional nuisance parameters is standard in semiparametric Bayesian inference \citep{bickel2012semiparametric,castillo2012semiparametric,castillo2012semiparametricBVM,ghosal2017fundamentals}. Assumption \ref{as:prior} states some regularity conditions for $\Pi$.
\begin{assumption}\label{as:prior}
The following conditions hold:
\begin{enumerate}
\item The prior is such that $P_{0}^{\infty}(\int_{\mathcal{B}\times \mathcal{M}}L_{n}(\beta,m)d \Pi(\beta,m) \in (0,\infty))=1$, where $P_{0}^{\infty}$ is the law of $\{(Y_{i},X_{i},W_{i}')'\}_{i\geq 1}$.
\item The prior $\Pi_{\mathcal{B}}$ has a Lebesgue density $\pi_{\mathcal{B}}$ that is continuous and positive over a neighborhood $\mathcal{B}_{0}$ of $\beta_{0}$.
\end{enumerate}
\end{assumption}
The sampling model (\ref{eq:samplingmodelY})--(\ref{eq:samplingmodelX}) and prior $\Pi$ implies a conditional distribution $\Pi((\beta,m) \in \cdot | \{(Y_{i},X_{i},W_{i}')'\}_{i=1}^{n})$ for the model parameters $(\beta,m)$ given the data $\{(Y_{i},X_{i},W_{i}')'\}_{i=1}^{n}$. This is known as the \textit{posterior}, and, under Assumption \ref{as:prior}, posterior probabilities are computed via Bayes rule. That is, for any event $A$,
\begin{align*}
\Pi((\beta,m) \in A|\{(Y_{i},X_{i},W_{i}')'\}_{i=1}^{n}) = \frac{\int_{A} L_{n}(\beta,m) d \Pi(\beta,m)}{\int_{\mathcal{B}\times \mathcal{M}}L_{n}(\beta,m)d\Pi(\beta,m)}.   
\end{align*}
Bayesian inference revolves around the posterior distribution. As examples, researchers interested in a point estimator for $\beta$ can report the posterior median or posterior mean, while researchers interested in quantifying uncertainty about $\beta$ report a \textit{ credible set}, a set estimator $CS_{n}(1-\alpha)$ that satisfies $\Pi(\beta \in CS_{n}(1-\alpha)|\{(Y_{i},X_{i},W_{i}')'\}_{i=1}^{n}) \geq 1-\alpha$. 
\begin{remark}[Posterior Computation] Although this paper focuses on theoretical properties of the posterior, I provide some brief remarks on posterior computation. An approach to sampling from the posterior is Gibbs sampling. That is, the elements of $(\beta,m_{1},m_{2})$ are updated component-wise while fixing the remaining parameters. Holding $\beta$ and $m_{2}$ fixed, (\ref{eq:samplingmodelY}) reveals that Bayesian updating of $m_{1}$ is that of a Gaussian nonparametric regresion of $Y-(X-m_{2}(W))\beta$ on $W$. Similarly, factorizing $p_{\beta,m}$ into the product of the conditional distribution of $X|Y,W$ and $Y|W$ indicates that Bayesian updating of $m_{2}$ (while holding $\beta$ and $m_{1}$ fixed) is that of a Gaussian nonparametric regression of $X-(\beta\sigma_{2}^{2}/(\sigma_{01}^{2}+\beta^{2}\sigma_{02}^{2}))(Y-m_{1}(W))$ on $W$. Finally, holding $m_{1}$ and $m_{2}$ fixed, the first part of (\ref{eq:samplingmodelY}) reveals that Bayesian updating of $\beta$ conforms to that of a Gaussian linear regression of $Y-m_{1}(W)$ on $X-m_{2}(W)$. Section \ref{sec:matern} provides some theoretical guarantees for the case where $m_{1} \independent m_{2}$ and $m_{j}$, $j \in \{1,2\}$, follows a Gaussian process. This is relevant for computation because these priors are conjugate for the first two steps.
\end{remark}
\section{Bernstein-von Mises Theorem}\label{sec:BVM}
This section provides Bernstein-von Mises theory for the marginal posterior for $\beta$. It has two subsections. The first establishes a general Bernstein-von Mises theorem based on high-level posterior consistency and empirical process conditions, and the second verifies the assumptions for two classes of priors.
\subsection{General Theorem}\label{sec:generalbvm}
I prove a general Bernstein-von Mises theorem for the marginal posterior of $\beta$. The asymptotic analysis is conducted conditional on the realizations of $\{W_{i}\}_{i \geq 1}$, meaning that $\{(Y_{i},X_{i})\}_{i=1}^{n}$ should be viewed as a sequence of independent but not identically distributed random vectors (i.e., $(Y_{i},X_{i}) \independent (Y_{j},X_{j})|\{W_{i}\}_{i \geq 1}$ for $i \neq j$). For notation, $P_{0,W}^{\infty}$ is the joint law of the i.i.d sequence $\{W_{i}\}_{i \geq 1}$, $\langle f, g \rangle_{n,2} = n^{-1}\sum_{i=1}^{n}f(w_{i})g(w_{i})$ and $||f||_{n,2}= \sqrt{\langle f,f \rangle_{n,2}}$ are the empirical $L^{2}$ inner product and induced norm, respectively, for functions $f,g: \mathcal{W} \rightarrow \mathbb{R}$ and  for a given realization $\{w_{i}\}_{i \geq 1}$ of $\{W_{i}\}_{i \geq 1}$, $B_{n,\mathcal{M}}(m_{0},\delta)$ is a ball centered at $m_{0}$ with radius $\delta > 0$ in the product empirical $L^{2}$ norm $\max\{||m_{1}||_{n,2},||m_{2}||_{n,2}\}$, and $P_{0,YX|W}^{(n)} = \bigotimes_{i=1}^{n}P_{0,YX|w_{i}}$ with $P_{0,YX|w_{i}}$ denoting the probability distribution associated with the probability density function $p_{0,YX|W}(\cdot|w_{i})$ for $i=1,...,n$. The next two assumptions concern the limiting behavior of the posterior.

\begin{assumption}\label{as:consistency}
There exists a sequence $\{\delta_{n}\}_{n \geq 1}$ such that 
\begin{align*}
\delta_{n} = o(n^{-1/4})    
\end{align*}
and 
\begin{align*}
\Pi(B_{n,\mathcal{M}}(m_{0},D\delta_{n})|\{(Y_{i},X_{i},w_{i}')'\}_{i=1}^{n}) \overset{P_{0,YX|W}^{(n)}}{\longrightarrow} 1    
\end{align*}
for $P_{0,W}^{\infty}$-almost every fixed realization $\{w_{i}\}_{i \geq 1}$ of $\{W_{i}\}_{i \geq 1}$ as $n\rightarrow \infty$ for some constant $D>0$ large.
\end{assumption}
\begin{assumption}\label{as:emp_process}
There are sets $\{\mathcal{M}_{n}\}_{n \geq 1} \subseteq \mathcal{M}$ such that 
\begin{align*}
\Pi(\mathcal{M}_{n}|\{(Y_{i},X_{i},w_{i}')'\}_{i=1}^{n}) &\overset{P_{0,YX|W}^{(n)}}{\longrightarrow} 1 \\    
\sup_{m \in \mathcal{M}_{n}\cap B_{n,\mathcal{M}}(m_{0},D\delta_{n})}\left| G_{n}^{(1)}(m) \right|&\overset{P_{0,YX|W}^{(n)}}{\longrightarrow} 0  \\  
\sup_{m \in \mathcal{M}_{n}\cap B_{n,\mathcal{M}}(m_{0},D\delta_{n})}\left| G_{n}^{(2)}(m) \right|&\overset{P_{0,YX|W}^{(n)}}{\longrightarrow} 0    
\end{align*}
for $P_{0,W}^{\infty}$-almost every fixed sequence $\{w_{i}\}_{i \geq 1}$ of $\{W_{i}\}_{i \geq 1}$ as $n\rightarrow \infty$, where $G_{n}^{(1)},G_{n}^{(2)}: \mathcal{M} \rightarrow \mathbb{R}$ satistify
\begin{align*}
    G_{n}^{(1)}(m) = \frac{1}{\sqrt{n}}\sum_{i=1}^{n}\varepsilon_{i2}(m_{1}(w_{i})-m_{01}(w_{i}))
\end{align*}
and
\begin{align*}
G_{n}^{(2)}(m) = \frac{1}{\sqrt{n}}\sum_{i=1}^{n}(U_{i}-\beta_{0}\varepsilon_{i2})(m_{2}(w_{i})-m_{02}(w_{i})),    
\end{align*}
with $\varepsilon_{i2} = X_{i}-m_{02}(w_{i})$ and $U_{i} = Y_{i}-m_{01}(w_{i})-(X_{i}-m_{02}(w_{i}))\beta_{0}$.
\end{assumption}
Assumptions \ref{as:consistency} and \ref{as:emp_process} are high-level conditions about the the limiting behavior of the posterior. Assumption \ref{as:consistency} enforces that the marginal posterior for $m$ concentrates around $m_{0}$ in the empirical $L^{2}$-norm and the rate of convergence $\delta_{n}$ is faster than $n^{-1/4}$. Assumption \ref{as:emp_process} imposes that there are sets $\{\mathcal{M}_{n}\}_{n \geq 1}$ that the posterior concentrates on such that, when intersected with $B_{n,\mathcal{M}}(m_{0},\delta_{n})$, are structured enough to ensure that multiplier empirical processes $\{G_{n}^{(1)}\}_{n\geq 1}$ and $\{G_{n}^{(2)}\}_{n \geq 1}$ satisfy stochastic equicontinuity-type conditions. These requirements are conceptually similar to those encountered in the literature on asymptotic normality of plug-in method of moments estimators based on orthogonal moment conditions \citep{andrews1994asymptotics,chernozhukov2018double} and semiparametric maximum likelihood estimators \citep{murphy2000profile}. Combined with Assumptions \ref{as:dgp} and \ref{as:prior}, Assumptions \ref{as:consistency} and \ref{as:emp_process} imply that the marginal posterior for $\beta$ satisfies a Bernstein-von Mises theorem.
\begin{theorem}\label{thm:BVM}
Let $\tilde{I}_{n}(m_{0})$ be given by 
\begin{align*}
\tilde{I}_{n}(m_{0}) = \frac{1}{n}\sum_{i=1}^{n}\left(\frac{X_{i}-m_{02}(w_{i})}{\sigma_{01}}\right)^{2}   
\end{align*} 
and let $\tilde{\Delta}_{n,0}$ be given by 
\begin{align*}
\tilde{\Delta}_{n,0} = \tilde{I}_{n}(m_{0})^{-1}\frac{1}{\sqrt{n}}\tilde{\ell}_{n}(\beta_{0},m_{0}),
\end{align*}
where
\begin{align*}
\tilde{\ell}_{n}(\beta_{0},m_{0}) = \frac{1}{\sigma_{01}^{2}}\sum_{i=1}^{n}(Y_{i}-m_{01}(w_{i})-(X_{i}-m_{02}(w_{i}))\beta_{0})(X_{i}-m_{02}(w_{i})).
\end{align*}
If Assumptions \ref{as:dgp}, \ref{as:prior}, \ref{as:consistency}, and \ref{as:emp_process} hold and $\beta_{0} \in \text{int}(\mathcal{B})$, then
\begin{align*}
  \left | \left |\Pi\left(\beta \in \cdot |\{(Y_{i},X_{i},w_{i}')'\}_{i=1}^{n}\right) - \mathcal{N}\left(\beta_{0}+\frac{\tilde{\Delta}_{n,0}}{\sqrt{n}},\frac{1}{n}\tilde{I}_{n}(m_{0})^{-1} \right) \right | \right |_{TV} \overset{P_{0,YX|W}^{(n)}}{\longrightarrow} 0
\end{align*}
for $P_{0,W}^{\infty}$-almost every fixed realization $\{w_{i}\}_{i \geq 1}$ of $\{W_{i}\}_{i \geq 1}$ as $n\rightarrow \infty$, where $||\cdot||_{TV}$ is total variation distance.\footnote{The total variation distance between probability measures $P$ and $Q$ is $||P-Q||_{TV} = \sup_{A}|P(A)-Q(A)|$.}
\end{theorem}
\begin{remark}[Unconditional Asymptotics]
Theorem \ref{thm:BVM} conditions on the sequence of control variables $\{W_{i}\}_{i \geq 1}$. The same conclusion holds in the \textit{unconditional} asymptotic thought experiment in which $P_{0,YX|W}^{(n)}$ is replaced by the $n$-fold product measure $P_{0}^{n}$ (i.e., the sequence $\{W_{i}\}_{i \geq 1}$ is also treated as random). This follows from an application of the Bounded Convergence Theorem and is stated as a corollary below.
\end{remark}
\begin{corollary}\label{cor:BVM_unconditional}
If Assumptions \ref{as:dgp}, \ref{as:prior}, \ref{as:consistency}, and \ref{as:emp_process} hold and $\beta_{0} \in \text{int}(\mathcal{B})$, then
\begin{align*}
  \left | \left |\Pi\left(\beta \in \cdot |\{(Y_{i},X_{i},W_{i}')'\}_{i=1}^{n}\right) - \mathcal{N}\left(\beta_{0}+\frac{\tilde{\Delta}_{n,0}}{\sqrt{n}},\frac{1}{n}\tilde{I}_{n}(m_{0})^{-1} \right) \right | \right |_{TV} \overset{P_{0}^{n}}{\longrightarrow} 0
\end{align*}
as $n\rightarrow \infty$.
\end{corollary}
\begin{remark}[Efficiency]
Under i.i.d sampling, the probability limit of $\tilde{I}_{n}^{-1}(m_{0})$ is the \cite{chamberlain1992efficiency} semiparametric efficiency bound under homoskedasticity and $\hat{\beta}_{n} =\beta_{0}+\tilde{\Delta}_{n,0}/\sqrt{n}$ is an estimator that achieves Chamberlain's bound. Consequently, the above results provide conditions under which the marginal posterior for $\beta$ is first-order asymptotically equivalent to a (locally) efficient frequentist estimator of the regression coefficient $\beta_{0}$. Efficiency is local because \cite{chamberlain1992efficiency} does not constrain the conditional variance of $U$ given $X$ and $W$ to be constant when deriving the information bound (see Section 4.3 of \cite{newey1990semiparametric} for more on local efficiency).
\end{remark}
\begin{remark}[Uncertainty Quantification]
Theorem \ref{thm:BVM} implies Bayesian credible sets are asymptotically (locally) efficient frequentist confidence sets for data-generating processes compatible with Assumption \ref{as:dgp}. To see why, let $c_{n}(q)$, $q \in (0,1)$, be the $q$-quantile of the marginal posterior $\Pi(\beta \in \cdot |\{(Y_{i},X_{i},W_{i}')'\}_{i=1}^{n})$ and consider the equitailed probability interval $[c_{n}(\alpha/2),c_{n}(1-\alpha/2)]$, where $\alpha \in (0,1/2)$. Corollary \ref{cor:quantile} establishes that the endpoints of these intervals are asymptotically equivalent to the endpoints of Wald confidence intervals based on the estimator $\hat{\beta}_{n}=\beta_{0}+\tilde{\Delta}_{n,0}/\sqrt{n}$ that achieves the \cite{chamberlain1992efficiency} efficiency bound under homoskedasticity. Consequently, the credible interval $[c_{n}(\alpha/2),c_{n}(1-\alpha/2)]$ achieves asymptotic frequentist coverage of $1-\alpha$. Note that the result is stated conditionally (i.e., $P_{0,YX|W}^{(n)}$), however, the same argument can be applied using Corollary \ref{cor:BVM_unconditional} to obtain it unconditionally (i.e., $P_{0}^{n}$).
\begin{corollary}\label{cor:quantile}
If Assumptions \ref{as:dgp}, \ref{as:prior}, \ref{as:consistency}, and \ref{as:emp_process} hold and $\beta_{0} \in \text{int}(\mathcal{B})$, then the $q$-quantile $c_{n}(q)$, $q \in (0,1)$, of the marginal posterior $\Pi(\beta \in \cdot| \{(Y_{i},X_{i},w_{i}')'\}_{i=1}^{n})$ satisfies
\begin{align*}
        c_{n}(q) = \hat{\beta}_{n} + \Phi^{-1}(q)\sqrt{\frac{\tilde{I}_{n}^{-1}(m_{0})}{n}} + o_{P_{0,YX|W}^{(n)}}\left(\frac{1}{\sqrt{n}}\right)
\end{align*}
for $P_{0,W}^{\infty}$-almost every fixed realization $\{w_{i}\}_{i \geq 1}$ of $\{W_{i}\}_{i \geq 1}$ as $n\rightarrow \infty$, where $\Phi$ is the cumulative distribution function of $\mathcal{N}(0,1)$, $\hat{\beta}_{n} = \beta_{0}+\tilde{\Delta}_{n,0}/\sqrt{n}$, and $\tilde{I}_{n}(m_{0})=\sigma_{01}^{-2}n^{-1}\sum_{i=1}^{n}(X_{i}-m_{02}(w_{i}))^{2}$.
\end{corollary}
\end{remark}
\subsection{Sufficient Conditions and Examples}\label{sec:examples}
This section has two subsections. Section \ref{sec:lowlevelsuff} presents sufficient conditions for the assumptions in Section \ref{sec:BVM}. Sections \ref{sec:waveletex} and \ref{sec:matern} verify Theorem \ref{thm:BVM} for two classes of priors using these sufficient conditions.
\subsubsection{Useful Preliminary Results}\label{sec:lowlevelsuff}
I start with a proposition that enables finding sequences $\{\delta_{n}\}_{ n\geq 1}$ for which Assumption \ref{as:consistency} holds. It has similarities with Theorem 2 of \cite{shen2001rates} in that it establishes posterior consistency by directly analyzing the (quasi-)likelihood ratio process. For notation, let $\varepsilon = (Y-m_{01}(W),X-m_{02}(W))'$, and let $\lambda_{min}(A)$ and $\lambda_{max}(A)$ denote the minimum and maximum eigenvalues of a matrix $A$, respectively.
\begin{proposition}\label{prop:suffgeneral}
Suppose that $\mathcal{B}$ is compact, $\sup_{w \in \mathcal{W}}\lambda_{max}(E(\varepsilon\varepsilon'|W=w))< \infty$, and there is a sequence $\{\delta_{n}\}_{n \geq 1}$ such that the following conditions hold:
\begin{enumerate}
    \item There exists a constant $C>0$ such that $\Pi_{\mathcal{M}}(m \in B_{n,\mathcal{M}}(m_{0},\delta_{n})) \gtrsim \exp(-Cn\delta_{n}^{2})$ conditionally given $P_{0,W}^{\infty}$-almost every realization $\{w_{i}\}_{i \geq 1}$ of $\{W_{i}\}_{i \geq 1}$.
    \item There exists increasing function $\delta \mapsto \omega_{n}(\delta)$ for which $\omega_{n}(\delta_{n}) \leq \sqrt{n}\delta_{n}^{2}$, $\delta \mapsto \omega_{n}(\delta)/\delta^{\upsilon}$ is decreasing for some $\upsilon \in (0, 2)$, and 
\begin{align}\label{eq:modcontinuity}
E_{P_{0,YX|W}^{(n)}}\left[\sup_{(\beta,m) \in \mathcal{B} \times B_{n,\mathcal{M}}(m_{0},\delta)}\left|\frac{1}{\sqrt{n}}\sum_{i=1}^{n}(m(w_{i})-m_{0}(w_{i}))'V^{-1}(\beta)\varepsilon_{i}\right|\right] \lesssim \omega_{n}(\delta)   
\end{align}
for $P_{0,W}^{\infty}$-almost every fixed realization $\{w_{i}\}_{i \geq 1}$ of $\{W_{i}\}_{i \geq 1}$.
\item The sequence $\{\delta_{n}\}_{n \geq 1}$ satisfies $n\delta_{n}^{2} \rightarrow \infty$ as $n\rightarrow \infty$.
\end{enumerate}Then the marginal posterior for $m$ satisfies
\begin{align*}
\Pi\left(m \in B_{n,\mathcal{M}}(m_{0},D\delta_{n})^{c}\middle |\{(Y_{i},X_{i},w_{i}')'\}_{i =1}^{n}\right) \overset{P_{0,YX|W}^{(n)}}{\longrightarrow} 0     
\end{align*} 
for $P_{0,W}^{\infty}$-almost every fixed realization $\{w_{i}\}_{i \geq 1}$ of $\{W_{i}\}_{i \geq 1}$ as $n\rightarrow \infty$ for some large constant $D>0$.
\end{proposition}
Proposition \ref{prop:suffgeneral} states that $\delta_{n}$ is determined by the 1. mass the nuisance prior $\Pi_{\mathcal{M}}$ assigns to empirical $L^{2}$ neighborhoods of $m_{0}$, and 2. the rate of convergence of the least squares estimator of $m_{0}$. Since $||\cdot||_{n,2} \leq ||\cdot||_{\infty}$ with $||f||_{\infty}= \sup_{w \in \mathcal{W}}|f(w)|$ denoting the supremum norm, the prior mass condition can be checked by analyzing the probability mass that $\Pi_{\mathcal{M}}$ assigns to uniform balls centered at $m_{0}$, a quantity that has been studied for many different priors \citep{ghosal2017fundamentals}. The least squares connection is through the second condition because $\delta \mapsto \omega_{n}(\delta)$ can be identified as the `continuity modulus' of the multiplier empirical process that determines the least squares rate of convergence (see, for example, Chapter 9 of \cite{geer2000empirical} and Section 3.2 of \cite{vaartwellner96book}). Given a valid $\delta_{n}$, conditions that ensure the rate restriction $\delta_{n} = o(n^{-1/4})$ can be established.
\begin{remark}[Relaxing Part 2 of Proposition \ref{prop:suffgeneral}]\label{remark:sieved}
The supremum in (\ref{eq:modcontinuity}) may be infinite for some choices of $\mathcal{M}$ (e.g., if $\mathcal{M}$ is not compact).\footnote{I do not relax compactness of $\mathcal{B}$ because, unlike infinite-dimensional parameter spaces, a compactness condition for a finite-dimensional parameter space is not overly restrictive.} The following result modifies Proposition \ref{prop:suffgeneral} to hold along sieves $\{\mathcal{M}_{n}\}_{n \geq 1}$ such that the posterior $\Pi(m \in \mathcal{M}_{n}|\{(Y_{i},X_{i},w_{i}')'\}_{i=1}^{n}) \rightarrow 1$ in $P_{0,YX|W}^{(n)}$-probability for $P_{0,W}^{\infty}$-almost every fixed realization $\{w_{i}\}_{i \geq 1}$ of $\{W_{i}\}_{i \geq 1}$ as $n\rightarrow \infty$. This relaxation is relevant for the Gaussian process example in Section \ref{sec:matern}, and, like Proposition \ref{prop:suffgeneral}, has some similarities with \cite{shen2001rates} (i.e., Theorem 4 of their paper).
\end{remark}
\begin{proposition}\label{prop:suffgeneralsieve}
Suppose that $\mathcal{B}$ is compact, $\sup_{w \in \mathcal{W}}\lambda_{max}(E(\varepsilon\varepsilon'|W=w))< \infty$, and there is a sequences $\{\delta_{n}\}_{n \geq 1}$ and $\{\mathcal{M}_{n}\}_{n \geq 1}$ such that the following conditions hold:
\begin{enumerate}
    \item There exists a constant $C>0$ such that $\Pi_{\mathcal{M}}(m \in B_{n,\mathcal{M}}(m_{0},\delta_{n})) \gtrsim \exp(-Cn\delta_{n}^{2})$ conditionally given $P_{0,W}^{\infty}$-almost every realization $\{w_{i}\}_{i \geq 1}$ of $\{W_{i}\}_{i \geq 1}$.
    \item The posterior satisfies $\Pi(\mathcal{M}_{n}^{c}|\{(Y_{i},X_{i},w_{i}')'\}_{i=1}^{n}) \rightarrow 0$ in $P_{0,YX|W}^{(n)}$-probability for $P_{0,W}^{\infty}$-almost every fixed realization $\{w_{i}\}_{i \geq 1}$ of $\{W_{i}\}_{i \geq 1}$ as $n\rightarrow \infty$.
    \item There exists an increasing function $\delta \mapsto \omega_{n}(\delta)$ for which $\omega_{n}(\delta_{n}) \leq \sqrt{n}\delta_{n}^{2}$, $\delta \mapsto \omega_{n}(\delta)/\delta^{\upsilon}$ is decreasing for some $\upsilon \in (0, 2)$, and 
\begin{align*}
E_{P_{0,YX|W}^{(n)}}\left[\sup_{(\beta,m) \in \mathcal{B} \times (\mathcal{M}_{n}\cap B_{n,\mathcal{M}}(m_{0},\delta))}\left|\frac{1}{\sqrt{n}}\sum_{i=1}^{n}(m(w_{i})-m_{0}(w_{i}))'V^{-1}(\beta)\varepsilon_{i}\right|\right] \lesssim \omega_{n}(\delta)   
\end{align*}
for $P_{0,W}^{\infty}$ almost-every fixed realization $\{w_{i}\}_{i \geq 1}$ of $\{W_{i}\}_{i \geq 1}$.
\item The sequence satisfies $n\delta_{n}^{2} \rightarrow \infty$ as $n\rightarrow \infty$.
\end{enumerate}
Then the marginal posterior for $m$ satisfies
\begin{align*}
\Pi\left(m \in B_{n,\mathcal{M}}(m_{0},D\delta_{n})^{c}\middle |\{(Y_{i},X_{i},w_{i}')'\}_{i =1}^{n}\right) \overset{P_{0,YX|W}^{(n)}}{\longrightarrow} 0     
\end{align*} 
for $P_{0,W}^{\infty}$-almost every fixed realization $\{w_{i}\}_{i \geq 1}$ of $\{W_{i}\}_{i \geq 1}$ as $n\rightarrow \infty$ for some large constant $D>0$.
\end{proposition}
\begin{remark}[Propositions \ref{prop:suffgeneral} and \ref{prop:suffgeneralsieve} Can Help Verify Assumption \ref{as:emp_process}]\label{remark:checkingA4}
Verifying Proposition \ref{prop:suffgeneral} can also lead to the verification of Assumption \ref{as:emp_process}. Since $\beta \mapsto V^{-1}(\beta)$ is continuous and $\mathcal{B}$ is compact, the verification of (\ref{eq:modcontinuity}) can be reduced to a collection of univariate problems of finding increasing functions $\delta \mapsto \omega_{n,j_{1},j_{2}}(\delta)$, $j_{1},j_{2} \in \{1,2\}$, such that $\delta \mapsto \omega_{n,j_{1},j_{2}}(\delta)/\delta^{\upsilon_{j_{1},j_{2}}}$ is decreasing for some $\upsilon_{j_{1},j_{2}} \in (0,2)$ and
\begin{align}\label{eq:componentwise}
E_{P_{0,YX|W}^{(n)}}\left[\sup_{m_{j_{1}} \in \mathcal{M}_{j_{1}}: ||m_{j_{1}}-m_{0,j_{1}}||_{n,2} < D \delta_{n}}\left|\frac{1}{\sqrt{n}}\sum_{i=1}^{n}(m_{j_{1}}(w_{i})-m_{0,j_{1}}(w_{i}))'\varepsilon_{i,j_{2}}\right|\right]  \lesssim \omega_{n,j_{1},j_{2}}(\delta).   
\end{align}
Indeed, given $\{\omega_{n,j_{1},j_{2}}: j_{1},j_{2} \in \{1,2\}\}$, I can check the inequality $\max_{j_{1},j_{2} \in \{1,2\}}\omega_{n,j_{1},j_{2}}(\delta_{n}) \leq \sqrt{n}\delta_{n}^{2}$ to verify that (\ref{eq:modcontinuity}) holds (ignoring multiplicative constants). Since it is required that $\delta_{n} = o(n^{-1/4})$ and the error $U$ satisfies $U = \varepsilon_{1} - \beta_{0}\varepsilon_{2}$, the inequality $$\max_{j_{1},j_{2} \in \{1,2\}}\omega_{n,j_{1},j_{2}}(\delta_{n}) \leq \sqrt{n}\delta_{n}^{2}$$ immediately implies Assumption \ref{as:emp_process} holds for $\mathcal{M}_{n} = \mathcal{M}$. For the case where the suprema in (\ref{eq:componentwise}) may be unbounded, a symmetric argument may be applied except that (\ref{eq:componentwise}) is modified to incorporate $\mathcal{M}_{n,j_{1}}$, $j_{1} \in \{1,2\}$. That is,
\begin{align*}
    E_{P_{0,YX|W}^{(n)}}\left[\sup_{m_{j_{1}} \in \mathcal{M}_{n,j_{1}}: ||m_{j_{1}}-m_{0,j_{1}}||_{n,2} < D \delta_{n}}\left|\frac{1}{\sqrt{n}}\sum_{i=1}^{n}(m_{j_{1}}(w_{i})-m_{0,j_{1}}(w_{i}))'\varepsilon_{i,j_{2}}\right|\right]  \lesssim \omega_{n,j_{1},j_{2}}(\delta). 
\end{align*}
In such case, Assumption \ref{as:emp_process} holds for a candidate $\mathcal{M}_{n}$ used to verify Proposition \ref{prop:suffgeneralsieve}. I use the technique outlined in this remark for the examples in Sections \ref{sec:waveletex} and \ref{sec:matern}, so that Assumptions \ref{as:consistency} and \ref{as:emp_process} are simultaneously verified.
\end{remark}
\subsubsection{Example 1 -- Uniform Wavelet Series Priors}\label{sec:waveletex}
I verify Theorem \ref{thm:BVM} for a class of priors based on untruncated uniform wavelet series. These priors have feature regularly in the Bayesian nonparametrics literature \citep{gine2011rates,10.1214/13-AOS1133,castillo2014supremum,l2023semiparametric}, however, as far as I am aware, the Bernstein-von Mises theorem for the partially linear model has not been verified using these priors. To make things formal, suppose that $\mathcal{W}= [0,1]$, the marginal distribution $P_{0,W}$ of $W$ has a Lebesgue density $p_{0,W}$ on $[0,1]$ that is bounded away from zero and infinity, and that there are known constants $M>0$ and $\alpha_{0,1},\alpha_{0,2}>1/2$ such that $||m_{0,j}||_{\infty,\infty,\alpha_{0,j}} \leq M$ for all $j \in \{1,2\}$, where $||\cdot||_{\infty,\infty,\alpha}$ denotes the norm of a H\"{o}lder-Zygmund space $B_{\infty,\infty}^{\alpha}([0,1])$, $\alpha > 0$.\footnote{For $\alpha \notin \mathbb{N}$, H\"{o}lder-Zygmund spaces coincide with the usual H\"{o}lder space $C^{\alpha}([0,1])$ of functions that are $\lfloor \alpha \rfloor$-times continuously differentiable and $\lfloor \alpha \rfloor$th derivative is $(\alpha-\lfloor \alpha \rfloor)$-Lipschitz and the norm $||\cdot||_{\infty,\infty,\alpha}$ is equivalent to the usual H\"{o}lder norm $||\cdot||_{\alpha}$ (i.e., given by (4.111) in \cite{Gine_Nickl_2015}). However, for $\alpha \in \mathbb{N}$, these spaces are no longer the same. See Section 4.3.3 of \cite{Gine_Nickl_2015} for more details.} Given these assumptions, I set the nuisance parameter space to be $\mathcal{M}_{j}=\{f \in L^{2}([0,1]): ||f||_{\infty,\infty,\alpha_{0,j}} \leq M\}$ for each $j \in \{1,2\}$ and a prior $\Pi_{\mathcal{M}}$ over $\mathcal{M}_{1} \times \mathcal{M}_{2}$ is constructed as follows. I take a boundary adapted wavelet basis $\{\psi_{lk}: l \geq 0, 0 \leq k \leq 2^{l}-1\}$ for $L^{2}([0,1])$ that is sufficiently regular to characterize the H\"{o}lder-Zygmund spaces $(B_{\infty,\infty}^{\alpha}([0,1]),||\cdot||_{\infty,\infty,\alpha})$, $\alpha \in \{\alpha_{0,1},\alpha_{0,2}\}$, and set the marginal prior $\Pi_{\mathcal{M}_{j}}$ for $m_{j}$ to be the law of the random function
\begin{align*}
m_{j} = \sum_{l = 0}^{\infty}\sum_{k = 0}^{2^{l}-1}2^{-l(\alpha_{0,j}+1/2)}\nu_{lk,j}\psi_{lk},    
\end{align*} 
where $\nu_{lk,j}$ are independent and identically distributed (across $l$, $k$, and $j$) uniform random variables supported on $[-M,M]$.\footnote{The wavelet basis being sufficiently regular to characterize $(B^{\alpha}_{\infty,\infty}([0,1]),||\cdot||_{\alpha,\infty,\infty})$ means that the norm $||\cdot||_{\infty,\infty,\alpha}$ satisfies $||f||_{\infty,\infty,\alpha}=\sup_{l \geq 0}\max_{0 \leq k \leq 2^{l}-1}2^{l(\alpha+1/2)}|\langle f, \psi_{lk} \rangle_{2} |$. An example is the wavelet basis proposed in \cite{cohen1993wavelets}.} Since the wavelet coefficients are i.i.d across $j$, it follows that the joint law of $m=(m_{1},m_{2})$ satisfies $\Pi_{\mathcal{M}}= \Pi_{\mathcal{M}_{1}}\otimes \Pi_{\mathcal{M}_{2}}$. One can also check that the support of $\Pi_{\mathcal{M}}$ is $\mathcal{M}_{1}\times \mathcal{M}_{2}$. The next proposition verifies Theorem \ref{thm:BVM} for these priors under a subgaussian restriction on the projection errors $\varepsilon$.
\begin{proposition}\label{prop:uniform}
Suppose that Assumption \ref{as:dgp} holds, $\varepsilon$ is subgaussian conditional on $W$ and the conditional distribution is continuous in $W$, the marginal distribution $P_{0,W}$ of the control variable $W$ has a probability density function $p_{0,W}$ with respect to the Lebesgue measure on $[0,1]$ that is bounded away from zero and infinity, $m_{0} \in \mathcal{M}_{1} \times \mathcal{M}_{2}$ with $\mathcal{M}_{j} = \{f \in L^{2}([0,1]): ||f||_{\infty,\infty,\alpha_{0,j}} \leq M\}$ for each $j \in \{1,2\}$, where $\alpha_{0,1},\alpha_{0,2} > 1/2$ and $M < \infty$, and $\beta_{0} \in (-B,B)$ for some $B \in (0,\infty)$. Then Theorem \ref{thm:BVM} holds for the prior $\Pi = \Pi_{\mathcal{B}} \otimes \Pi_{\mathcal{M}}$, where $\Pi_{\mathcal{B}}$ is the uniform distribution over $\mathcal{B} = [-B,B]$ and $\Pi_{\mathcal{M}}=\Pi_{\mathcal{M}_{1}}\otimes \Pi_{\mathcal{M}_{2}}$ with $\Pi_{\mathcal{M}_{j}}$ being the law of the uniform wavelet series.
\end{proposition}
\begin{remark}[Model Specification and Role of Subgaussianity] 
Proposition \ref{prop:uniform} indicates that the uniform wavelet series priors are a class of priors for which the Bernstein-von Mises theorem holds even when the Gaussian model (\ref{eq:samplingmodelY})--(\ref{eq:samplingmodelX}) is misspecified (i.e., it only requires $\varepsilon$ be subgaussian conditional on $W$). The subgaussian restriction is only imposed to relate $\omega_{n}(\delta)$ to the metric entropy integrals $\int_{0}^{\delta} \sqrt{\log N(\tau,\mathcal{M}_{j},||\cdot||_{n,2})}d\tau$, $j \in \{1,2\}$, so that well-known metric entropy bounds for Besov balls (e.g., Theorem 4.3.36 of \cite{Gine_Nickl_2015}) can be applied to find sequences that satisfy Part 2 of Proposition \ref{prop:suffgeneral}.\footnote{The notation $N(\tau,\mathcal{F},||\cdot||)$ refers to the \textit{$\tau$-covering number} of the set $\mathcal{F}$ under the (semi)norm $||\cdot||$. That is, the smallest number of $||\cdot||$-balls with radius $\tau$ required to cover $\mathcal{F}$. The logarithm of the covering number is known as the \textit{metric entropy}.} For this reason, I consider subgaussianity to be a technical restriction and it may be that heavier tailed distributions for $\varepsilon$ can be accommodated using a more refined complexity measure for $\mathcal{M}_{j}$.
\end{remark}
\begin{remark}[No Cross Restrictions on Regularity of $m_{01}$ and $m_{02}$]\label{remark:waveletreg}
There are no restrictions on the relative magnitude of $\alpha_{0,1}$ and $\alpha_{0,2}$ because it is only required that $\alpha_{0,1} > 1/2$ and $\alpha_{0,2}>1/2$. In other words, the uniform wavelet series example demonstrates that the $(\beta,m)$-parametrization can allow for $m_{01}$ to be very smooth relative to $m_{02}$ (and vice versa), so long as this baseline smoothness condition is met (i.e., regularity greater than $1/2$)
\end{remark}
\subsubsection{Example 2 -- Mat\'{e}rn Gaussian Process Priors}\label{sec:matern}
This section verifies Theorem \ref{thm:BVM} for a popular class of Gaussian process priors. Suppose that $\mathcal{W} = [0,1]^{d_{w}}$ for some $1 \leq d_{w} < \infty$. The prior for $m=(m_{1},m_{2})$ is $\Pi_{\mathcal{M}} = \Pi_{\mathcal{M}_{1}} \otimes \Pi_{\mathcal{M}_{2}}$, where $\Pi_{\mathcal{M}_{j}}$ is the law of a centered Mat\'{e}rn Gaussian process on $[0,1]^{d_{w}}$ with regularity parameter $\alpha_{j} > 0$. Specifically, $m_{j} \sim \Pi_{\mathcal{M}_{j}}$ if and only if, for any set of indices $w_{1},...,w_{k} \in [0,1]^{d_{w}}$, the vector of function values $(m_{j}(w_{1}),...,m_{j}(w_{k}))'$ follows a centered multivariate normal distribution with covariance matrix elements given by 
\begin{align*}
\kappa_{\alpha_{j}}(w_{s},w_{t}) = \int_{\mathbb{R}^{d_{w}}}\exp(-i\lambda'(w_{s}-w_{t}))(1+||\lambda||_{2}^{2})^{-\alpha_{j}-\frac{d_{w}}{2}}d\lambda, \quad s,t \in \{1,...,k\}.  
\end{align*} 
The regularity parameter $\alpha_{j}$ determines the smoothness of the sample paths because it can be shown that $m_{j}$ takes values in the H\"{o}lder space $(C^{a}([0,1]^{d_{w}}),||\cdot||_{a})$ for any $a < \alpha_{j}$ (see Section 3.1 of \cite{van2011information}).\footnote{The H\"{o}lder space $C^{\alpha}([0,1])$ is the space of functions on $[0,1]^{d_{w}}$ that are $\lfloor \alpha \rfloor$-times continuously differentiable and $\lfloor \alpha \rfloor$th derivative is $(\alpha-\lfloor \alpha \rfloor)$-Lipschitz. The H\"{o}lder norm $||\cdot||_{\alpha}$ is given by (4.111) in \cite{Gine_Nickl_2015}, however, its specific form is not important for Proposition \ref{prop:materngp} (nor Proposition \ref{prop:materninformationloss} in Section \ref{sec:comparison} that establishes the analogous result for the $(\beta,\eta)$-parametrization).} The next proposition verifies Theorem \ref{thm:BVM} for the Mat\'{e}rn prior in the case where (\ref{eq:samplingmodelY})--(\ref{eq:samplingmodelX}) is correctly specified (see Remark \ref{remark:correcspec}). For notation, $H^{\alpha}([0,1]^{d_{w}})$ denotes the Sobolev space of functions $f$ on $[0,1]^{d_{w}}$ that can be extended to a function on $\mathbb{R}^{d}$ with Fourier transform $\hat{f}$ that satisfies $\int_{\mathbb{R}^{d_{w}}}|\hat{f}(\lambda)|^{2}(1+||\lambda||^{2}_{2})^{\alpha}d \lambda < \infty$.
\begin{proposition}\label{prop:materngp}
Suppose that $\Pi = \Pi_{\mathcal{B}} \otimes \Pi_{\mathcal{M}}$, where $\Pi_{\mathcal{B}}$ is the uniform distribution over $\mathcal{B} =[-B,B]$ and $\Pi_{\mathcal{M}} = \Pi_{\mathcal{M}_{1}} \otimes \Pi_{\mathcal{M}_{2}}$ with $\Pi_{\mathcal{M}_{j}}$ being the law of a Mat\'{e}rn process with regularity $\alpha_{j} > 0$. Further, suppose that the true conditional distribution of $(Y,X)'$ given $W$ satisfies $(Y,X)'|W \sim \mathcal{N}(m_{0}(W),V(\beta_{0}))$ for $\beta_{0} \in (-B,B)$ and $m_{0}= (m_{01},m_{02})$ with $m_{0,j} \in C^{\alpha_{0,j}}([0,1]^{d_{w}}) \cap H^{\alpha_{0,j}}([0,1]^{d_{w}})$ for each $j \in \{1,2\}$. Theorem \ref{thm:BVM} holds if the regularity parameters $\alpha_{1},\alpha_{0,1},\alpha_{2},\alpha_{0,2}>0$ satisfy the following inequalities:
\begin{align}\label{eq:betamineq}
\alpha_{1} > \frac{d_{w}}{2}, \quad \alpha_{0,1} > \frac{\alpha_{1}}{2}+ \frac{d_{w}}{4} \quad \alpha_{2} > \frac{d_{w}}{2}, \quad \alpha_{0,2} > \frac{\alpha_{2}}{2}+ \frac{d_{w}}{4} . 
\end{align}
\end{proposition}
\begin{remark}[Role of Correct Specification]\label{remark:correcspec}
Proposition \ref{prop:materngp} assumes that the sampling model (\ref{eq:samplingmodelY})--(\ref{eq:samplingmodelX}) is correctly specified. This is restrictive relative to Theorem \ref{thm:BVM} (and Proposition \ref{prop:uniform}). Correct specification is used at only one point of the proof of Proposition \ref{prop:materngp} and that is to verify Part 2 of Proposition \ref{prop:suffgeneralsieve}. Specifically, I use the property $E_{P_{0,YX|W}^{(n)}}[L_{n}(\beta,m)/L_{n}(\beta_{0},m_{0})] = 1$ to show the expected value of the numerator of the posterior satisfies
\begin{align*}
    E_{P_{0,YX|W}^{(n)}}\int_{[-B,B] \times \mathcal{M}_{n}^{c}}\frac{L_{n}(\beta,m)}{L_{n}(\beta_{0},m_{0})}d \Pi(\beta,m) \leq \Pi_{\mathcal{M}}(\mathcal{M}_{n}^{c}).
\end{align*}This allows $\Pi_{\mathcal{M}}(\mathcal{M}_{c}^{c}) \leq \exp(-Mn\delta_{n}^{2})$ for $M>0$ large to verify $\Pi(\mathcal{M}_{n}^{c}|\{(Y_{i},X_{i},w_{i}')'\}_{i=1}^{n})\rightarrow 0$ in $P_{0,YX|W}^{(n)}$-probability, an inequality that the sieves $\{\mathcal{M}_{n}\}_{n \geq 1}$ introduced in the Step 2 of the proof satisfy. If it could be shown that $\Pi(\mathcal{M}_{n}^{c}|\{(Y_{i},X_{i},w_{i}')'\}_{i=1}^{n})\rightarrow 0$ in $P_{0,YX|W}^{(n)}$-probability without correct specification of (\ref{eq:samplingmodelY})--(\ref{eq:samplingmodelX}) for the sieves $\{\mathcal{M}_{n}\}_{n \geq 1}$ (or for some modification of $\{\mathcal{M}_{n}\}_{n \geq 1}$ that does not change the metric entropy in a meaningful way), then Theorem \ref{thm:BVM} could be verified only under the restriction that $\varepsilon$ is subgaussian conditional on $W$ because the maximal inequalities used to verify Part 3 of Proposition \ref{prop:suffgeneralsieve} (and Assumption \ref{as:emp_process}) only require subgaussianity (i.e., they apply Corollary 2.2.9 of \cite{vaartwellner96book}).
\end{remark}
\begin{remark}[No Cross Restrictions on Regularity of $m_{01}$ and $m_{02}$]\label{remark:maternreg}
Proposition \ref{prop:materngp} is similar to the uniform wavelet series example from Section \ref{sec:waveletex} in the sense that the result does not impose cross restrictions on the nuisance function. The inequalities of the proposition only compare the regularity of the prior draws $m_{j}$ to the function $m_{0,j}$ they are targeting. Specifically, the restriction $\alpha_{0,j} > \alpha_{j}/2 + d_{w}/4$ means draws $m_{j}$ from the marginal $\Pi_{\mathcal{M}_{j}}$ cannot be too smooth relative to the true nuisance parameter $m_{0,j}$, however, this does not restrict the smoothness of the draws from $\Pi_{\mathcal{M}_{1}}$ and $m_{02}$ (and vice versa). I will return to this point in Section \ref{sec:comparison}.
\end{remark}
\section{Discussion}\label{sec:discuss}
This section presents a discussion of Theorem \ref{thm:BVM} and compares the Bernstein-von Mises theorem for the $(\beta,m)$-parametrization and the $(\beta,\eta)$-parametrization of the partially linear model.
\subsection{Why Theorem \ref{thm:BVM} Works}\label{sec:whyitworks}
Theorem \ref{thm:BVM} proves that the marginal posterior for the coefficient of interest $\beta$ is asymptotically normal. The main requirements are that the marginal posterior for $m$ concentrates around $m_{0}$ at a rate faster than $n^{-1/4}$ and certain multiplier empirical processes satisfy asymptotic equicontinuity conditions. These conditions are mild and mirror those required for the asymptotic normality of method of moments estimators based on orthogonal estimating equations \citep{andrews1994asymptotics,chernozhukov2018double}. Moreover, Theorem \ref{thm:BVM} can be valid when the sampling model (\ref{eq:samplingmodelY})--(\ref{eq:samplingmodelX}) is misspecified. Specifically, the uniform wavelet series prior example in Section \ref{sec:waveletex} satisfies Theorem \ref{thm:BVM} without requiring that the true conditional distribution $P_{0,YX|W}$ of $(Y,X)$ given $W$ be compatible with (\ref{eq:samplingmodelY})--(\ref{eq:samplingmodelX}). The next result is important for explaining these features.
\begin{theorem}\label{thm:KL}
If Assumption \ref{as:dgp} holds, then
\begin{enumerate}
    \item For any $\beta \in \mathcal{B}$ fixed, $m_{0}=(m_{01},m_{02})$ is the unique solution to the minimization problem
\begin{align}\label{eq:profile}
    \min_{m \in \mathcal{M}}E_{P_{0}}\left[KL\left(p_{0,YX|W}(\cdot|W),p_{\beta,m}(\cdot|W)\right)\right].
\end{align}
\item The true parameter $(\beta_{0},m_{0})$ is the unique solution to the minimization problem
\begin{align}\label{eq:full}
    \min_{(\beta,m) \in \mathcal{B} \times \mathcal{M}}E_{P_{0}}\left[KL\left(p_{0,YX|W}(\cdot|W),p_{\beta,m}(\cdot|W)\right)\right],
\end{align}
\end{enumerate}
where $KL(p,q)$ denotes the Kullback-Leibler divergence between two densities $p$ and $q$.
\end{theorem}
The connection between Theorem \ref{thm:KL} and Theorem \ref{thm:BVM} is as follows. Part 1 shows that the true nuisance function $m_{0}$ maximizes the population log-likelihood function when $\beta$ is fixed, and, since the Gaussian distributions with unknown means and known variances form linear exponential families, it may be viewed as a nonparametric application of Theorem 1 of \cite{gourieroux1984pseudo} (and Theorem 5.4 of \cite{white1994estimation}).\footnote{Since $E_{P_{0}}[\log p_{\beta,m}(Y,X|W)] = -E_{P_{0}}[KL(p_{0,YX|W}(\cdot|W),p_{\beta,m}(\cdot|W))] + E_{P_{0}}[\log p_{0,YX|W}(Y,X|W)]$, the solutions to (\ref{eq:profile}) and (\ref{eq:full}) coincide with that obtained when maximimizing $E_{P_{0}}[\log p_{\beta,m}(Y,X|W)]$ with $m$ and $(\beta,m)$, respectively.} The result is important for Theorem \ref{thm:BVM} because the independence of the solution $m_{0}$ from $\beta$ reveals that there is no loss of information associated with not knowing the nuisance parameter $m_{0}$. This means the $(\beta,m)$-parametrization is an example of an \textit{adaptive semiparametric model} \citep{bickel1982adaptive}. Adaptive semiparametric models have the property that the ordinary score (i.e., the derivative of $\log p_{\beta,m}$ with respect to $\beta$ at $(\beta_{0},m_{0})$) coincides with the \textit{efficient score} (i.e., the ordinary score of \textit{least favorable parametric submodel} $\beta \mapsto p_{\beta,m_{0}}$ at $\beta_{0}$). Consequently, the asymptotic analysis of the semiparametric model is essentially that of a regular parametric model because ordinary local asymptotic normality (LAN) expansions along parametric submodels $\beta \mapsto p_{\beta,m}$ respect the semiparametric structure of the model. This makes the primary concern whether $(\beta_{0},m_{0})$ is identifiable from the sampling model (\ref{eq:samplingmodelY})--(\ref{eq:samplingmodelX}). Part 2 confirms identifiability because it establishes that the true parameter $(\beta_{0},m_{0})$ uniquely maximizes the population log-likelihood even though the sampling model may be misspecified, and, similar to Part 1, may be viewed as a nonparametric extension of Theorem 6 of \cite{gourieroux1984pseudo} since Gaussian distributions with unknown means and variances form quadratic exponential families. The sole purpose of Assumptions \ref{as:consistency} and \ref{as:emp_process} is to ensure that remainders of these ordinary LAN expansions (i.e., due to estimation of $m_{0}$) vanish appropriately as the sample size grows.
\begin{remark}[Efficiency]
The achievement of the \cite{chamberlain1992efficiency} efficiency bound under homoskedasticity despite (\ref{eq:samplingmodelY})--(\ref{eq:samplingmodelX}) being possibly misspecified reflects the fact that maximum likelihood estimation in the least favorable parametric submodel $\{p_{\beta,m_{0}}: \beta \in \mathcal{B}\}$ is that of a linear exponential family, and, as a result, the maximum likelihood estimator has the same asymptotic distribution as a weighted least squares estimator based on the conditional moment restrictions $E_{P_{0}}[Y|X,W] = m_{01}(W)+(X-m_{02}(W))\beta_{0}$ with weights $1/\sigma_{01}^{2}$ (see, for example, Theorem 4 of \cite{gourieroux1984pseudo} and Theorem 6.13 of \cite{white1994estimation}).    
\end{remark}
\begin{remark}[Connection to \cite{castillo2012semiparametric}]
Theorem \ref{thm:KL} establishes that the $(\beta,m)$-parametrization is an example of an adaptive semiparametric model, and, as a result, it falls within the class of semiparametric models \cite{castillo2012semiparametric} refers to as `the case without loss of information' (see Section 1.3 of his paper). Consequently, it is foreseeable that Theorem 1 of his paper could be verified to establish the Bernstein-von Mises theorem (at least in the correctly specified case). The structure of the Gaussian (quasi-)likelihood makes it convenient to apply direct arguments rather than employ his more abstract, yet more generally applicable, LAN conditions.
\end{remark}
\subsection{Comparison with the Original Parametrization}\label{sec:comparison}
I compare my results with the Bernstein-von Mises theory for the $(\beta,\eta)$-parametrization. Bayesian inference in the $(\beta,\eta)$-parametrization is typically based on the following model for the conditional distribution of $Y$ given $X$ and $W$ (for example, Section 7 of \cite{bickel2012semiparametric}):
\begin{align}\label{eq:betaeta}
    Y_{i}|\{(X_{i},W_{i}')'\}_{i=1}^{n},\beta,\eta \overset{ind}{\sim} \mathcal{N}(X_{i}\beta+\eta(W_{i}),\sigma_{01}^{2}), \ i=1,..., n, \quad \beta \sim \Pi_{\mathcal{B}}, \ \eta \sim \Pi_{\mathcal{H}}, \ \beta \independent \eta,
\end{align}
where $\beta \in \mathcal{B} \subseteq \mathbb{R}$, $\eta \in \mathcal{H} \subseteq L^{2}(\mathcal{W})$, $\Pi_{\mathcal{B}}$ is a probability distribution over $\mathcal{B}$, $\Pi_{\mathcal{H}}$ is a probability distribution over $\mathcal{H}$, and $\sigma_{01}^{2} \in (0,\infty)$ is a known constant. Since the parameters $(\beta_{0},\eta_{0})$ are identifiable from the conditional distribution of $Y$ given $X$ and $W$, there is no need to include a model for the conditional distribution of $X$ given $W$.

The key distinction between the sampling model (\ref{eq:betaeta}) and the sampling model (\ref{eq:samplingmodelY})--(\ref{eq:samplingmodelX}) is that the former is not adaptive. Indeed, for $\beta \in \mathcal{B}$ fixed, the maximizer of the population log-likelihood function with respect to $\eta$ is $\eta_{\beta} = \eta_{0} - (\beta-\beta_{0})m_{02}$, which means there is loss of information associated with not knowing the nuisance parameter $\eta$ (unless $m_{02} = 0$). This makes finding priors/data generating processes that satisfy the Bernstein-von Mises theorem more challenging because the prior $\Pi_{\mathcal{H}}$ must be chosen so it is suitable for both estimation of $\eta_{0}$ \textit{and} $m_{02}$. Proposition \ref{prop:materninformationloss} demonstrates this for the case where $\Pi_{\mathcal{H}}$ is the law of a centered Mat\'{e}rn Gaussian process.\footnote{The proposition also imposes the same restrictions on the data-generating process as Proposition \ref{prop:materngp} to ensure fair comparison.} A discussion follows.
\begin{proposition}\label{prop:materninformationloss}
Suppose that $Y|X,W \sim \mathcal{N}(X\beta_{0}+\eta_{0}(W),\sigma_{01}^{2})$ and $X|W \sim \mathcal{N}(m_{02}(W),\sigma_{02}^{2})$, where $\beta_{0} \in (-B,B)$, $\eta_{0} \in C^{\alpha_{0,\eta}}([0,1]^{d_{w}}) \cap H^{\alpha_{0,\eta}}([0,1]^{d_{w}})$, $\alpha_{0,\eta} > 0$, $m_{02} \in C^{\alpha_{0,2}}([0,1]^{d_{w}}) \cap H^{\alpha_{0,2}}([0,1]^{d_{w}})$, $\alpha_{0,2}>0$, and $\sigma_{01}^{2},\sigma_{02}^{2} \in (0,\infty)$ are known. Suppose further that Bayesian inference is based on (\ref{eq:betaeta}) with $\Pi_{\mathcal{B}}$ equal to the uniform distribution over $\mathcal{B} =[-B,B]$ and $\Pi_{\mathcal{H}}$ equal to the probability law of a centered Mat\'{e}rn Gaussian process with regularity parameter $\alpha_{\eta} > 0$. If $\alpha_{\eta},\alpha_{0,\eta}$, and $\alpha_{0,2}$ satisfy the following inequalities
\begin{align}\label{eq:betaetaineq}
\alpha_{\eta} > \frac{d_{w}}{2}, \quad \alpha_{0,\eta} > \frac{\alpha_{\eta}}{2} + \frac{d_{w}}{4}, \quad \alpha_{0,2} > \frac{\alpha_{\eta}}{2} + \frac{d_{w}}{4},    
\end{align}
then the marginal posterior $\Pi\left(\beta \in \cdot |\{(Y_{i},X_{i},W_{i}')'\}_{i=1}^{n}\right)$ satisfies
\begin{align*}
  \left | \left |\Pi\left(\beta \in \cdot |\{(Y_{i},X_{i},w_{i}')'\}_{i=1}^{n}\right) - \mathcal{N}\left(\beta_{0}+\frac{\tilde{\Delta}_{n,0}}{\sqrt{n}},\frac{1}{n}\tilde{I}_{n}(m_{0})^{-1} \right) \right | \right |_{TV} \overset{P_{0,YX|W}^{(n)}}{\longrightarrow} 0
\end{align*}
for $P_{0,W}^{\infty}$-almost every fixed realization $\{w_{i}\}_{i \geq 1}$ of $\{W_{i}\}_{i \geq 1}$ as $n\rightarrow \infty$.
\end{proposition}
I compare Proposition \ref{prop:materngp} and \ref{prop:materninformationloss} to elaborate on the aforementioned challenges associated with Bayesian inference for the $(\beta,\eta)$-parametrization. Since both propositions assume correct specification and are based on the same class of Gaussian process priors, the inequalities for the regularity parameters is the natural point of comparison. The first two inequalities of (\ref{eq:betaetaineq}) are comparable with the inequalities (\ref{eq:betamineq}) in Proposition \ref{prop:materngp} in that they require the draws from the prior to not be too smooth relative to the function \textit{they are targeting} (i.e., draws from $\Pi_{\mathcal{H}},\Pi_{\mathcal{M}_{1}}$, and $\Pi_{\mathcal{M}_{2}}$ cannot be too smooth relative to $\eta_{0}$, $m_{01}$, and $m_{02}$ respectively). However, as mentioned earlier, Proposition \ref{prop:materninformationloss} also restricts the regularity of the draws from the prior $\Pi_{\mathcal{H}}$ \textit{relative} to $m_{02}$ because the third part of (\ref{eq:betaetaineq}) states that $\eta\sim \Pi_{\mathcal{H}}$ cannot be too smooth relative to $m_{02}$. This relates to the nonadaptivity of the $(\beta,\eta)$-parametrization. Indeed, verifying the appropriate LAN expansions requires a change of parametrization $(\beta,\eta)\mapsto (\beta_{0},\tilde{\eta}_{n}(\beta,\eta))$, where $\tilde{\eta}_{n}(\beta,\eta) = \eta + (\beta-\beta_{0})m_{n,2}$ for some sequence $\{m_{n,2}\}_{n \geq 1}$ in the reproducing kernel Hilbert space $(\mathbb{H}_{\eta},||\cdot||_{\mathbb{H}_{\eta}})$ of the Gaussian process $\eta$ for which $||m_{n,2}||_{\mathbb{H}_{\eta}} \leq 2\sqrt{n}\rho_{n}$ and $||m_{n,2}-m_{02}||_{\infty} \leq \rho_{n}$ for $\rho_{n} = n^{-\min\{\alpha_{\eta},\alpha_{0,2}\}/(2\alpha_{\eta}+d_{w})}$. Control of the LAN remainder requires $\rho_{n} = o(n^{-1/4})$, a condition that holds if and only if $\alpha_{\eta} > d_{w}/2$ and $\alpha_{0,2} > \alpha_{\eta}/2 + d_{w}/4$ (i.e., the first and third parts of (\ref{eq:betaetaineq})). This indicates that the Bernstein-von Mises theorem may fail for the $(\beta,\eta)$-parametrization in situations where $m_{02}$ is not smooth relative to $\eta \sim \Pi_{\mathcal{H}}$. In contrast, Remark \ref{remark:maternreg} points out that there are no cross-restrictions between the regularity of the nuisance functions and the draws from the prior $\Pi_{\mathcal{M}}$ in Proposition \ref{prop:materngp} (i.e., $m_{02}$ could potentially be much less smooth than $m_{1} \sim \Pi_{\mathcal{M}_{1}}$ without violating Proposition \ref{prop:materngp}). Put differently, the $(\beta,m)$-parametrization enables the researcher to direct nuisance priors towards the specific functions they are estimating without concern that this will adversely affect estimation the regression coefficient $\beta_{0}$.

A more subtle point of comparison relates to the requirement that the sequence $\{m_{n,2}\}_{n \geq 1}$ be in the reproducing kernel Hilbert space $(\mathbb{H}_{\eta},||\cdot||_{\mathbb{H}_{\eta}})$ of the Gaussian process $\eta$. In applying the change of parametrization $(\beta,m)\mapsto (\beta_{0},\tilde{\eta}_{n}(\beta,\eta))$, it must also be shown that the prior $\Pi_{\mathcal{H}}$ is roughly unchanged under the mapping $\eta \mapsto \tilde{\eta}_{n}(\beta,\eta)$ for $\beta$ fixed. Since, for $\beta$ fixed, $\tilde{\eta}_{n}(\beta,\eta)$ concerns a shift of a Gaussian process, it is necessary and sufficient that $\{m_{n,2}\}_{ n\geq 1}$ be in $(\mathbb{H}_{\eta},||\cdot||_{\mathbb{H}_{\eta}})$ so that the Cameron-Martin theorem (Proposition I.20 of \cite{ghosal2017fundamentals}) can be applied to establish the required stability of $\Pi_{\mathcal{H}}$ via an infinite-dimensional change of measure. Since there is no need to change variables in the $(\beta,m)$-parametrization (i.e., ordinary LAN expansions respect the semiparametric structure), this prior stability condition does not arise in my proposed framework. This is important because it can be technically difficult to check this condition for non-Gaussian priors supported on infinite-dimensional spaces (i.e., Cameron-Martin is a tool specific to Gaussian priors), whereas Propositions \ref{prop:suffgeneral} and \ref{prop:suffgeneralsieve} indicate that the key assumptions of Theorem \ref{thm:BVM} can be verified using standard small ball probability conditions for the prior and maximal inequalities for multiplier empirical processes.\footnote{Readers should not be concerned about whether the sets $\{\mathcal{M}_{n}\}_{n \geq 1}$ in Proposition \ref{prop:suffgeneralsieve} pose additional challenges for the $(\beta,m)$-parametrization relative to the $(\beta,\eta)$-parametrization. The proof of Proposition \ref{prop:materninformationloss} illustrates that exact analogs of the sets $\{\mathcal{M}_{n}\}_{n \geq 1}$ constructed in the proof of Proposition \ref{prop:materngp} are utilized to verify corresponding asymptotic equicontinuity conditions in the LAN expansions for the $(\beta,\eta)$-parametrization (specifically, see $\{\tilde{\mathcal{H}}_{n}\}_{n \geq 1}$ in Step 1 of the proof of Proposition \ref{prop:materninformationloss}).}  This allows me to check Theorem \ref{thm:BVM} for other priors (e.g., untruncated uniform wavelet series prior) using similar technical devices to the Gaussian case. The points raised in this paragraph and the previous paragraph illustrate that there may be benefits from the $(\beta,m)$-parametrization.
\section{Conclusion}\label{sec:conclusion}
This paper proves a Bernstein-von Mises theorem for the partially linear model. The theorem is based on a feasible adaptive parametrization of the model that is based on \cite{robinson1988root} transformation. An important feature of the result is that it avoids the prior invariance condition. The examples in Sections \ref{sec:waveletex} and \ref{sec:matern} indicate that this is beneficial because they show that my approach does not impose cross-restrictions on the regularity of unknown functions. In contrast, the original parametrization with independent priors requires that the prior for the nuisance function $\eta$ be simultaneously appropriate for estimating $\eta_{0}$ \textit{and} $m_{02}$, a condition that leads to restrictions on smoothness of the conditional expectation $m_{02}$ relative to the prior draws $\eta \sim \Pi_{\mathcal{M}}$. Moreover, the avoidance of a change of variables in the $(\beta,m)$-parametrization means that my approach can bypass an infinite-dimensional change of measure, a technical condition that may be hard to verify for complicated priors.

There are a number of extensions that would be interesting to pursue. First, it would be interesting to consider rate-adaptive priors for $m_{1}$ and $m_{2}$. I conjecture Theorem \ref{thm:BVM} could be verified for rate adaptive priors (at least in the case of correct specification). The Mat\'{e}rn prior example in Section \ref{sec:matern} follows a similar argument as Theorem 3.1 of \cite{dejonge2013semiparametric}, and, as a result, it is reasonable that the argument of Theorem 4.1 in their paper could also be modified to verify Theorem \ref{thm:BVM} in the case of hierarchical spline priors, an example of a prior that can yield adaptive, rate-optimal posterior rates of contraction \citep{dejonge2012adaptive}. However, the extent to which adaptive priors can be accommodated should be formally investigated. Second, it would be interesting to consider the case where $\eta(W) = W'\eta$ but the dimension of $W$ is large relative to the sample size $n$. In this case, the sampling model (\ref{eq:samplingmodelY})--(\ref{eq:samplingmodelX}) coincides with the setup of \cite{hahn2018regularization} and their simulations indicate favorable performance when using shrinkage priors for the nuisance parameters. \cite{hahn2018regularization} do not provide any formal asymptotic analysis, and, for this reason, it would be interesting to examine whether an asymptotic normality result could be established in this setting. Finally, the general idea in this paper is to use knowledge of the profile likelihood \citep{severini1992profile,murphy2000profile} to derive an adaptive parametrization of the partially linear model. An important extension is determining a more general class of models for which this technique could be applied.
\bibliographystyle{abbrvnat}
\bibliography{PartialLinear} 
\newpage
\appendix
\section{Proof of Theorem \ref{thm:BVM}, Corollary \ref{cor:BVM_unconditional}--\ref{cor:quantile}, and Associated Lemmas}
\subsection{Proof of Theorem \ref{thm:BVM}}
I state two important lemmas relevant for Theorem \ref{thm:BVM}. Their proofs are found in Appendix \ref{ap:lemmas}.
\begin{lemma}\label{lem:technical}
Let $\{\mathcal{M}_{n} \cap B_{n,\mathcal{M}}(m_{0},D\delta_{n})\}_{n=1}^{\infty}$ be the sets defined in Assumption \ref{as:consistency} and \ref{as:emp_process}. If Assumptions \ref{as:dgp}, \ref{as:consistency}, and \ref{as:emp_process} hold, then the following holds:
\begin{enumerate}
\item The collection of functions $\{\tilde{\ell}_{n}(\beta_{0},m): m \in \mathcal{M}\}$ satisfies
\begin{align*}
  \sup_{m \in B_{n,\mathcal{M}}(m_{0},D\delta_{n})\cap \mathcal{M}_{n}}\left| \frac{1}{\sqrt{n}}\tilde{\ell}_{n}(\beta_{0},m) - \frac{1}{\sqrt{n}}\tilde{\ell}_{n}(\beta_{0},m_{0}) \right|  \overset{P_{0,YX|W}^{(n)}}{\longrightarrow} 0
\end{align*}
 for $P_{0,W}^{\infty}$-almost every fixed realization $\{w_{i}\}_{i \geq 1}$ of $\{W_{i}\}_{i \geq 1}$ as $n\rightarrow \infty$.
\item The collection of functions $\{\tilde{I}_{n}(m): m \in \mathcal{M}\}$ satisfies
\begin{align*}
 \sup_{m \in B_{n,\mathcal{M}}(m_{0},D\delta_{n})\cap \mathcal{M}_{n}}\left| \tilde{I}_{n}(m)-\tilde{I}_{n}(m_{0}) \right| \overset{P_{0,YX|W}^{(n)}}{\longrightarrow} 0   
\end{align*}
 for $P_{0,W}^{\infty}$-almost every fixed realization $\{w_{i}\}_{i \geq 1}$ of $\{W_{i}\}_{i \geq 1}$ as $n\rightarrow \infty$.
\item Let $\underline{c}_{0}=\underline{c}/\sigma_{01}^{2}$ and $\overline{c}_{0} = \overline{c}/\sigma_{01}^{2}$, where $\overline{c}$ and $\underline{c}$ are the constants Part 2 of Assumption \ref{as:dgp}. The following statements hold:
\begin{align*}
   P_{0,YX|W}^{(n)}\left( \inf_{m \in \mathcal{M}_{n} \cap B_{\mathcal{M},\infty}(m_{0},D\delta_{n})}\tilde{I}_{n}(m)\geq \underline{c}_{0} \right) \rightarrow 1
   \end{align*}
   and
   \begin{align*}
    P_{0,YX|W}^{(n)}\left(\sup_{m \in \mathcal{M}_{n} \cap B_{\mathcal{M},\infty}(m_{0},D\delta_{n})}\tilde{I}_{n}(m) \leq \overline{c}_{0}\right) \rightarrow 
\end{align*}
 for $P_{0,W}^{\infty}$-almost every fixed realization $\{w_{i}\}_{i \geq 1}$ of $\{W_{i}\}_{i \geq 1}$ as $n\rightarrow \infty$.
\item Let $R_{n}(\beta,m)$ be given by
\begin{align*}
R_{n}(\beta,m) = \ell_{n}(\beta,m)-\ell_{n}(\beta_{0},m)- \frac{1}{\sqrt{n}}\tilde{\ell}_{n}(\beta_{0},m_{0})\sqrt{n}(\beta-\beta_{0})+\frac{n}{2}(\beta-\beta_{0})^{2}\tilde{I}_{n}(m_{0})     
\end{align*}
for each $(\beta,m) \in \mathcal{B}\times \mathcal{M}$ and let $B_{\mathcal{B}}(\beta_{0},M_{n}/\sqrt{n}) = \{\beta \in \mathcal{B}: |\beta-\beta_{0}| < M_{n}/\sqrt{n}\}$. If $M_{n}\rightarrow \infty$ arbitrarily slowly as $n\rightarrow \infty$, then
\begin{align*}
\sup_{(\beta,m) \in B_{\mathcal{B}}(\beta_{0},M_{n}/\sqrt{n})\times (B_{\mathcal{M},\infty}(m_{0},D\delta_{n})\cap \mathcal{M}_{n})}\left | R_{n}(\beta,m) \right| \overset{P_{0,YX|W}^{(n)}}{\longrightarrow} 0
\end{align*}
 for $P_{0,W}^{\infty}$-almost every fixed realization $\{w_{i}\}_{i \geq 1}$ of $\{W_{i}\}_{i \geq 1}$ as $n\rightarrow \infty$.
\end{enumerate}
\end{lemma}
\begin{lemma}\label{lem:rootn}
Let $\{\mathcal{M}_{n} \cap B_{n,\mathcal{M}}(m_{0},D\delta_{n})\}_{n=1}^{\infty}$ be the sets defined in Assumptions \ref{as:consistency} and \ref{as:emp_process}. If Assumptions \ref{as:dgp}, \ref{as:prior}, \ref{as:consistency}, and \ref{as:emp_process} hold, then for every $M_{n}\rightarrow \infty$,
\begin{align*}
    \sup_{m \in \mathcal{M}_{n} \cap B_{n,\mathcal{M}}(m_{0},D\delta_{n})}\Pi\left(\beta \in B_{\mathcal{B}}(\beta_{0},M_{n}/\sqrt{n})^{c}|\{(Y_{i},X_{i},W_{i}')'\}_{i=1}^{n},m\right) \overset{P_{0,YX|W}^{(n)}}{\rightarrow} 0
\end{align*}
for $P_{0,W}^{\infty}$-almost every fixed realization $\{w_{i}\}_{i \geq 1}$ of $\{W_{i}\}_{i \geq 1}$ as $n\rightarrow \infty$.
\end{lemma}
 \begin{proof}[Proof of Theorem \ref{thm:BVM}]
The proof has four steps. Step 1 and Step 2 utilize Assumption \ref{as:consistency} and Lemma \ref{lem:rootn} to conclude it suffices that $\sup_{m \in \mathcal{M}_{n} \cap B_{\mathcal{M},\infty}(m_{0},D\delta_{n})}\Pi(\beta \in \cdot |\{(Y_{i},X_{i},W_{i}')'\}_{i=1}^{n},\beta \in B_{\mathcal{B}}(\beta_{0},M_{n}/\sqrt{n}),m)$ and $\mathcal{N}(\tilde{\Delta}_{n,0},\tilde{I}_{0}(m_{0})^{-1})$ are first-order asymptotically equivalent. Step 3 uses the assumption about the density of $\Pi_{\mathcal{B}}$ and Part 2 of Lemma \ref{lem:technical} to further simplify the problem to one in which $\beta$ is the only free parameter. Step 4 shows asymptotic normality.

\textit{Step 1.} Applying the law of total probability,
\begin{align*}
&\Pi(\beta \in A|\{(Y_{i},X_{i},w_{i}')'\}_{i=1}^{n}) \\
&\quad = \int_{\mathcal{M}_{n} \cap B_{n,\mathcal{M}}(m_{0},D\delta_{n})}\Pi(\beta \in A|\{(Y_{i},X_{i},w_{i}')'\}_{i=1}^{n},m) d \Pi(m|\{(Y_{i},X_{i},w_{i}')'\}_{i=1}^{n}) \\
&\quad \quad  + \int_{\mathcal{M}_{n} \cap B_{n,\mathcal{M}}(m_{0},D\delta_{n})^{c}}\Pi(\beta \in A|\{(Y_{i},X_{i},w_{i}')'\}_{i=1}^{n},m) d \Pi(m|\{(Y_{i},X_{i},w_{i}')'\}_{i=1}^{n}) \\
&\quad \quad + \int_{\mathcal{M}_{n}^{c}}\Pi(\beta \in A|\{(Y_{i},X_{i},w_{i}')'\}_{i=1}^{n},m) d \Pi(m|\{(Y_{i},X_{i},w_{i}')'\}_{i=1}^{n}).
\end{align*}
for any event $A$. As $0\leq \Pi(\beta \in \cdot|\{(Y_{i},X_{i},w_{i}')'\}_{i=1}^{n},m) \leq 1$, it follows that 
\begin{align*}
    &\sup_{A }\int_{\mathcal{M}_{n} \cap B_{n,\mathcal{M}}(m_{0},D\delta_{n})^{c}}\Pi(\beta \in A|\{(Y_{i},X_{i},w_{i}')'\}_{i=1}^{n},m) d \Pi(m|\{(Y_{i},X_{i},w_{i}')'\}_{i=1}^{n}) \\
    &\quad \leq \Pi(m \in B_{\mathcal{M},\infty}(m_{0},D\delta_{n})^{c}|\{(Y_{i},X_{i},w_{i}')'\}_{i=1}^{n})
\end{align*}
and
\begin{align*}
 \sup_{A }\int_{\mathcal{M}_{n}^{c}}\Pi(\beta \in A|\{(Y_{i},X_{i},w_{i}')'\}_{i=1}^{n},m) d \Pi(m|\{(Y_{i},X_{i},w_{i}')'\}_{i=1}^{n}) \leq \Pi(m \in \mathcal{M}_{n}^{c}|\{(Y_{i},X_{i},w_{i}')'\}_{i=1}^{n}).    
\end{align*}
Applying Assumptions \ref{as:consistency} and \ref{as:emp_process}, it then follows that
\begin{align*}
\sup_{A }\int_{\mathcal{M}_{n} \cap B_{n,\mathcal{M}}(m_{0},D\delta_{n})^{c}}\Pi(\beta \in A|\{(Y_{i},X_{i},w_{i}')'\}_{i=1}^{n},m) d \Pi(m|\{(Y_{i},X_{i},w_{i}')'\}_{i=1}^{n})  \overset{P_{0,YX|W}^{(n)}}{\longrightarrow} 0    
\end{align*}
and
\begin{align*}
\sup_{A }\int_{\mathcal{M}_{n}^{c}}\Pi(\beta \in A|\{(Y_{i},X_{i},w_{i}')'\}_{i=1}^{n},m) d \Pi(m|\{(Y_{i},X_{i},W_{i}')'\}_{i=1}^{n})  \overset{P_{0,YX|W}^{(n)}}{\longrightarrow} 0.    
\end{align*}
for $P_{0,W}^{\infty}$-almost every fixed realization $\{w_{i}\}_{i\geq 1}$ of $\{W_{i}\}_{i \geq 1}$ as $n\rightarrow \infty$. As a result,
\begin{align*}
    \sup_{A }\left | \Pi(\beta \in A|\{(Y_{i},X_{i},w_{i}')'\}_{i=1}^{n})-F_{n,1}(A) \right | \overset{P_{0,YX|W}^{(n)}}{\longrightarrow} 0
\end{align*}
for $P_{0,W}^{\infty}$-almost every fixed realization $\{w_{i}\}_{i\geq 1}$ of $\{W_{i}\}_{i \geq 1}$ as $n\rightarrow \infty$, where $A \mapsto F_{n,1}(A)$ is a set function that satisfies
\begin{align*}
F_{n,1}(A) = \int_{\mathcal{M}_{n} \cap B_{n,\mathcal{M}}(m_{0},D\delta_{n})}\Pi(\beta \in A|\{(Y_{i},X_{i},w_{i}')'\}_{i=1}^{n},m) d \Pi(m|\{(Y_{i},X_{i},w_{i}')'\}_{i=1}^{n}).    
\end{align*}
\textit{Step 2.} Let $M_{n}\rightarrow \infty$ arbitrarily slowly. Since $A\cap B_{\mathcal{B}}(\beta_{0},M_{n}/\sqrt{n})$ and $A \cap B_{\mathcal{B}}(\beta_{0},M_{n}/\sqrt{n})^{c}$ define a partition of $A$, it follows that
\begin{align*}
    &\sup_{A }|F_{n,1}(A)-F_{n,1}(A \cap B_{\mathcal{B}}(\beta_{0},M_{n}/\sqrt{n}))|  \\
    &\quad = \sup_{A }F_{n,1}(A \cap B_{\mathcal{B}}(\beta_{0},M_{n}/\sqrt{n})^{c}) \\
       &\quad \leq \sup_{A }\sup_{m \in \mathcal{M}_{n} \cap B_{n,\mathcal{M}}(m_{0},D\delta_{n})}\Pi(A \cap B_{\mathcal{B}}(\beta_{0},M_{n}/\sqrt{n})^{c}|\{(Y_{i},X_{i},w_{i}')'\}_{i=1}^{n},m)  \\
       &\quad \quad \times \Pi(\mathcal{M}_{n} \cap B_{n,\mathcal{M}}(m_{0},D\delta_{n})|\{(Y_{i},X_{i},w_{i}')'\}_{i=1}^{n}) \\
    &\quad \leq \sup_{m \in \mathcal{M}_{n} \cap B_{n,\mathcal{M}}(m_{0},D\delta_{n})}\Pi(B_{\mathcal{B}}(\beta_{0},M_{n}/\sqrt{n})^{c}|\{(Y_{i},X_{i},w_{i}')'\}_{i=1}^{n},m)
\end{align*}
with the first inequality is H\"{o}lder's inequality, and the second uses that $A \cap B_{\mathcal{B}}(\beta_{0},M_{n}/\sqrt{n})^{c} \subseteq B_{\mathcal{B}}(\beta_{0},M_{n}/\sqrt{n})^{c}$ for all $A $ and that probabilities are bounded from above by one. Applying Lemma \ref{lem:rootn}, I then obtain
\begin{align*}
    \sup_{m \in \mathcal{M}_{n} \cap B_{n,\mathcal{M}}(m_{0},D\delta_{n})}\Pi(B_{\mathcal{B}}(\beta_{0},M_{n}/\sqrt{n})^{c}|\{(Y_{i},X_{i},w_{i}')'\}_{i=1}^{n},m) \overset{P_{0,YX|W}^{(n)}}{\longrightarrow} 0
\end{align*}
for $P_{0,W}^{\infty}$-almost every fixed realization $\{w_{i}\}_{i\geq 1}$ of $\{W_{i}\}_{i \geq 1}$ as $n\rightarrow \infty$, and, as a result,
\begin{align*}
    \sup_{A }|F_{n,1}(A)-F_{n,1}(A \cap B_{\mathcal{B}}(\beta_{0},M_{n}/\sqrt{n}))| \overset{P_{0,YX|W}^{(n)}}{\longrightarrow} 0
\end{align*}
for $P_{0,W}^{\infty}$-almost every fixed realization $\{w_{i}\}_{i\geq 1}$ of $\{W_{i}\}_{i \geq 1}$ as $n\rightarrow \infty$. Moreover, I know that
\begin{align*}
&\left |\int_{\mathcal{M}_{n} \cap B_{n,\mathcal{M}}(m_{0},D\delta_{n})}\left(\frac{\int_{B_{\mathcal{B}}(\beta_{0},M_{n}/\sqrt{n})}L_{n}(\beta,m)d \Pi_{\mathcal{B}}(\beta)}{\int_{\mathcal{B}}L_{n}(\beta,m)d \Pi_{\mathcal{B}}(\beta)} -1 \right) 
 d \Pi(m|\{(Y_{i},X_{i},w_{i}')'\}_{i=1}^{n})\right| \\
&\quad = \int_{\mathcal{M}_{n} \cap B_{n,\mathcal{M}}(m_{0},D\delta_{n})}\Pi(\beta \in B_{\mathcal{B}}(\beta_{0},M_{n}/\sqrt{n})^{c}|\{(Y_{i},X_{i},w_{i}')'\}_{i=1}^{n},m) d \Pi(m|\{(Y_{i},X_{i},w_{i}')'\}_{i=1}^{n}) \\
&\quad \leq \sup_{m \in \mathcal{M}_{n} \cap B_{n,\mathcal{M}}(m_{0},D\delta_{n})}\Pi(\beta \in B_{\mathcal{B}}(\beta_{0},M_{n}/\sqrt{n})^{c}|\{(Y_{i},X_{i},w_{i}')'\}_{i=1}^{n},m) 
\end{align*}
with the inequality applying H\"{o}lder's inequality and that probabilities are bounded by one. Lemma \ref{lem:rootn} then implies that
\begin{align*}
\sup_{m \in \mathcal{M}_{n} \cap B_{n,\mathcal{M}}(m_{0},D\delta_{n})}\Pi(\beta \in B_{\mathcal{B}}(\beta_{0},M_{n}/\sqrt{n})^{c}|\{(Y_{i},X_{i},w_{i}')'\}_{i=1}^{n},m) \overset{P_{0,YX|W}^{(n)}}{\longrightarrow} 0    
\end{align*}
for $P_{0,W}^{\infty}$-almost every fixed realization $\{w_{i}\}_{i\geq 1}$ of $\{W_{i}\}_{i \geq 1}$ as $n\rightarrow \infty$, and, as a result,
\begin{align*}
    \left |\int_{\mathcal{M}_{n} \cap B_{n,\mathcal{M}}(m_{0},D\delta_{n})}\left(\frac{\int_{B_{\mathcal{B}}(\beta_{0},M_{n}/\sqrt{n})}L_{n}(\beta,m)d \Pi_{\mathcal{B}}(\beta)}{\int_{\mathcal{B}}L_{n}(\beta,m)d \Pi_{\mathcal{B}}(\beta)} -1 \right)d \Pi(m|\{(Y_{i},X_{i},w_{i}')'\}_{i=1}^{n} \right|  \overset{P_{0,YX|W}^{(n)}}{\longrightarrow} 0 
\end{align*}
for $P_{0,W}^{\infty}$-almost every fixed realization $\{w_{i}\}_{i\geq 1}$ of $\{W_{i}\}_{i \geq 1}$ as $n\rightarrow \infty$. Consequently, I can conclude that
\begin{align*}
  \sup_{A }\left|F_{n,1}(A)-F_{n,2}(A)\right| \overset{P_{0,YX|W}^{(n)}}{\longrightarrow} 0  
\end{align*}
for $P_{0,W}^{\infty}$-almost every fixed realization $\{w_{i}\}_{i\geq 1}$ of $\{W_{i}\}_{i \geq 1}$ as $n\rightarrow \infty$, where $A \mapsto F_{n,2}(A)$ is a set function given by
\begin{align*}
F_{n,2}(A) = \int_{\mathcal{M}_{n} \cap B_{n,\mathcal{M}}(m_{0},D\delta_{n})}\left(\frac{\int_{A \cap B_{\mathcal{B}}(\beta_{0},M_{n}/\sqrt{n})}L_{n}(\beta,m) d \Pi_{\mathcal{B}}(\beta)}{\int_{B_{\mathcal{B}}(\beta_{0},M_{n}/\sqrt{n})}L_{n}(\beta,m) d \Pi_{\mathcal{B}}(\beta)}\right)d \Pi(m|\{(Y_{i},X_{i},w_{i}')'\}_{i=1}^{n})
\end{align*}
for all events $A $.

\textit{Step 3.} Exponentiating and adding/subtracting $\ell_{n}(\beta_{0},m)$, we have that
\begin{align*}
\frac{\int_{A \cap B_{\mathcal{B}}(\beta_{0},M_{n}/\sqrt{n})}L_{n}(\beta,m) d \Pi_{\mathcal{B}}(\beta)}{\int_{B_{\mathcal{B}}(\beta_{0},M_{n}/\sqrt{n})}L_{n}(\beta,m) d \Pi_{\mathcal{B}}(\beta)}=\frac{\int_{A \cap B_{\mathcal{B}}(\beta_{0},M_{n}/\sqrt{n})}\exp(\ell_{n}(\beta,m)-\ell_{n}(\beta_{0},m)) d \Pi_{\mathcal{B}}(\beta)}{\int_{B_{\mathcal{B}}(\beta_{0},M_{n}/\sqrt{n})}\exp(\ell_{n}(\beta,m)-\ell_{n}(\beta_{0},m)) d \Pi_{\mathcal{B}}(\beta)}
\end{align*}
for all events $A $. Recall the definition
\begin{align*}
R_{n}(\beta,m) = \ell_{n}(\beta,m)-\ell_{n}(\beta_{0},m) - \frac{1}{\sqrt{n}}\tilde{\ell}_{n}(\beta_{0},m_{0})\sqrt{n}(\beta-\beta_{0})+\frac{1}{2}\tilde{I}_{n}(m_{0})n(\beta-\beta_{0})^{2}
\end{align*}
for $(\beta,m) \in \mathcal{B}\times \mathcal{M}$. Part 4 of Lemma \ref{lem:technical} implies that 
\begin{align*}
    \overline{R}_{n}:=\sup_{(\beta,m) \in B_{\mathcal{B}}(\beta_{0},M_{n}/\sqrt{n})\times B_{n,\mathcal{M}}(m_{0},D\delta_{n})}\exp(R_{n}(\beta,m)) \overset{P_{0,YX|W}^{(n)}}{\longrightarrow} 1
\end{align*}
and
\begin{align*}
\underline{R}_{n}:=\inf_{(\beta,m) \in B_{\mathcal{B}}(\beta_{0},M_{n}/\sqrt{n})\times B_{n,\mathcal{M}}(m_{0},D\delta_{n})}\exp(R_{n}(\beta,m)) \overset{P_{0,YX|W}^{(n)}}{\longrightarrow} 1.
\end{align*}
for $P_{0,W}^{\infty}$-almost every fixed realization $\{w_{i}\}_{i\geq 1}$ of $\{W_{i}\}_{i \geq 1}$ as $n\rightarrow \infty$. Consequently, there exists a positive sequence $\{c_{n,1}\}_{n=1}^{\infty}$ such that $c_{n,1}\rightarrow 0$ and
\begin{align*}
   P_{0,YX|W}^{(n)}\left( \left | \frac{\overline{R}_{n}}{\underline{R}_{n}} - 1 \right | \leq c_{n,1}\right) \longrightarrow 1
\end{align*}
for $P_{0,W}^{\infty}$-almost every fixed realization $\{w_{i}\}_{i\geq 1}$ of $\{W_{i}\}_{i \geq 1}$ as $n\rightarrow \infty$. Next, the assumption that $\Pi_{\mathcal{B}}$ has a density that is continuous and positive in a neighborhood of $\beta_{0}$ guarantees that there exists a positive sequence $c_{n,2}$ such that $c_{n,2} \rightarrow 0$ and
\begin{align*}
\left | \frac{\sup_{\beta \in B_{\mathcal{B}}(\beta_{0},M_{n}/\sqrt{n})}\pi_{\mathcal{B}}(\beta) }{\inf_{\beta \in B_{\mathcal{B}}(\beta_{0},M_{n}/\sqrt{n})}\pi_{\mathcal{B}}(\beta)} - 1 \right | \leq c_{n,2}.
\end{align*}
As a result, with $P_{0,YX|W}^{(n)}$-probability approaching one, the following quantities
\begin{align*}
    (1+c_{n,1})(1+c_{n,2})\frac{\int_{A \cap B_{\mathcal{B}}(\beta_{0},M_{n}/\sqrt{n})}\exp\left(\frac{1}{\sqrt{n}}\tilde{\ell}_{n}(\beta_{0},m_{0})\sqrt{n}(\beta-\beta_{0})-\frac{1}{2}\tilde{I}_{n}(m_{0})n(\beta-\beta_{0})^{2}\right) d \beta}{\int_{B_{\mathcal{B}}(\beta_{0},M_{n}/\sqrt{n})}\exp\left(\frac{1}{\sqrt{n}}\tilde{\ell}_{n}(\beta_{0},m_{0})\sqrt{n}(\beta-\beta_{0})-\frac{1}{2}\tilde{I}_{n}(m_{0})n(\beta-\beta_{0})^{2}\right) d \beta}
\end{align*}
and
\begin{align*}
   [(1+c_{n,1})(1+c_{n,2})]^{-1}\frac{\int_{A \cap B_{\mathcal{B}}(\beta_{0},M_{n}/\sqrt{n})}\exp\left(\tilde{\ell}_{n}(\beta_{0},m_{0})\sqrt{n}(\beta-\beta_{0})-\frac{1}{2}\tilde{I}_{n}(m_{0})n(\beta-\beta_{0})^{2}\right) d \beta}{\int_{B_{\mathcal{B}}(\beta_{0},M_{n}/\sqrt{n})}\exp\left(\tilde{\ell}_{n}(\beta_{0},m_{0})\sqrt{n}(\beta-\beta_{0})-\frac{1}{2}\tilde{I}_{n}(m_{0})n(\beta-\beta_{0})^{2}\right) d \beta}
\end{align*}
are upper and lower bounds, respectively, for $F_{n,2}(A)$, conditionally given $P_{0,W}^{\infty}$-almost every realization $\{w_{i}\}_{i \geq 1}$ of $\{W_{i}\}_{i \geq 1}$. Since $c_{n,1},c_{n,2} \rightarrow 0$ as $n\rightarrow \infty$, it follows that
\begin{align*}
\sup_{A }\left | F_{n,2}(A)- \frac{\int_{A \cap B_{\mathcal{B}}(\beta_{0},M_{n}/\sqrt{n})}\exp\left(\frac{1}{\sqrt{n}}\tilde{\ell}_{n}(\beta_{0},m_{0})\sqrt{n}(\beta-\beta_{0})-\frac{1}{2}\tilde{I}_{n}(m_{0})n(\beta-\beta_{0})^{2}\right) d \beta}{\int_{B_{\mathcal{B}}(\beta_{0},M_{n}/\sqrt{n})}\exp\left(\frac{1}{\sqrt{n}}\tilde{\ell}_{n}(\beta_{0},m_{0})\sqrt{n}(\beta-\beta_{0})-\frac{1}{2}\tilde{I}_{n}(m_{0})n(\beta-\beta_{0})^{2}\right) d \beta} \right| \overset{P_{0,YX|W}^{(n)}}{\longrightarrow} 0
\end{align*}
for $P_{0,W}^{\infty}$-almost every fixed realization $\{w_{i}\}_{i\geq 1}$ of $\{W_{i}\}_{i \geq 1}$ as $n\rightarrow \infty$.

\textit{Step 4.} Recalling that $\tilde{\Delta}_{n,0}= \tilde{I}_{n}(m_{0})^{-1}\tilde{\ell}_{n}(\beta_{0},m_{0})/\sqrt{n}$, I complete the square to obtain
\begin{align*}
    &\exp\left(\frac{1}{\sqrt{n}}\tilde{\ell}_{n}(\beta_{0},m_{0})\sqrt{n}(\beta-\beta_{0})-\frac{1}{2}\tilde{I}_{n}(m_{0})n(\beta-\beta_{0})^{2}\right)\\
    &\quad = \exp\left( -\frac{1}{2}\tilde{I}_{n}(m_{0})\left(\sqrt{n}(\beta-\beta_{0}) -\tilde{\Delta}_{n,0}\right)^{2} \right)\exp\left(\frac{1}{2}\tilde{I}_{n}(m_{0})\tilde{\Delta}_{n,0}^{2}\right) \\
    &=\exp\left( -\frac{1}{2}n\tilde{I}_{n}(m_{0})\left(\beta-\beta_{0} -\frac{\tilde{\Delta}_{n,0}}{\sqrt{n}}\right)^{2} \right)\exp\left(\frac{1}{2}\tilde{I}_{n}(m_{0})\tilde{\Delta}_{n,0}^{2}\right).
\end{align*}
It then follows that
\begin{align}\label{eq:gauss}
&\frac{\int_{A \cap B_{\mathcal{B}}(\beta_{0},M_{n}/\sqrt{n})}\exp\left(\frac{1}{\sqrt{n}}\tilde{\ell}_{n}(\beta_{0},m_{0})\sqrt{n}(\beta-\beta_{0})-\frac{1}{2}\tilde{I}_{n}(m_{0})n(\beta-\beta_{0})^{2}\right) d \beta}{\int_{B_{\mathcal{B}}(\beta_{0},M_{n}/\sqrt{n})}\exp\left(\frac{1}{\sqrt{n}}\tilde{\ell}_{n}(\beta_{0},m_{0})\sqrt{n}(\beta-\beta_{0})-\frac{1}{2}\tilde{I}_{n}(m_{0})n(\beta-\beta_{0})^{2}\right)d \beta} \nonumber \\
&\quad =\frac{\int_{A \cap B_{\mathcal{B}}(\beta_{0},M_{n}/\sqrt{n})}\exp\left(-\frac{1}{2}n\tilde{I}_{n}(m_{0})\left(\beta-\beta_{0} -\frac{\tilde{\Delta}_{n,0}}{\sqrt{n}}\right)^{2}\right) d \beta}{\int_{ B_{\mathcal{B}}(\beta_{0},M_{n}/\sqrt{n})}\exp\left(-\frac{1}{2}n\tilde{I}_{n}(m_{0})\left(\beta-\beta_{0} -\frac{\tilde{\Delta}_{n,0}}{\sqrt{n}}\right)^{2}\right) d \beta} \nonumber \\
&\quad =\frac{\int_{A \cap B_{\mathcal{B}}(\beta_{0},M_{n}/\sqrt{n})}\phi_{\beta_{0}+\frac{\tilde{\Delta}_{n,0}}{\sqrt{n}},\frac{1}{n}\tilde{I}_{n}(m_{0})^{-1}}(\beta)d\beta}{\int_{B_{\mathcal{B}}(\beta_{0},M_{n}/\sqrt{n})}\phi_{\beta_{0}+\frac{\tilde{\Delta}_{n,0}}{\sqrt{n}},\frac{1}{n}\tilde{I}_{n}(m_{0})^{-1}}(\beta)d\beta} 
\end{align}
for all events $A $, where $\phi_{\beta_{0}+\frac{\tilde{\Delta}_{n,0}}{\sqrt{n}},\frac{1}{n}\tilde{I}_{n}(m_{0})^{-1}}$ is the density of $\mathcal{N}(\beta_{0}+\tilde{\Delta}_{n,0}/\sqrt{n},(n\tilde{I}_{n}(m_{0}))^{-1})$. The denominator satisfies
\begin{align*}
    \int_{B_{\mathcal{B}}(\beta_{0},M_{n}/\sqrt{n})}\phi_{\beta_{0}+\frac{\tilde{\Delta}_{n,0}}{\sqrt{n}},\frac{1}{n}\tilde{I}_{n}(m_{0})^{-1}}(\beta)d\beta = \Phi\left(\frac{M_{n}-\tilde{\Delta}_{n,0}}{\sqrt{\tilde{I}_{n}(m_{0})^{-1}}}\right)-\Phi\left(\frac{-M_{n}-\tilde{\Delta}_{n,0}}{\sqrt{\tilde{I}_{n}(m_{0})^{-1}}}\right)
\end{align*}
where $\Phi(\cdot)$ is standard normal cumulative distribution function. Following Part 4 of Assumption \ref{as:dgp}, I can apply Chebyshev's inequality the Strong Law of Large Numbers to conclude that
\begin{align*}
 \tilde{I}_{n}(m_{0}) \overset{P_{0,YX|W}^{(n)}}{\longrightarrow} \tilde{I}_{0}(m_{0})
\end{align*}
as $n\rightarrow \infty$, where $\tilde{I}_{0}(m_{0}) = \sigma_{01}^{-2}E_{P_{0}}[(X-m_{02}(W))^{2}]$. An application of the Continuous Mapping Theorem (valid by Part 2 of Assumption \ref{as:dgp}) then leads to the conclusion that
\begin{align*}
    \tilde{I}_{n}^{-1}(m_{0}) \overset{P_{0,YX|W}^{(n)}}{\longrightarrow} \tilde{I}_{0}^{-1}(m_{0})
\end{align*}
for $P_{0,W}^{\infty}$-almost every fixed realization $\{w_{i}\}_{i\geq 1}$ of $\{W_{i}\}_{i \geq 1}$ as $n\rightarrow \infty$. Moreover, Chebyshev's inequality, Assumption \ref{as:dgp}, and the Strong Law of Large Numbers implies that 
\begin{align*}
\frac{1}{\sqrt{n}}\tilde{\ell}_{n}(\beta_{0},m_{0}) = O_{P_{0,YX|W}^{(n)}}(1)
\end{align*}
for $P_{0,W}^{\infty}$-almost every fixed realization $\{w_{i}\}_{i\geq 1}$ of $\{W_{i}\}_{i \geq 1}$ as $n\rightarrow \infty$, and, as a result, the Continuous Mapping Theorem implies that
\begin{align*}
    \tilde{\Delta}_{n,0} = O_{P_{0,YX|W}^{(n)}}(1)
\end{align*}
for $P_{0,W}^{\infty}$-almost every fixed realization $\{w_{i}\}_{i\geq 1}$ of $\{W_{i}\}_{i \geq 1}$ as $n\rightarrow \infty$. Since $M_{n}\rightarrow \infty$, it then follows that
\begin{align*}
        \int_{B_{\mathcal{B}}(\beta_{0},M_{n}/\sqrt{n})}\phi_{\beta_{0}+\frac{\tilde{\Delta}_{n,0}}{\sqrt{n}},\frac{1}{n}\tilde{I}_{n}(m_{0})^{-1}}(\beta)d\beta  = 1+o_{P_{0,YX|W}^{(n)}}(1).
\end{align*}
for $P_{0,W}^{\infty}$-almost every fixed realization $\{w_{i}\}_{i\geq 1}$ of $\{W_{i}\}_{i \geq 1}$ as $n\rightarrow \infty$. Consequently,
\begin{align*}
\sup_{A} \left |\frac{\int_{A \cap B_{\mathcal{B}}(\beta_{0},M_{n}/\sqrt{n})}\phi_{\beta_{0}+\frac{\tilde{\Delta}_{n,0}}{\sqrt{n}},\frac{1}{n}\tilde{I}_{n}(m_{0})^{-1}}(\beta)d\beta}{\int_{B_{\mathcal{B}}(\beta_{0},M_{n}/\sqrt{n})}\phi_{\beta_{0}+\frac{\tilde{\Delta}_{n,0}}{\sqrt{n}},\frac{1}{n}\tilde{I}_{n}(m_{0})^{-1}}(\beta)d\beta} -\int_{A \cap B_{\mathcal{B}}(\beta_{0},M_{n}/\sqrt{n})}\phi_{\beta_{0}+\frac{\tilde{\Delta}_{n,0}}{\sqrt{n}},\frac{1}{n}\tilde{I}_{n}(m_{0})^{-1}}(\beta)d\beta \right | \overset{P_{0,YX|W}^{(n)}}{\longrightarrow} 0    
\end{align*}
for $P_{0,W}^{\infty}$-almost every fixed realization $\{w_{i}\}_{i \geq 1}$ of $\{W_{i}\}_{i \geq 1}$ as $n\rightarrow \infty$. Since
\begin{align*}
    \int_{A \cap B_{\mathcal{B}}(\beta_{0},M_{n}/\sqrt{n})}\phi_{\beta_{0}+\frac{\tilde{\Delta}_{n,0}}{\sqrt{n}},\frac{1}{n}\tilde{I}_{n}(m_{0})^{-1}}(\beta)d\beta \leq  \int_{A }\phi_{\beta_{0}+\frac{\tilde{\Delta}_{n,0}}{\sqrt{n}},\frac{1}{n}\tilde{I}_{n}(m_{0})^{-1}}(\beta)d\beta
\end{align*}
for all events $A$, I need to show that
\begin{align}\label{eq:probabilityBVM}
 P_{0,YX|W}^{(n)}\left(    \int_{A \cap B_{\mathcal{B}}(\beta_{0},M_{n}/\sqrt{n})}\phi_{\beta_{0}+\frac{\tilde{\Delta}_{n,0}}{\sqrt{n}},\frac{1}{n}\tilde{I}_{n}(m_{0})^{-1}}(\beta)d\beta \geq  \int_{A }\phi_{\beta_{0}+\frac{\tilde{\Delta}_{n,0}}{\sqrt{n}},\frac{1}{n}\tilde{I}_{n}(m_{0})^{-1}}(\beta)d\beta \ \forall \ A \right) \longrightarrow 1
\end{align}
for $P_{0,W}^{\infty}$-almost every fixed realization $\{w_{i}\}_{i\geq 1}$ of $\{W_{i}\}_{i \geq 1}$ as $n\rightarrow \infty$. Indeed, this implies that
\begin{align*}
    \sup_{A}\left | \int_{A \cap B_{\mathcal{B}}(\beta_{0},M_{n}/\sqrt{n})}\phi_{\beta_{0}+\frac{\tilde{\Delta}_{n,0}}{\sqrt{n}},\frac{1}{n}\tilde{I}_{n}(m_{0})^{-1}}(\beta)d\beta- \int_{A }\phi_{\beta_{0}+\frac{\tilde{\Delta}_{n,0}}{\sqrt{n}},\frac{1}{n}\tilde{I}_{n}(m_{0})^{-1}}(\beta)d\beta\right | \overset{P_{0,YX|W}^{(n)}}{\longrightarrow} 0, 
\end{align*}
for $P_{0,W}^{\infty}$-almost every fixed realization $\{w_{i}\}_{i\geq 1}$ of $\{W_{i}\}_{i \geq 1}$ as $n\rightarrow \infty$, and, as a result, an application of the triangle inequality leads to the conclusion that
\begin{align*}
    \left | \left | \Pi(\beta \in A|\{(Y_{i},X_{i},W_{i}')'\}_{i=1}^{n}) - \mathcal{N}\left(\beta_{0}+\frac{\tilde{\Delta}_{n,0}}{\sqrt{n}},\frac{1}{n}\tilde{I}_{n}(m_{0})^{-1}\right) \right| \right |_{TV} \overset{P_{0,YX|W}^{(n)}}{\longrightarrow} 0
\end{align*}
for $P_{0,W}^{\infty}$-almost every fixed realization $\{w_{i}\}_{i\geq 1}$ of $\{W_{i}\}_{i \geq 1}$ as $n\rightarrow \infty$. For (\ref{eq:probabilityBVM}), the events $A \cap B_{\mathcal{B}}(\beta_{0},M_{n}/\sqrt{n})$ and $A \cap B_{\mathcal{B}}(\beta_{0},M_{n}/\sqrt{n})^{c}$ define a partition of $A$ meaning that
\begin{align*}
    &\int_{A \cap B_{\mathcal{B}}(\beta_{0},M_{n}/\sqrt{n})}\phi_{\beta_{0}+\frac{\tilde{\Delta}_{n,0}}{\sqrt{n}},\frac{1}{n}\tilde{I}_{n}(m_{0})^{-1}}(\beta)d\beta\\
    &\quad =\int_{A }\phi_{\beta_{0}+\frac{\tilde{\Delta}_{n,0}}{\sqrt{n}},\frac{1}{n}\tilde{I}_{n}(m_{0})^{-1}}(\beta)d\beta -  \int_{A 
 \cap B_{\mathcal{B}}(\beta_{0},M_{n}/\sqrt{n})^{c}}\phi_{\beta_{0}+\frac{\tilde{\Delta}_{n,0}}{\sqrt{n}},\frac{1}{n}\tilde{I}_{n}(m_{0})^{-1}}(\beta)d\beta \\
 &\quad \geq \int_{A }\phi_{\beta_{0}+\frac{\tilde{\Delta}_{n,0}}{\sqrt{n}},\frac{1}{n}\tilde{I}_{n}(m_{0})^{-1}}(\beta)d\beta -  \int_{B_{\mathcal{B}}(\beta_{0},M_{n}/\sqrt{n})^{c}}\phi_{\beta_{0}+\frac{\tilde{\Delta}_{n,0}}{\sqrt{n}},\frac{1}{n}\tilde{I}_{n}(m_{0})^{-1}}(\beta)d\beta
\end{align*}
for all events $A$, where the lower bound holds because
\begin{align*}
    \int_{A \cap B_{\mathcal{B}}(\beta_{0},M_{n}/\sqrt{n})^{c}}\phi_{\beta_{0}+\frac{\tilde{\Delta}_{n,0}}{\sqrt{n}},\frac{1}{n}\tilde{I}_{n}(m_{0})^{-1}}(\beta)d\beta \leq \int_{B_{\mathcal{B}}(\beta_{0},M_{n}/\sqrt{n})^{c}}\phi_{\beta_{0}+\frac{\tilde{\Delta}_{n,0}}{\sqrt{n}},\frac{1}{n}\tilde{I}_{n}(m_{0})^{-1}}(\beta)d\beta
\end{align*}
for all events $A$ using the monotonicity of probability measures. An earlier argument showed that
\begin{align*}
    \int_{B_{\mathcal{B}}(\beta_{0},M_{n}/\sqrt{n})^{c}}\phi_{\beta_{0}+\frac{\tilde{\Delta}_{n,0}}{\sqrt{n}},\frac{1}{n}\tilde{I}_{n}(m_{0})^{-1}}(\beta)d\beta = 1+o_{P_{0,YX|W}^{(n)}}(1),
\end{align*}
for $P_{0,W}^{\infty}$-almost every fixed realization $\{w_{i}\}_{i\geq 1}$ of $\{W_{i}\}_{i \geq 1}$ as $n\rightarrow \infty$, so the complement rule for probability measures implies that
\begin{align*}
    \int_{B_{\mathcal{B}}(\beta_{0},M_{n}/\sqrt{n})^{c}}\phi_{\beta_{0}+\frac{\tilde{\Delta}_{n,0}}{\sqrt{n}},\frac{1}{n}\tilde{I}_{n}(m_{0})^{-1}}(\beta)d\beta = o_{P_{0,YX|W}^{(n)}}(1)
\end{align*}
as $n\rightarrow \infty$. As such, I conclude that (\ref{eq:probabilityBVM}) holds and complete the proof.
\end{proof}
\subsection{Proof of Corollaries \ref{cor:BVM_unconditional} and \ref{cor:quantile}}
\begin{proof}[Proof of Corollary \ref{cor:BVM_unconditional}]
Theorem \ref{thm:BVM} implies that for every $\zeta> 0$, $P_{0,YX|W}^{(n)}(A_{n}(\zeta)) \rightarrow 0$ for $P_{0,W}^{\infty}$-almost every fixed realization $\{w_{i}\}_{i \geq 1}$ of $\{W_{i}\}_{i \geq 1}$ as $n\rightarrow \infty$, where
\begin{align*}
A_{n}(\zeta) = \left\{\left | \left | \Pi(\beta \in A|\{(Y_{i},X_{i},w_{i}')'\}_{i=1}^{n}) - \mathcal{N}\left(\beta_{0}+\frac{\tilde{\Delta}_{n,0}}{\sqrt{n}},\frac{1}{n}\tilde{I}_{n}(m_{0})^{-1}\right) \right| \right |_{TV} \geq \zeta \right\}.
\end{align*}
Since $\{w_{i}\}_{i \geq 1} \mapsto P_{0,YX|W}^{(n)}(A_{n}(\xi)) $ is bounded, the bounded convergence theorem implies that, for each $\zeta>0$, $E_{P_{0,W}^{\infty}}[P_{0,YX|W}^{(n)}(A_{n}(\zeta))] \rightarrow 0$ as $n\rightarrow \infty$. Since i.i.d sampling and the law of iterated expectations implies $E_{P_{0,W}^{\infty}}[P_{0,YX|W}^{(n)}(A_{n}(\zeta))] = P^{n}_{0}(A_{n}(\zeta))$, I am able to conclude the claim of Corollary \ref{cor:BVM_unconditional}.
\end{proof}
\begin{proof}[Proof of Corollary \ref{cor:quantile}]
Let $F_{\mathcal{N}(0,\tilde{I}_{n}^{-1}(m_{0}))}(\cdot)$ be the cumulative distribution function of $\mathcal{N}(0,\tilde{I}_{n}^{-1}(m_{0}))$ and let $\hat{\beta}_{n} = \beta_{0} + \tilde{\Delta}_{n,0}/\sqrt{n}$. Theorem \ref{thm:BVM} implies that
\begin{align*}
    \sup_{b \in \mathbb{R}}\left|\Pi(\sqrt{n}(\beta-\hat{\beta}_{n}) \leq b|\{(Y_{i},X_{i},w_{i}')'\}_{i=1}^{n})- F_{\mathcal{N}(0,\tilde{I}_{n}^{-1}(m_{0}))}( b)\right| \overset{P_{0,YX|W}^{(n)}}{\longrightarrow} 0
\end{align*}
for $P_{0,W}^{\infty}$-almost every fixed realization $\{w_{i}\}_{i \geq 1}$ of $\{W_{i}\}_{i \geq 1}$ as $n\rightarrow \infty$. Since the set of continuity points of the Gaussian distribution $\mathcal{N}(0,\tilde{I}_{n}^{-1}(m_{0}))$ is $\mathbb{R}$, I can apply Lemma 21.2 of \cite{van1998asymptotic} and the equivariance of quantiles under monotonic transformations to conclude that
\begin{align*}
    c_{n}(q) = \hat{\beta}_{n} + \Phi^{-1}\left(q\right)\sqrt{\frac{\tilde{I}_{n}^{-1}(m_{0})}{n}} + o_{P_{0,YX|W}^{(n)}}\left(\frac{1}{\sqrt{n}}\right)
\end{align*}
holds for $P_{0,W}^{\infty}$-almost every fixed realization $\{w_{i}\}_{i \geq 1}$ of $\{W_{i}\}_{i \geq 1}$ as $n\rightarrow \infty$.
\end{proof}
\subsection{Proofs of Lemmas \ref{lem:technical} and \ref{lem:rootn}}\label{ap:lemmas}
\begin{proof}[Proof of Lemma \ref{lem:technical}]
The proof has four steps with each step corresponding to the appropriate part of the lemma (i.e., Step 1 verifies Part 1 of the lemma and so on).

\textit{Step 1.} Let $\tilde{\ell}(Y,X,W,\beta,m) = (Y-m_{1}(W)-(X-m_{2}(W))\beta)(X-m_{2}(W))/\sigma_{01}^{2}$ for each $(\beta,m)$. Since $Y_{i} = m_{01}(w_{i})+(X_{i}-m_{02}(w_{i}))\beta_{0}+U_{i}$ and $X_{i} = m_{02}(w_{i})+\varepsilon_{i2}$, I know that
\begin{align*}
\tilde{\ell}(Y_{i},X_{i},w_{i},\beta_{0},m) &= \frac{1}{\sigma_{01}^{2}}\left(Y_{i}-m_{1}(w_{i})-(X_{i}-m_{2}(w_{i}))\beta_{0}\right)\left(X_{i}-m_{2}(w_{i})\right) \\
&=\frac{1}{\sigma_{01}^{2}}\left(m_{01}(w_{i})-m_{1}(w_{i})-(m_{02}(w_{i})-m_{2}(w_{i}))\beta_{0}+U_{i}\right)\left(m_{02}(w_{i})-m_{2}(w_{i})+\varepsilon_{i2}\right) \\
&=\frac{1}{\sigma_{01}^{2}}\left\{(m_{01}(w_{i})-m_{1}(w_{i}))(m_{02}(w_{i})-m_{2}(w_{i}))-(m_{02}(w_{i})-m_{2}(w_{i}))^{2}\beta_{0} \right. \\
&\quad \quad \quad  \left. +(m_{02}(w_{i})-m_{2}(w_{i}))(U_{i}-\beta_{0}\varepsilon_{i2})+(m_{01}(w_{i})-m_{1}(w_{i}))\varepsilon_{i2} + U_{i}\varepsilon_{i2}\right\}.
\end{align*}
Observing that $\tilde{\ell}(Y_{i},X_{i},w_{i},\beta_{0},m_{0}) = U_{i}\varepsilon_{i2}/\sigma_{01}^{2}$, the previous derivation implies that
\begin{align*}
    &\frac{\sigma_{01}^{2}}{\sqrt{n}}\left(\tilde{\ell}_{n}(\beta_{0},m)- \tilde{\ell}_{n}(\beta_{0},m_{0})\right)\\
    &\quad = \sqrt{n}\langle m_{1}-m_{01},m_{2}-m_{02} \rangle_{n,2} - \sqrt{n}||m_{2}-m_{02}||_{n,2}^{2}\beta_{0} -G_{n}^{(1)}(m) - G_{n}^{(2)}(m),
\end{align*}
where I recall that the empirical processes $\{G_{n}^{(1)}\}_{n=1}^{\infty}$ and $\{G_{n}^{(2)}\}_{n=1}^{\infty}$ are given by
\begin{align*}
    G_{n}^{(1)}(m) = \frac{1}{\sqrt{n}}\sum_{i=1}^{n}(m_{1}(w_{i})-m_{01}(w_{i}))\varepsilon_{i2}
\end{align*}
and
\begin{align*}
    G_{n}^{(2)}(m) = \frac{1}{\sqrt{n}}\sum_{i=1}^{n}\left\{\left(m_{2}(w_{i})-m_{02}(w_{i})\right)\left(U_{i}-\beta_{0}\varepsilon_{i2}\right)\right\}
\end{align*}
for $m \in \mathcal{M}$. Applying Assumption \ref{as:emp_process}, I know that
\begin{align*}
    \sup_{m \in \mathcal{M}_{n} \cap B_{n,\mathcal{M}}(m_{0},D\delta_{n})}|G_{n}^{(1)}(m)| = o_{P_{0,YX|W}^{(n)}}(1)
\end{align*}
and
\begin{align*}
       \sup_{m \in \mathcal{M}_{n} \cap B_{n,\mathcal{M}}(m_{0},D\delta_{n})}|G_{n}^{(2)}(m)| = o_{P_{0,YX|W}^{(n)}}(1)
\end{align*}
for $P_{0,W}^{\infty}$-almost every fixed realization $\{w_{i}\}_{i \geq 1}$ of $\{W_{i}\}_{i \geq 1}$. Moreover, the Cauchy-Schwarz inequality and the definition of $B_{n,\mathcal{M}}(m_{0},D\delta_{n})$ reveals that
\begin{align*}
&\sup_{m \in \mathcal{M}_{n} \cap B_{n,\mathcal{M}}(m_{0},D\delta_{n})}\left|\sqrt{n}\langle m_{1}-m_{01},m_{2}-m_{02} \rangle_{n,2} \right| \lesssim \sqrt{n}\delta_{n}^{2}
\end{align*}
and
\begin{align*}
\sup_{m \in\mathcal{M}_{n} \cap B_{n,\mathcal{M}}(m_{0},D\delta_{n})}\sqrt{n}||m_{2}-m_{02}||_{n,2}^{2}\lesssim \sqrt{n}\delta_{n}^{2}.    
\end{align*}
Using Assumption \ref{as:consistency}, I know that $\sqrt{n}\delta_{n}^{2} = o(1)$. Consequently, when combined with the arguments above, I am able to conclude that 
\begin{align*}
\sup_{m \in \mathcal{M}_{n} \cap B_{n,\mathcal{M}}(m_{0},D\delta_{n})}\left|\frac{1}{\sqrt{n}}\tilde{\ell}_{n}(\beta_{0},m)-\frac{1}{\sqrt{n}}\tilde{\ell}_{n}(\beta_{0},m_{0}) \right| \overset{P_{0,YX|W}^{(n)}}{\longrightarrow} 0 
\end{align*}
for $P_{0,W}^{\infty}$-almost every fixed realization $\{w_{i}\}_{i \geq 1}$ of $\{W_{i}\}_{i \geq 1}$ as $n\rightarrow \infty$.

\textit{Step 2.} Since $X_{i} = m_{02}(w_{i})+\varepsilon_{i2}$, it follows that
\begin{align*}
    (X_{i}-m_{2}(w_{i}))^{2} = \left(m_{2}(w_{i})-m_{02}(w_{i})\right)^{2} - 2(m_{2}(w_{i})-m_{02}(w_{i}))\varepsilon_{i2} + \varepsilon_{i2}^{2},
\end{align*}
and, as a result, I obtain
\begin{align*}
\tilde{I}_{n}(m) - \tilde{I}_{n}(m_{0}) = \frac{1}{\sigma_{01}^{2}}||m_{2}-m_{02}||_{n,2}^{2} - \frac{2}{\sigma_{01}^{2}}\frac{1}{n}\sum_{i=1}^{n}(m_{2}(w_{i})-m_{02}(w_{i}))\varepsilon_{i2}.
\end{align*}
Applying Assumption \ref{as:emp_process}, I know that
\begin{align*}
    \sup_{m \in \mathcal{M}_{n} \cap B_{n,\mathcal{M}}(m_{0},D,\delta_{n})}\left|\frac{1}{n}\sum_{i=1}^{n}(m_{2}(w_{i})-m_{02}(w_{i}))\varepsilon_{i2}\right| \overset{P_{0,YX|W}^{(n)}}{\longrightarrow} 0
\end{align*}
for $P_{0,W}^{\infty}$-almost every realization $\{w_{i}\}_{i \geq 1}$ of $\{W_{i}\}_{i \geq 1}$ as $n\rightarrow \infty$. Moreover, using the definition of $B_{n,\mathcal{M}}(m_{0},D\delta_{n})$ and Assumption \ref{as:consistency}, I know that
\begin{align*}
\sup_{m \in \mathcal{M}_{n} \cap B_{n,\mathcal{M}}(m_{0},D\delta_{n})}||m_{2}-m_{02}||_{n,2}^{2}  \lesssim \delta_{n}^{2} = o(1)   
\end{align*}
as $n\rightarrow \infty$. Consequently, I conclude that
 \begin{align*}
\sup_{m \in\mathcal{M}_{n} \cap B_{n,\mathcal{M}}(m_{0},D\delta_{n})} \left|\tilde{I}_{n}(m) - \tilde{I}_{n}(m_{0})\right| \overset{P_{0,YX|W}^{(n)}}{\longrightarrow} 0
\end{align*}
for $P_{0,W}^{\infty}$-almost every realization $\{w_{i}\}_{i \geq 1}$ of $\{W_{i}\}_{i \geq 1}$ as $n\rightarrow \infty$.

\textit{Step 3.} I add and subtract $\tilde{I}_{n}(m_{0})$ and apply the law of total probability to conclude that
\begin{align*}
    &P_{0,YX|W}^{(n)}\left(\inf_{m \in \mathcal{M}_{n} \cap B_{\mathcal{M},\infty}(m_{0},D\delta_{n})}\tilde{I}_{n}(m) < \underline{c}_{0}\right) \\
    &\quad= P_{0,YX|W}^{(n)}\left(\inf_{m \in \mathcal{M}_{n} \cap B_{\mathcal{M},\infty}(m_{0},D\delta_{n})}\tilde{I}_{n}(m) < \underline{c}_{0}, \ \tilde{I}_{n}(m_{0})\geq \overline{c}_{0}\right) \\
    &\quad \quad + P_{0,YX|W}^{(n)}\left(\inf_{m \in \mathcal{M}_{n} \cap B_{\mathcal{M},\infty}(m_{0},D\delta_{n})}\tilde{I}_{n}(m) < \underline{c}_{0}, \ \tilde{I}_{n}(m_{0})< \overline{c}_{0}\right) \\
    &\quad \leq   P_{0,YX|W}^{(n)}\left(\inf_{m \in \mathcal{M}_{n} \cap B_{\mathcal{M},\infty}(m_{0},D\delta_{n})}(\tilde{I}_{n}(m)-\tilde{I}_{n}(m_{0})) < 0\right) +  P_{0,YX|W}^{(n)}\left(\tilde{I}_{n}(m_{0}) < \underline{c}_{0}\right) 
\end{align*}
Applying Chebyshev's inequality and the conditional independence, I know that, for $P_{0,W}^{\infty}$-almost every fixed realization $\{w_{i}\}_{i \geq 1}$ of $\{W_{i}\}_{i \geq 1}$,
\begin{align*}
P_{0,YX|W}^{(n)}\left(\tilde{I}_{n}(m_{0}) \geq \underline{c}_{0}\right) \lesssim \frac{1}{n^{2}}\sum_{i=1}^{n}E_{P_{0,YX|W}}[(X-m_{02}(W))^{4}|W=w_{i}] = o(1),
\end{align*}
where the second equality applies the Strong Law of Large Numbers to the sample average
\begin{align*}
\frac{1}{n}\sum_{i=1}^{n}E_{P_{0,YX|W}}[(X-m_{02}(W))^{4}|W=w_{i}],    
\end{align*}
which is valid because of Part 4 of Assumption \ref{as:dgp} and the assumption that $\{W_{i}\}_{i \geq 1}$ is an i.i.d sequence. Applying Part 2 of the lemma, I deduce that
\begin{align*}
    P_{0,YX|W}^{(n)}\left(\inf_{m \in \mathcal{M}_{n} \cap B_{\mathcal{M},\infty}(m_{0},D\delta_{n})}(\tilde{I}_{n}(m)-\tilde{I}_{n}(m_{0})) < 0\right) \longrightarrow 0
\end{align*}
for $P_{0,W}^{\infty}$-almost every fixed realization $\{w_{i}\}_{i \geq 1}$ of $\{W_{i}\}_{i \geq 1}$ as $n\rightarrow \infty$. It then follows that
\begin{align*}
     P_{0,YX|W}^{(n)}\left(\inf_{m \in \mathcal{M}_{n} \cap B_{\mathcal{M},\infty}(m_{0},D\delta_{n})}\tilde{I}_{n}(m) < \underline{c}_{0}\right) \longrightarrow 0
\end{align*}
for $P_{0,W}^{\infty}$-almost every fixed realization $\{w_{i}\}_{i \geq 1}$ of $\{W_{i}\}_{i \geq 1}$ as $n\rightarrow \infty$. A symmetric argument establishes
\begin{align*}
 P_{0,YX|W}^{(n)}\left(\sup_{m \in \mathcal{M}_{n} \cap B_{\mathcal{M},\infty}(m_{0},D\delta_{n})}\tilde{I}_{n}(m) > \overline{c}_{0}\right) \longrightarrow 0 
\end{align*}
for $P_{0,W}^{\infty}$-almost every fixed realization $\{w_{i}\}_{i \geq 1}$ of $\{W_{i}\}_{i \geq 1}$ as $n\rightarrow \infty$.

\textit{Step 4.} Since $p_{\beta,m}(Y_{i},X_{i}|w_{i})$ is the density of $\mathcal{N}(m(w_{i}),V(\beta))$, I obtain the following expansion of the log-likelihood function around the true parameter $\beta_{0}$:
\begin{align*}
  &\ell_{n}(\beta,m) - \ell_{n}(\beta_{0},m) = \frac{1}{\sqrt{n}}\tilde{\ell}_{n}(\beta_{0},m)\sqrt{n}(\beta-\beta_{0}) - \frac{1}{2}n(\beta-\beta_{0})^{2}\tilde{I}_{n}(m),
    \end{align*}
and, as a result, I apply the triangle inequality to conclude
\begin{align*}
&\left | \ell_{n}(\beta,m) - \ell_{n}(\beta_{0},m) -  \frac{1}{\sqrt{n}}\tilde{\ell}_{n}(\beta_{0},m_{0})\sqrt{n}(\beta-\beta_{0}) + \frac{1}{2}n(\beta-\beta_{0})^{2}\tilde{I}_{n}(m_{0})  \right| \\
&\quad \leq \left | \frac{1}{\sqrt{n}}\tilde{\ell}_{n}(\beta_{0},m)-\frac{1}{\sqrt{n}}\tilde{\ell}_{n}(\beta_{0},m_{0}) \right| \cdot \left | \sqrt{n}(\beta-\beta_{0})\right | + \frac{n}{2}(\beta-\beta_{0})^{2}\left | \tilde{I}_{n}(m)-\tilde{I}_{n}(m_{0}) \right| \\
&\quad \leq  \left | \frac{1}{\sqrt{n}}\tilde{\ell}_{n}(\beta_{0},m)-\frac{1}{\sqrt{n}}\tilde{\ell}_{n}(\beta_{0},m_{0}) \right| \cdot M_{n} + \frac{M_{n}^{2}}{2}\left | \tilde{I}_{n}(m)-\tilde{I}_{n}(m_{0}) \right|
\end{align*}
Since $M_{n}\rightarrow \infty$ arbitrarily slowly as $n\rightarrow \infty$, it follows from Part 1 and Part 2 of the lemma that
\begin{align*}
\sup_{m \in \mathcal{M}_{n} \cap B_{n,\mathcal{M}}(m_{0},D\delta_{n})}\left | \frac{1}{\sqrt{n}}\tilde{\ell}_{n}(\beta_{0},m)-\frac{1}{\sqrt{n}}\tilde{\ell}_{n}(\beta_{0},m_{0}) \right| \cdot M_{n} \overset{P_{0,YX|W}^{(n)}}{\longrightarrow} 0
\end{align*}
and
\begin{align*}
\frac{M_{n}^{2}}{2}\sup_{m \in \mathcal{M}_{n} \cap B_{n,\mathcal{M}}(m_{0},D\delta_{n})}\left | \tilde{I}_{n}(m)-\tilde{I}_{n}(m_{0}) \right| \overset{P_{0,YX|W}^{(n)}}{\longrightarrow} 0 
\end{align*}
conditionally given $P_{0,W}^{\infty}$ almost every realization $\{w_{i}\}_{i \geq 1}$ of $\{W_{i}\}_{i \geq 1}$ as $n\rightarrow \infty$. This concludes the proof because it implies that
\begin{align*}
\sup_{(\beta,m) \in B_{\mathcal{B}}(\beta_{0},M_{n}/\sqrt{n})\times (B_{\mathcal{M},\infty}(m_{0},D\delta_{n})\cap \mathcal{M}_{n})}\left | R_{n}(\beta,m) \right| \overset{P_{0,YX|W}^{(n)}}{\longrightarrow} 0
\end{align*}
for $P_{0,W}^{\infty}$-almost every realization $\{w_{i}\}_{i \geq 1}$ of $\{W_{i}\}_{i \geq 1}$ as $n\rightarrow \infty$.
\end{proof}
\begin{proof}[Proof of Lemma \ref{lem:rootn}] The proof is comprised of three steps. Step 1 obtains an expression for the posterior based on a quadratic expansion of the log-likelihood function. Step 2 shows that the numerator is uniformly bounded from above by $\exp(-(\underline{c}_{0}/4) M_{n}^{2})$ with $P_{0,YX|W}^{(n)}$-probability approaching $1$, where $\underline{c}_{0}>0$ is the constant defined in Part 3 of Lemma \ref{lem:technical}. Step 3 shows that the denominator is uniformly bounded from below by $\exp(-(\underline{c}_{0}/8)M_{n}^{2})$ with $P_{0,YX|W}^{(n)}$-probability approaching $1$. The arguments for Step 2 and Step 3 are similar to those encountered on pp. 233-234 and the proof of Lemma 6.3, respectively, in \cite{bickel2012semiparametric} with appropriate modifications.

\textit{Step 1.} I start with the definition of the posterior:
\begin{align*}
&\Pi\left(\beta \in B_{\mathcal{B}}(\beta_{0},M_{n}/\sqrt{n})^{c}|\{(Y_{i},X_{i},w_{i})\}_{i=1}^{n},m\right) \\
&\quad =  \frac{\int_{B_{\mathcal{B}}(\beta_{0},M_{n}/\sqrt{n})^{c}}\exp\left(\ell_{n}(\beta,m) - \ell_{n}(\beta_{0},m)\right)d\Pi_{\mathcal{B}}(\beta)}{\int_{\mathcal{B}}\exp\left(\ell_{n}(\beta,m) - \ell_{n}(\beta_{0},m)\right)d\Pi_{\mathcal{B}}(\beta)}.   
\end{align*}
As $p_{\beta,m}(Y_{i},X_{i}|w_{i})$ is the density of $\mathcal{N}(m(w_{i}),V(\beta))$, I obtain the following expansion of the log-likelihood function around the true parameter $\beta_{0}$
\begin{align*}
    &\ell_{n}(\beta,m) - \ell_{n}(\beta_{0},m)  = \frac{1}{\sqrt{n}}\tilde{\ell}_{n}(\beta_{0},m)\sqrt{n}(\beta-\beta_{0}) - \frac{1}{2}n(\beta-\beta_{0})^{2}\tilde{I}_{n}(m).
\end{align*}
Consequently,
\begin{align*}
    &\Pi\left(\beta \in B_{\mathcal{B}}(\beta_{0},M_{n}/\sqrt{n})^{c}|\{(Y_{i},X_{i},w_{i}')'\}_{i=1}^{n},m\right)\\
    &=\frac{\int_{B_{\mathcal{B}}(\beta_{0},M_{n}/\sqrt{n})^{c}}\exp\left(\frac{1}{\sqrt{n}}\tilde{\ell}_{n}(\beta_{0},m)\sqrt{n}(\beta-\beta_{0}) - \frac{1}{2}n(\beta-\beta_{0})^{2}\tilde{I}_{n}(m)\right)d\Pi_{\mathcal{B}}(\beta)}{\int_{\mathcal{B}}\exp\left(\frac{1}{\sqrt{n}}\tilde{\ell}_{n}(\beta_{0},m)\sqrt{n}(\beta-\beta_{0}) - \frac{1}{2}n(\beta-\beta_{0})^{2}\tilde{I}_{n}(m)\right)d\Pi_{\mathcal{B}}(\beta)}.
\end{align*}

\textit{Step 2.} Let $A_{n,1}$ be the event that
\begin{align*}
    \sup_{m \in B_{n,\mathcal{M}}(m_{0},D\delta_{n}) \cap \mathcal{M}_{n}}\int_{B_{\mathcal{B}}(\beta_{0},M_{n}/\sqrt{n})^{c}}\exp\left(\frac{1}{\sqrt{n}}\tilde{\ell}_{n}(\beta_{0},m)\sqrt{n}(\beta-\beta_{0}) - \frac{1}{2}n(\beta-\beta_{0})^{2}\tilde{I}_{n}(m)\right)d\Pi_{\mathcal{B}}(\beta) \leq \exp\left(-\frac{\underline{c}_{0}M_{n}^{2}}{4}\right)
    \end{align*}
The first task is to show that $P_{0,YX|W}^{(n)}(A_{n,1}^{c}) \rightarrow 0$ for $P_{0,W}^{\infty}$-almost every fixed realization $\{w_{i}\}_{i \geq 1}$ of $\{W_{i}\}_{i \geq 1}$ as $n\rightarrow \infty$. Using the definition of the supremum and that probabilities are bounded from above by $1$, I obtain the following string of inequalities
\begin{align}
&\int_{B_{\mathcal{B}}(\beta_{0},M_{n}/\sqrt{n})^{c}}\exp\left(\frac{1}{\sqrt{n}}\tilde{\ell}_{n}(\beta_{0},m)\sqrt{n}(\beta-\beta_{0}) - \frac{1}{2}n(\beta-\beta_{0})^{2}\tilde{I}_{n}(m)\right)d\Pi_{\mathcal{B}}(\beta) \nonumber \\
&\quad \leq \sup_{\beta \in B_{\mathcal{B}}(\beta_{0},M_{n}/\sqrt{n})^{c}}\exp\left(\frac{1}{\sqrt{n}}\tilde{\ell}_{n}(\beta_{0},m)\sqrt{n}(\beta-\beta_{0}) - \frac{1}{2}n(\beta-\beta_{0})^{2}\tilde{I}_{n}(m)\right) \nonumber \\
&\quad \leq \exp\left(\sup_{\beta \in B_{\mathcal{B}}(\beta_{0},M_{n}/\sqrt{n})^{c}}\left(\frac{1}{\sqrt{n}}\tilde{\ell}_{n}(\beta_{0},m)\sqrt{n}(\beta-\beta_{0}) - \frac{1}{2}n(\beta-\beta_{0})^{2}\tilde{I}_{n}(m)\right)\right) \nonumber \\
&\quad \leq \exp\left(\sup_{\beta \in B_{\mathcal{B}}(\beta_{0},M_{n}/\sqrt{n})^{c}}\left(\left|\frac{1}{\sqrt{n}}\tilde{\ell}_{n}(\beta_{0},m)\right|\cdot|\sqrt{n}(\beta-\beta_{0})| - \frac{1}{2}n(\beta-\beta_{0})^{2}\tilde{I}_{n}(m)\right)\right) \label{eq:control0}
\end{align}
for all $m \in \mathcal{M}_{n} \cap B_{n,\mathcal{M}}(m_{0},D\delta_{n})$. From Part 1 of Lemma \ref{lem:technical}, I know that 
\begin{align*}
\sup_{m \in \mathcal{M}_{n} \cap B_{n,\mathcal{M}}(m_{0},D\delta_{n})} \left |\frac{1}{\sqrt{n}}\tilde{\ell}_{n}(\beta_{0},m)-\frac{1}{\sqrt{n}}\tilde{\ell}_{n}(\beta_{0},m_{0})\right | \overset{P_{0,XY|W}^{(n)}}{\longrightarrow} 0   
\end{align*}
for $P_{0,W}^{\infty}$-almost every fixed realization $\{w_{i}\}_{i \geq 1}$ of $\{W_{i}\}_{i \geq 1}$ as $n\rightarrow \infty$. Moreover, an application of Chebyshev's inequality and the Strong Law of Large Numbers implies that  
\begin{align*}
 \frac{1}{\sqrt{n}}\tilde{\ell}_{n}(\beta_{0},m_{0})   = O_{P_{0,YX|W}^{(n)}}(1)
\end{align*}
for $P_{0,W}^{\infty}$-almost every fixed realization $\{w_{i}\}_{i \geq 1}$ of $\{W_{i}\}_{i \geq 1}$ as $n\rightarrow \infty$. Indeed, for any $\tau> 0$, independence across $i$, Chebyshev's inequality, Part 1 of Assumption \ref{as:dgp}, and the Law of Iterated Expectations implies that
\begin{align*}
 P_{0,YX|W}^{(n)}\left(\left| \frac{1}{\sqrt{n}}\tilde{\ell}_{n}(\beta_{0},m_{0})   \right| \geq \tau \right) \leq \frac{1}{\tau^{2}}\frac{1}{n}\sum_{i=1}^{n}E[U^{2}\varepsilon_{2}^{2}|W=w_{i}] = \frac{\sigma_{01}^{2}}{\tau^{2}}\frac{1}{n}\sum_{i=1}^{n}Var(\varepsilon_{2}|W=w_{i}),
\end{align*}
which implies 
\begin{align*}
\frac{1}{\sqrt{n}}\tilde{\ell}_{n}(\beta_{0},m_{0})   = O_{P_{0,YX|W}^{(n)}}(1)    
\end{align*}
for $P_{0,W}^{\infty}$-almost every fixed realization $\{w_{i}\}_{i \geq 1}$ of $\{W_{i}\}_{i \geq 1}$ as $n\rightarrow \infty$ after applying the Strong Law of Large Numbers to the sample average $n^{-1}\sum_{i=1}^{n}Var(\varepsilon_{2}|W=w_{i})$, which is valid by Part 2 of Assumption \ref{as:dgp} and the fact that $\{W_{i}\}_{i \geq 1}$ is i.i.d. Consequently, I am able to conclude that
\begin{align}\label{eq:control1}
    P_{0,YX|W}^{(n)}\left(\sup_{m \in B_{\mathcal{M},\infty}(m_{0},D\delta_{n})\cap \mathcal{M}_{n}}\left|\frac{1}{\sqrt{n}}\tilde{\ell}_{n}(\beta_{0},m)\right| < \frac{\underline{c}_{0}}{4}M_{n}\right) \longrightarrow 1.
\end{align}
for $P_{0,W}^{\infty}$-almost every fixed realization $\{w_{i}\}_{i \geq 1}$ of $\{W_{i}\}_{i \geq 1}$ as $n\rightarrow \infty$. Further, Part 3 of Lemma \ref{lem:technical} implies that
\begin{align}\label{eq:control2}
    P_{0,YX|W}^{(n)}\left(\inf_{m \in \mathcal{M}_{n} \cap B_{\mathcal{M},\infty}(m_{0},D\delta_{n})}\tilde{I}_{n}(m) \geq \underline{c}_{0}  \right) \longrightarrow 1
\end{align}
for $P_{0,W}^{\infty}$-almost every fixed realization $\{w_{i}\}_{i \geq 1}$ of $\{W_{i}\}_{i \geq 1}$ as $n\rightarrow \infty$. Hence, I can combine (\ref{eq:control0}), (\ref{eq:control1}), and (\ref{eq:control2}) to conclude that 
\begin{align*}
 P_{0,YX|W}^{(n)}(A_{n,1}^{c}) \leq  \mathbf{1}\left\{\underline{c}_{0}\sup_{\beta \in B_{\mathcal{B}}(\beta_{0},M_{n}/\sqrt{n})^{c}}\left(\frac{1}{4}M_{n}\cdot|\sqrt{n}(\beta-\beta_{0})|- \frac{1}{2}n(\beta-\beta_{0})^{2}\right) > -\frac{\underline{c}_{0}}{4}M_{n}^{2} \right\} + o(1).
\end{align*}
as $n\rightarrow \infty$. Defining $t = \sqrt{n}|\beta-\beta_{0}|$, it follows that
\begin{align*}
\sup_{\beta \in B_{\mathcal{B}}(\beta_{0},M_{n}/\sqrt{n})^{c}}\left[\frac{M_{n}}{4}\sqrt{n}|\beta-\beta_{0}|- \frac{1}{2}n(\beta-\beta_{0})^{2}\right] = \sup_{t \in B_{\mathcal{B}}(0,M_{n})^{c}}\left[\frac{M_{n}}{4}t- \frac{1}{2}t^{2}\right] = - \frac{M_{n}^{2}}{4}  
\end{align*}
as a concave quadratic function maximized subject to the constraint that $t \geq M_{n}$. This implies that
\begin{align*}
\mathbf{1}\left\{\underline{c}_{0}\sup_{\beta \in B_{\mathcal{B}}(\beta_{0},M_{n}/\sqrt{n})^{c}}\left(\frac{1}{4}M_{n}\cdot|\sqrt{n}(\beta-\beta_{0})|- \frac{1}{2}n(\beta-\beta_{0})^{2}\right) > -\frac{\underline{c}_{0}}{4}M_{n}^{2} \right\} = 0    
\end{align*}
and leads to the conclusion that $P_{0,YX|W}^{(n)}(A_{n,1}^{c}) \longrightarrow 0$ for $P_{0,W}^{\infty}$-almost every fixed realization $\{w_{i}\}_{i \geq 1}$ of $\{W_{i}\}_{i \geq 1}$ as $n\rightarrow \infty$, and, as a result, the complement rule implies $P_{0,YX|W}^{(n)}(A_{n,1}) \rightarrow 1$ as $n\rightarrow \infty$ conditionally given $P_{0,W}^{\infty}$-almost every realization $\{w_{i}\}_{i \geq 1}$ of $\{W_{i}\}_{i \geq 1}$ as $n\rightarrow \infty$.

\textit{Step 3.} Let $A_{n,2}$ be the event that
\begin{align*}
    \inf_{m \in B_{n,\mathcal{M}}(m_{0},D\delta_{n})\cap \mathcal{M}_{n}}\int_{\mathcal{B}}\exp\left(\frac{1}{\sqrt{n}}\tilde{\ell}_{n}(\beta_{0},m)\sqrt{n}(\beta-\beta_{0})- \frac{1}{2}n(\beta-\beta_{0})^{2}\tilde{I}_{n}(m)\right)d\Pi_{\mathcal{B}}(\beta) \geq \exp\left(-\underline{c}_{0}M_{n}^{2}/{8}\right).
\end{align*}
I want to show that $P_{0,YX|W}^{(n)}(A_{n,2}^{c}) \rightarrow 0$ for $P_{0,W}^{\infty}$-almost every fixed realization $\{w_{i}\}_{i \geq 1}$ of $\{W_{i}\}_{i \geq 1}$ as $n\rightarrow \infty$ because, combined with the result from \textit{Step 2}, it is sufficient to conclude the result of the lemma because it implies 
\begin{align*}
\sup_{m \in \mathcal{M}_{n} \cap B_{n,\mathcal{M}}(m_{0},D\delta_{n})}\Pi\left(\beta \in B_{\mathcal{B}}(\beta_{0},M_{n}/\sqrt{n})^{c}|\{(Y_{i},X_{i},w_{i})\}_{i=1}^{n},m\right) \leq \exp\left(-\frac{\underline{c}_{0}}{8}M_{n}^{2}\right)
\end{align*}
with $P_{0,YX|W}^{(n)}$-probability approaching one for $P_{0,W}^{\infty}$-almost every fixed realization $\{w_{i}\}_{i \geq 1}$ of $\{W_{i}\}_{i \geq 1}$, as $n\rightarrow \infty$
as $n\rightarrow \infty$. Let $C = \underline{c}_{0}/8\overline{c}_{0}$ with $0<\underline{c}_{0}<\overline{c}_{0}<\infty$ defined as in Part 3 of Lemma 1. For this step, I elect to work in the local parametrization $s = \sqrt{n}(\beta-\beta_{0})$. Consequently, I let $\Pi_{n}$ be the induced prior for $s $ (i.e., the pushforward of $\Pi_{\mathcal{B}}$ through the map $\beta \mapsto \sqrt{n}(\beta-\beta_{0})$), let $\mathcal{S}$ be the support of $\Pi_{n}$, and let $\mathcal{S}_{C} = \{s: |s| \leq \sqrt{C}\}$. Under Assumption 2, for large $n$, the induced prior $\Pi_{n}$ admits a Lebesgue density $\pi_{n}$ such that $\inf_{s \in \mathcal{S}_{C}}\pi_{n}(s) \geq \tilde{C}$ for some large constant $\tilde{C}>0$. Consequently, I apply a change of variables and conclude that for $n$ large,
\begin{align*}
 &\inf_{m \in \mathcal{M}_{n} \cap B_{n,\mathcal{M}}(m_{0},D\delta_{n})}\int_{\mathcal{B}}\exp\left(\frac{1}{\sqrt{n}}\tilde{\ell}_{n}(\beta_{0},m)\sqrt{n}(\beta-\beta_{0})- \frac{1}{2}n(\beta-\beta_{0})^{2}\tilde{I}_{n}(m)\right)d\Pi_{\mathcal{B}}(\beta) \\
  &\inf_{m \in \mathcal{M}_{n} \cap B_{n,\mathcal{M}}(m_{0},D\delta_{n})}\int_{\mathcal{S}}\exp\left(\frac{1}{\sqrt{n}}\tilde{\ell}_{n}(\beta_{0},m)s - \frac{1}{2}s^{2}\tilde{I}_{n}(m)\right)d\Pi_{n}(s) \\
 &\geq \tilde{C}\lambda(\mathcal{S}_{C})\inf_{m \in \mathcal{M}_{n} \cap B_{n,\mathcal{M}}(m_{0},D\delta_{n})} \int\exp\left(\frac{1}{\sqrt{n}}\tilde{\ell}_{n}(\beta_{0},m)s - \frac{1}{2}s^{2}\tilde{I}_{n}(m)\right)\pi_{\mathcal{S}_{C}}(s)ds,
\end{align*}
where $\lambda$ is the Lebesgue measure on $\mathbb{R}$ and $\pi_{\mathcal{S}_{C}}(s) = \mathbf{1}\{s \in \mathcal{S}_{C}\}/\lambda(\mathcal{S}_{C})$ is the density of a uniform distribution on $\mathcal{S}_{C}$. Applying Part 3 of Lemma 1, it follows that, for $P_{0,W}^{\infty}$-almost every fixed realization $\{w_{i}\}_{i \geq 1}$ of $\{W_{i}\}_{i \geq 1}$, the following holds:
\begin{align*}
P_{0,YX|W}^{(n)}\left(\sup_{m \in \mathcal{M}_{n} \cap B_{\mathcal{M},\infty}(m_{0},D\delta_{n})}\tilde{I}_{n}(m) \leq \overline{c}_{0}M_{n}^{2}\right) \rightarrow 1    
\end{align*} 
for any $M_{n}\rightarrow \infty$ as $n\rightarrow \infty$. Using that $C = \underline{c}_{0}/8\overline{c}_{0}$, it then follows that, for $P_{0,W}^{\infty}$-almost every fixed realization $\{w_{i}\}_{i \geq 1}$ of $\{W_{i}\}_{i \geq 1}$, the following is true:
\begin{align*}
-\frac{1}{2}s^{2}\tilde{I}_{n}(m) \geq -\frac{\underline{c}_{0}}{16}M_{n}^{2} \quad \forall \ (s,m) \in \mathcal{S}_{C}\times (\mathcal{M}_{n} \cap B_{n,\mathcal{M}}(m_{0},D\delta_{n}))
\end{align*}
with $P_{0,YX|W}^{(n)}$-probability approaching $1$ as $n\rightarrow \infty$. Consequently, 
\begin{align*}
P_{0,YX|W}^{(n)}\left(A_{n,2}^{c}\right) \leq P_{0,YX|W}^{(n)}(\tilde{A}_{n,2}^{c}) + o(1),
\end{align*}
for $P_{0,W}^{\infty}$-almost every fixed realization $\{w_{i}\}_{i \geq 1}$ of $\{W_{i}\}_{i \geq 1}$ as $n\rightarrow \infty$, where $\tilde{A}_{n,2}$ is the event
\begin{align*}
\inf_{m \in \mathcal{M}_{n} \cap B_{n,\mathcal{M}}(m_{0},D\delta_{n})}\int \exp\left(\frac{1}{\sqrt{n}}\tilde{\ell}_{n}(\beta_{0},m)s\right)\pi_{\mathcal{S}_{C}}(s) ds \geq  \left(\lambda(\mathcal{S}_{C})\tilde{C}\right)^{-1}\exp\left(-\frac{\underline{c}_{0}}{16}M_{n}^{2}\right).
\end{align*}
Applying Jensen's inequality,
\begin{align*}
&\int\exp\left(\frac{1}{\sqrt{n}}\tilde{\ell}_{n}(\beta_{0},m)s\right)\pi_{\mathcal{S}_{C}}(s)ds \geq \exp\left(\frac{1}{\sqrt{n}}\tilde{\ell}_{n}(\beta_{0},m)\int s\pi_{\mathcal{S}_{C}}(s)ds\right)
\end{align*}
for all $m \in \mathcal{M}_{n} \cap B_{n,\mathcal{M}}(m_{0},D\delta_{n})$.
As the infimum is the greatest lower bound, I then obtain
\begin{align*}
    &\inf_{m \in \mathcal{M}_{n} \cap B_{n,\mathcal{M}}(m_{0},D\delta_{n})}\int\exp\left(\frac{1}{\sqrt{n}}\tilde{\ell}_{n}(\beta_{0},m)s\right)\pi_{\mathcal{S}_{C}}(s)ds \\
&\quad \geq \inf_{m \in \mathcal{M}_{n} \cap B_{n,\mathcal{M}}(m_{0},D\delta_{n})}\exp\left(\frac{1}{\sqrt{n}}\tilde{\ell}_{n}(\beta_{0},m)\int s\pi_{\mathcal{S}_{C}}(s)ds\right) \\
&\quad \geq \exp\left(\inf_{m \in \mathcal{M}_{n} \cap B_{n,\mathcal{M}}(m_{0},D\delta_{n})}\frac{1}{\sqrt{n}}\tilde{\ell}_{n}(\beta_{0},m)\int s\pi_{\mathcal{S}_{C}}(s)ds\right) \\
&\quad \geq  \exp\left(\inf_{m \in \mathcal{M}_{n} \cap B_{n,\mathcal{M}}(m_{0},D\delta_{n})}\left(\left(\frac{1}{\sqrt{n}}\tilde{\ell}_{n}(\beta_{0},m)-\frac{1}{\sqrt{n}}\tilde{\ell}_{n}(\beta_{0},m_{0})\right)\int s\pi_{\mathcal{S}_{C}}(s)ds \right)\right) \\
&\quad \quad \times \exp\left( \frac{1}{\sqrt{n}}\tilde{\ell}_{n}(\beta_{0},m_{0})\int s\pi_{\mathcal{S}_{C}}(s)ds \right),   
\end{align*}
where the last inequality adds/subtracts $(1/\sqrt{n})\tilde{\ell}_{n}(\beta_{0},m_{0})$ and uses that the sum of the infima is less than or equal to the infimum of the sum. Since
\begin{align*}
    \left | \int s \pi_{\mathcal{S}_{C}}(s)ds\right| \leq \int |s|\pi_{\mathcal{S}_{C}}(s)ds \leq \sqrt{C}
\end{align*}
because $\pi_{\mathcal{S}_{C}}$ is supported on $\{s: |s| \leq \sqrt{C}\}$ and $\pi_{\mathcal{S}_{C}}$ integrates to $1$ (i.e., it is a probability density function), it follows from Part 1 of Lemma \ref{lem:technical} that
\begin{align*}
    \exp\left(\inf_{m \in \mathcal{M}_{n} \cap B_{n,\mathcal{M}}(m_{0},D\delta_{n})}\left(\left(\frac{1}{\sqrt{n}}\tilde{\ell}_{n}(\beta_{0},m)-\frac{1}{\sqrt{n}}\tilde{\ell}_{n}(\beta_{0},m_{0})\right)\int s\pi_{\mathcal{S}_{C}}(s)ds \right)\right) = 1 + o_{P_{0,YX|W}^{(n)}}(1)
\end{align*}
conditionally given $P_{0,W}^{\infty}$ almost every realization $\{w_{i}\}_{i \geq 1}$ of $\{W_{i}\}_{i \geq 1}$ as $n\rightarrow \infty$. As such, I arrive at the conclusion
\begin{align*}
&P_{0,YX|W}^{(n)}\left(\inf_{m \in \mathcal{M}_{n} \cap B_{n,\mathcal{M}}(m_{0},D\delta_{n})}\int_{\mathcal{B}}\exp(\ell_{n}(\beta,m)-\ell_{n}(\beta_{0},m))d \Pi_{\mathcal{B}}(\beta) < \exp\left(-\frac{\underline{c}_{0}}{8}M_{n}^{2}\right)\right) \\
&\quad \leq P_{0,YX|W}^{(n)}\left(\frac{1}{\sqrt{n}}\tilde{\ell}_{n}(\beta_{0},m_{0})\int s \pi_{\mathcal{S}_{C}}(s)ds< - \log \tilde{C}-\log\lambda(\mathcal{S}_{C}) - \frac{\underline{c}_{0}}{16}M_{n}^{2}\right) + o(1).
\end{align*}
for $P_{0,W}^{\infty}$-almost every fixed realization $\{w_{i}\}_{i \geq 1}$ of $\{W_{i}\}_{i \geq 1}$ as $n\rightarrow \infty$. Since the inequality 
\begin{align*}
-\log \tilde{C}-\log\lambda(\mathcal{S}_{C}) \leq \frac{\underline{c}_{0}}{32}M_{n}^{2}     
\end{align*}
holds for $n$ large, I can further conclude
\begin{align*}
&P_{0,YX|W}^{(n)}\left(A_{n,2}^{c}\right)  \leq P_{0,YX|W}^{(n)}\left(\frac{1}{\sqrt{n}}\tilde{\ell}_{n}(\beta_{0},m_{0})\int s \pi_{\mathcal{S}_{C}}(s)ds <  - \frac{\underline{c}_{0}}{32}M_{n}^{2}\right) + o(1).
\end{align*}
for $P_{0,W}^{\infty}$-almost every fixed realization $\{w_{i}\}_{i \geq 1}$ of $\{W_{i}\}_{i \geq 1}$ as $n\rightarrow \infty$. Next, the event
\begin{align*}
     \frac{1}{\sqrt{n}}\tilde{\ell}_{n}(\beta_{0},m_{0})\int s \pi_{\mathcal{S}_{C}}(s)ds<-\frac{\underline{c}_{0}}{32}M_{n}^{2}
     \end{align*}
     implies
     \begin{align*}
     \left |\frac{1}{\sqrt{n}}\tilde{\ell}_{n}(\beta_{0},m_{0})\int s \pi_{\mathcal{S}_{C}}(s)ds \right | > \frac{\underline{c}_{0}}{32}M_{n}^{2},
\end{align*}
and, as a result, I can apply Chebyshev's inequality to conclude that
\begin{align*}
&P_{0,YX|W}^{(n)}\left(\frac{1}{\sqrt{n}}\tilde{\ell}_{n}(\beta_{0},m_{0})\int s \pi_{\mathcal{S}_{C}}(s)ds<-\frac{\underline{c}_{0}}{32}M_{n}^{2} \right ) \\
&\quad \lesssim\frac{1}{M_{n}^{4}}E_{P_{0,YX|W}^{(n)}}\left|\frac{1}{\sqrt{n}}\tilde{\ell}_{n}(\beta_{0},m_{0})\right|^{2}\left |   \int s \pi_{\mathcal{S}_{C}}(s)ds \right |^{2} \\
&\quad \lesssim \frac{1}{M_{n}^{4}}E_{P_{0,YX|W}^{(n)}}\left|\frac{1}{\sqrt{n}}\tilde{\ell}_{n}(\beta_{0},m_{0})\right|^{2}
\end{align*}
with the second inequality using Jensen's inequality and that $\sup_{s \in \mathcal{S}_{C}}|s|^{2} \leq C.$ Since the random variable $\tilde{\ell}_{n}(\beta_{0},m_{0})$ satisfies $E_{P_{0,YX|W}^{(n)}}[\tilde{\ell}_{n}(\beta_{0},m_{0})] = 0$, it follows that 
\begin{align*}
E_{P_{0,YX|W}^{(n)}}\left|\frac{1}{\sqrt{n}}\tilde{\ell}_{n}(\beta_{0},m_{0}) \right|^{2} = Var_{P_{0,YX|W}^{(n)}}\left(\frac{1}{\sqrt{n}}\tilde{\ell}_{n}(\beta_{0},m_{0})\right) = \frac{1}{\sigma_{01}^{2}}\frac{1}{n}\sum_{i=1}^{n}E_{P_{0,YX|W}}[\varepsilon_{2}^{2}|W=w_{i}] =O(1) 
\end{align*} 
for $P_{0,W}^{\infty}$-almost every fixed realization $\{w_{i}\}_{i \geq 1}$ of $\{W_{i}\}_{i \geq 1}$, where the last equality reflects Part 2 of Assumption 1 and the Strong Law of Large Numbers. Consequently, I obtain
\begin{align*}
    P_{0,YX|W}^{(n)}\left(\frac{1}{\sqrt{n}}\tilde{\ell}_{n}(\beta_{0},m_{0})\int s \pi_{\mathcal{S}_{C}}(s)ds<-\frac{\underline{c}_{0}}{32}M_{n}^{2} \right ) = O(M_{n}^{-4}), 
\end{align*}
as $n\rightarrow \infty$, and, as a result, $P_{0,YX|W}^{(n)}(A_{n,2}^{c}) \rightarrow 0$ for $P_{0,W}^{\infty}$-almost every fixed realization $\{w_{i}\}_{i \geq 1}$ of $\{W_{i}\}_{i \geq 1}$ as $n\rightarrow \infty$ because $M_{n}\rightarrow \infty$ as $n\rightarrow \infty$. Since I chose the sequence $\{M_{n}\}_{n=1}^{\infty}$ with $M_{n}\rightarrow \infty$ arbitrarily, the argument holds for any such choice and I conclude the proof.
\end{proof}
\section{Proof of Propositions \ref{prop:suffgeneral}, \ref{prop:suffgeneralsieve}, \ref{prop:uniform}, \ref{prop:materngp}, and \ref{prop:materninformationloss}.}
\subsection{Proof of Proposition \ref{prop:suffgeneral} and Proposition \ref{prop:suffgeneralsieve}}\label{ap:leastsquares}
\begin{proof}[Proof of Proposition \ref{prop:suffgeneral}]
The proof has three steps. First, I show that it suffices to study the uniform behavior of the conditional posterior $\Pi(m \in B_{n,\mathcal{M}}(m_{0},D\delta_{n})^{c}|\{(Y_{i},X_{i},w_{i}')'\}_{i=1}^{n},\beta)$, where the uniformity holds over $\mathcal{B}$. Second, I use the first assumption of the proposition to derive a lower bound for the normalizing constant of the conditional posterior that holds uniformly over $\mathcal{B}$. Third, I combine the second assumption of the proposition with a `peeling' argument to establish the proposition.

\textit{Step 1.} Applying the law of total probability and using that probability measures are bounded from above by one, I obtain.
\begin{align*}
    &\Pi\left(m \in B_{n,\mathcal{M}}(m_{0},D\delta_{n})^{c}\middle |\{(Y_{i},X_{i},w_{i}')'\}_{i =1}^{n}\right) \\
    &\quad = \int_{\mathcal{B}}\Pi\left(m \in B_{n,\mathcal{M}}(m_{0},D\delta_{n})^{c}\middle |\{(Y_{i},X_{i},w_{i}')'\}_{i =1}^{n},\beta\right)d\Pi(\beta|\{(Y_{i},X_{i},w_{i}')'\}_{i=1}^{n}) \\
    &\quad \leq \sup_{\beta \in \mathcal{B}}\Pi\left(m \in B_{n,\mathcal{M}}(m_{0},D\delta_{n})^{c}\middle |\{(Y_{i},X_{i},w_{i}')'\}_{i =1}^{n},\beta\right).
\end{align*}
Consequently, it suffices to show that
\begin{align}\label{eq:keyobject}
    \sup_{\beta \in \mathcal{B}}\Pi\left(m \in B_{n,\mathcal{M}}(m_{0},D\delta_{n})^{c}\middle |\{(Y_{i},X_{i},w_{i}')'\}_{i =1}^{n},\beta\right) = o_{P_{0,YX|W}^{(n)}}(1).
\end{align}
for $P_{0,W}^{\infty}$ almost-every fixed realization $\{w_{i}\}_{i \geq 1}$ of $\{W_{i}\}_{i \geq 1}$. This is the purpose of the remaining steps.

\textit{Step 2.} I start by showing that the normalizing constant of the conditional posterior satisfies
\begin{align}\label{eq:keyeventdenominator}
     P_{0,YX|W}^{(n)}\left(\inf_{\beta \in \mathcal{B}}\int_{\mathcal{M}}\exp(\ell_{n}(\beta,m)- \ell_{n} (\beta,m_{0}))d \Pi_{\mathcal{M}}(m)\gtrsim \exp(-Cn\delta_{n}^{2})\right) \longrightarrow 1
\end{align}
for $P_{0,W}^{\infty}$-almost every fixed realization $\{w_{i}\}_{i \geq 1}$ of $\{W_{i}\}_{i \geq 1}$ as $n\rightarrow \infty$, where $C > 0$ is a constant. I first note that the log-(quasi)likelihood ratio process satisfies
\begin{align*}
    \ell_{n}(\beta,m)-\ell_{n}(\beta,m_{0}) &= \sum_{i=1}^{n}(m(w_{i})-m_{0}(w_{i}))'V^{-1}(\beta)(Z_{i}-m_{0}(w_{i})) \\
    &\quad - \frac{1}{2}\sum_{i=1}^{n}(m(w_{i})-m_{0}(w_{i}))'V^{-1}(\beta)(m(w_{i})-m_{0}(w_{i}))
\end{align*}
Since $B_{n,\mathcal{M}}(m_{0},\delta_{n}) \subseteq \mathcal{M}$, I can use the nonnegativity of the integrand and the expression for $\ell_{n}(\beta,m)-\ell_{n}(\beta,m_{0})$ to conclude that 
\begin{align*}
    &\int_{\mathcal{M}}\exp(\ell_{n}(\beta,m)- \ell_{n} (\beta,m_{0}))d \Pi_{\mathcal{M}}(m) \\
    &\quad \geq \int_{B_{n,\mathcal{M}}(m_{0},\delta_{n})} \exp(\ell_{n}(\beta,m)- \ell_{n} (\beta,m_{0}))d\Pi_{\mathcal{M}}(m) \\
    &\quad \geq \exp\left(-\frac{n\delta_{n}^{2}}{2}\sup_{\beta \in \mathcal{B}}\lambda_{max}(V^{-1}(\beta))\right)\\
    &\quad\quad \times \int_{B_{n,\mathcal{M}}(m_{0},\delta_{n})}\exp\left(\sum_{i=1}^{n}(m(w_{i})-m_{0}(w_{i}))'V^{-1}(\beta)\varepsilon_{i}\right)d\Pi_{\mathcal{M}}(m) \\
    &\quad =\exp\left(-\frac{n\delta_{n}^{2}}{2}\sup_{\beta \in \mathcal{B}}\lambda_{max}(V^{-1}(\beta))\right)\Pi_{\mathcal{M}}(B_{n,\mathcal{M}}(m_{0},\delta_{n}))\\
    &\quad \quad \times \int_{B_{n,\mathcal{M}}(m_{0},\delta_{n})}\exp\left(\sum_{i=1}^{n}(m(w_{i})-m_{0}(w_{i}))'V^{-1}(\beta)\varepsilon_{i}\right)d\Pi_{\mathcal{M},\delta_{n}}(m),
\end{align*}
where $\Pi_{\mathcal{M},\delta_{n}}$ is the renormalized restriction of $\Pi_{\mathcal{M}}$ to the event $B_{n,\mathcal{M}}(m_{0},\delta_{n})$. Note that the compactness of $\mathcal{B}$, the continuity of $\beta \mapsto V(\beta)$, and the positive-definiteness of $V(\beta)$ guarantees $\sup_{\beta \in \mathcal{B}}\lambda_{max}(V^{-1}(\beta))< \infty$. I now want to argue that
\begin{align}\label{eq:animportant}
    P_{0,YX|W}^{(n)}\left(\int_{B_{n,\mathcal{M}}(m_{0},\delta_{n})}\exp\left(\sum_{i=1}^{n}(m(w_{i})-m_{0}(w_{i}))'V^{-1}(\beta)\varepsilon_{i}\right)d\Pi_{\mathcal{M},\delta_{n}}(m) \geq \exp(-n\delta_{n}^{2})\right)\longrightarrow 1
\end{align}
for $P_{0,W}^{\infty}$-almost every fixed realization $\{w_{i}\}_{i \geq 1}$ of $\{W_{i}\}_{i \geq 1}$. Indeed, if this holds, then
\begin{align*}
 P_{0,YX|W}^{(n)}\left( \int_{\mathcal{M}}\exp(\ell_{n}(\beta,m)- \ell_{n} (\beta,m_{0}))d \Pi_{\mathcal{M}}(m) \geq \exp\left(-\bar{C}n\delta_{n}^{2}\right)\Pi_{\mathcal{M}}(B_{n,\mathcal{M}}(m_{0},\delta_{n}))\right)\longrightarrow 1   
\end{align*}
conditionally given $P_{0,W}^{\infty}$-almost every realization $\{w_{i}\}_{i \geq 1}$ of $\{W_{i}\}_{i \geq 1}$ as $n\rightarrow \infty$ for the choice of constant $\bar{C} =\frac{1}{2}\sup_{\beta \in \mathcal{B}}\lambda_{max}(V^{-1}(\beta))+ 1>0$. This implies (\ref{eq:keyeventdenominator}) following an application of Condition 1 in the statement of the proposition. To show (\ref{eq:animportant}), Jensen's inequality reveals
\begin{align*}
&\log \int_{\mathcal{B}_{n,\mathcal{M}}(m_{0},\delta_{n})}\exp\left(\sum_{i=1}^{n}(m(w_{i})-m_{0}(w_{i}))'V^{-1}(\beta)\varepsilon_{i}\right)d \Pi_{\mathcal{M},\delta_{n}}(m)    \\
&\geq \int_{\mathcal{B}_{n,\mathcal{M}}(m_{0},\delta_{n})}\left(\sum_{i=1}^{n}(m(w_{i})-m_{0}(w_{i}))'V^{-1}(\beta)\varepsilon_{i} \right)d\Pi_{\mathcal{M},\delta_{n}}(m)
\end{align*}
Consequently,
\begin{align*}
  &P_{0,YX|W}^{(n)}\left(\log \int_{\mathcal{B}_{n,\mathcal{M}}(m_{0},\delta_{n})}\exp\left(\sum_{i=1}^{n}(m(w_{i})-m_{0}(w_{i}))'V^{-1}(\beta)\varepsilon_{i}\right)d \Pi_{\mathcal{M},\delta_{n}}(m) < -n\delta_{n}^{2} \right)  \\
  &\quad \leq P_{0,YX|W}^{(n)}\left(\int_{\mathcal{B}_{n,\mathcal{M}}(m_{0},\delta_{n})}\left(\sum_{i=1}^{n}(m(w_{i})-m_{0}(w_{i}))'V^{-1}(\beta)\varepsilon_{i}\right)d \Pi_{\mathcal{M},\delta_{n}}(m) < -n\delta_{n}^{2} \right) \\
  &\quad \leq P_{0,YX|W}^{(n)}\left(\left|\int_{\mathcal{B}_{n,\mathcal{M}}(m_{0},\delta_{n})}\left(\sum_{i=1}^{n}(m(w_{i})-m_{0}(w_{i}))'V^{-1}(\beta)\varepsilon_{i}\right)d \Pi_{\mathcal{M},\delta_{n}}(m)\right| > n\delta_{n}^{2} \right) \\
  &\quad \leq \frac{1}{n^{2}\delta_{n}^{4}}E_{P_{0,YX|W}^{(n)}}\left|\int_{\mathcal{B}_{n,\mathcal{M}}(m_{0},\delta_{n})}\left(\sum_{i=1}^{n}(m(w_{i})-m_{0}(w_{i}))'V^{-1}(\beta)\varepsilon_{i}\right)d \Pi_{\mathcal{M},\delta_{n}}(m)\right|^{2},
\end{align*}
where the last inequality is Chebyshev's inequality. Applying Jensen's inequality and Fubini's theorem, I know that
\begin{align*}
    &E_{P_{0,YX|W}^{(n)}}\left|\int_{\mathcal{B}_{n,\mathcal{M}}(m_{0},\delta_{n})}\sum_{i=1}^{n}(m(w_{i})-m_{0}(w_{i}))'V^{-1}(\beta)\varepsilon_{i} d\Pi_{\mathcal{M},\delta_{n}}(m)\right|^{2}  \\
    &\quad \leq \int_{\mathcal{B}_{n,\mathcal{M}}(m_{0},\delta_{n})}\sum_{i=1}^{n}(m(w_{i})-m_{0}(w_{i}))'V^{-1}(\beta)E_{P_{0,YX|W}}(\varepsilon \varepsilon'|W=w_{i}) V^{-1}(\beta)(m(w_{i})-m_{0}(w_{i})) \\
    &\quad \lesssim n \delta_{n}^{2}, 
\end{align*}
for $P_{0,W}^{\infty}$-almost every sequence $\{w_{i}\}_{i \geq 1}$ of $\{W_{i}\}_{i \geq 1}$, where the last inequality uses the maximum eigenvalue assumption $\sup_{w \in \mathcal{W}}\lambda_{max}(E(\varepsilon\varepsilon'|W=w)) < \infty$ as well as the fact that $\sup_{\beta \in \mathcal{B}}\lambda_{max}(V^{-1}(\beta)) < \infty$ by the compactness of $\mathcal{B}$, the continuity of $\beta \mapsto V^{-1}(\beta)$, and the positive definiteness of $V(\beta)$. Since $n\delta_{n}^{2}\rightarrow \infty$, this implies that
\begin{align*}
    P_{0,YX|W}^{(n)}\left(\int_{\mathcal{B}_{n,\mathcal{M}}(m_{0},\delta_{n})}\exp\left(\sum_{i=1}^{n}(m(w_{i})-m_{0}(w_{i}))'V^{-1}(\beta)\varepsilon_{i}\right)d \Pi_{\mathcal{M},\delta_{n}}(m) < \exp(-n\delta_{n}^{2}) \right) = o(1)
\end{align*}
for $P_{0,W}^{\infty}$-almost every fixed realization $\{w_{i}\}_{i \geq 1}$ of $\{W_{i}\}_{i \geq 1}$ as $n\rightarrow \infty$. This verifies (\ref{eq:animportant}).

\textit{Step 3.} I study the numerator 
\begin{align*}
 \int_{B_{n,\mathcal{M}}(m_{0},D\delta_{n})^{c}}\exp(\ell_{n}(\beta,m)-\ell_{n}(\beta,m_{0}))d \Pi_{\mathcal{M}}(m).   
\end{align*}
Since 
\begin{align*}
B_{n,\mathcal{M}}(m_{0},D\delta_{n})^{c}= \bigcup_{j=0}^{\infty}A_{n,j}, \quad A_{n,j} = \{D2^{j}\delta_{n} \leq \max\{||m_{1}-m_{01}||_{n,2},||m_{2}-m_{02}||_{n,2}\} < D2^{j+1}\delta_{n}\},    
\end{align*} 
I can write
\begin{align*}
    &\int_{B_{n,\mathcal{M}}(m_{0},D\delta_{n})^{c}}\exp(\ell_{n}(\beta,m)-\ell_{n}(\beta,m_{0}))d \Pi_{\mathcal{M}}(m)=\sum_{j=0}^{\infty} \int_{A_{n,j}}\exp(\ell_{n}(\beta,m)-\ell_{n}(\beta,m_{0}))d\Pi_{\mathcal{M}}(m)
\end{align*}
Using the expression for $\ell_{n}(\beta,m)$, I know that
\begin{align*}
 \ell_{n}(\beta,m)-\ell_{n}(\beta,m_{0}) \leq  \sum_{i=1}^{n}(m(w_{i})-m_{0}(w_{i}))'V^{-1}(\beta)\varepsilon_{i} - D^{2}2^{2j}n\delta_{n}^{2}\inf_{\beta \in \mathcal{B}}\lambda_{min}(V^{-1}(\beta))   
\end{align*}
for all $(\beta,m) \in \mathcal{B} \times A_{n,j}$ and all $j \geq 0$. Consequently,
\begin{align*}
&\sup_{\beta \in \mathcal{B}}\sum_{j=0}^{\infty} \int_{A_{n,j}}\exp\left(\ell_{n}(\beta,m)-\ell_{n}(\beta,m_{0})\right)d\Pi_{\mathcal{M}}(m) \\
&\leq \sup_{\beta \in \mathcal{B}}\sum_{j=0}^{\infty} \int_{A_{n,j}}\exp\left(\sum_{i=1}^{n}(m(w_{i})-m_{0}(w_{i}))'V^{-1}(\beta)\varepsilon_{i}-\tilde{D}2^{2j}n\delta_{n}^{2}\right)d \Pi_{\mathcal{M}}(m) \\
&\leq \sum_{j=0}^{\infty} \sup_{\beta \in \mathcal{B}}\int_{A_{n,j}}\exp\left(\sqrt{n} \frac{1}{\sqrt{n}}\sum_{i=1}^{n}(m(w_{i})-m_{0}(w_{i}))'V^{-1}(\beta)\varepsilon_{i}-\tilde{D}2^{2j}n\delta_{n}^{2}\right)d \Pi_{\mathcal{M}}(m),
\end{align*}
where $\tilde{D}=\inf_{\beta \in \mathcal{B}}\lambda_{min}(V^{-1}(\beta))D^{2}>0$. Therefore, to complete the proof, it suffices to show that
\begin{align}\label{eq:condition}
    P_{0,YX|W}^{(n)}\left(\sup_{\beta \in \mathcal{B}}\sup_{m \in B_{n,\mathcal{M}}(m_{0},D2^{j+1}\delta_{n})}\left|\frac{1}{\sqrt{n}}\sum_{i=1}^{n}(m(w_{i})-m_{0}(w_{i}))'V^{-1}(\beta)\varepsilon_{i}\right| \leq \tilde{D}2^{2j-1}\sqrt{n}\delta_{n}^{2} \ \forall \ j\right) \rightarrow 1 
\end{align}
for $P_{0,W}^{\infty}$-almost every fixed realization $\{w_{i}\}_{i \geq 1}$ of $\{W_{i}\}_{i \geq 1}$ as $n\rightarrow \infty$. Indeed, it implies I can condition on the event given in the probability in (\ref{eq:condition}) and conclude that 
\begin{align*}
    &\sum_{j=0}^{\infty} \sup_{\beta \in \mathcal{B}}\int_{A_{n,j}}\exp\left(\sqrt{n} \frac{1}{\sqrt{n}}\sum_{i=1}^{n}(m(w_{i})-m_{0}(w_{i}))'V^{-1}(\beta)\varepsilon_{i}-\tilde{D}2^{2j}n\delta_{n}^{2}\right)d \Pi_{\mathcal{M}}(m) \\
    &\quad \leq \sum_{j=0}^{\infty}\exp\left(\tilde{D}2^{2j-1}n\delta_{n}^{2}-\tilde{D}2^{2j}n\delta_{n}^{2}\right) \\
    &\quad = \sum_{j=0}^{\infty}\exp\left(-\frac{\tilde{D}}{2}n\delta_{n}^{2}2^{2j}\right),
\end{align*}
which, when combined with (\ref{eq:keyeventdenominator}), implies that 
\begin{align*}
\sup_{\beta \in \mathcal{B}}\Pi(B_{n,\mathcal{M}}(m_{0},D\delta_{n})^{c}|\{(Y_{i},X_{i},w_{i}')'\}_{i=1}^{n}) \lesssim \sum_{j=0}^{\infty}\exp\left(\left(C_{1}-\frac{\tilde{D}}{2}2^{2j}\right)n\delta_{n}^{2}\right)
\end{align*}
with $P_{0,YX|W}^{(n)}$-probability approaching one, and, as a result, the claim of the proposition holds if $D > 0$ is sufficiently large because $n\delta_{n}^{2}\rightarrow \infty$. To verify (\ref{eq:condition}), I apply the union bound, Markov's inequality, and the assumptions on $\delta \mapsto \omega_{n}(\delta)$ to obtain that, for $D>0$ large, the following holds:
\begin{align*}
   &P_{0,YX|W}^{(n)}\left(\ \exists\ j \ s.t. \sup_{\beta \in \mathcal{B}}\sup_{m \in B_{n,\mathcal{M}}(m_{0},D2^{j+1}\delta_{n})}\left|\frac{1}{\sqrt{n}}\sum_{i=1}^{n}(m(w_{i})-m_{0}(w_{i}))'V^{-1}(\beta)\varepsilon_{i}\right| > \tilde{D}2^{2j-1}\sqrt{n}\delta_{n}^{2}\right) \\
   &\quad \leq \sum_{j=0}^{\infty}P_{0,YX|W}^{(n)}\left(\sup_{\beta \in \mathcal{B}}\sup_{m \in B_{n,\mathcal{M}}(m_{0},D2^{j+1}\delta_{n})}\left|\frac{1}{\sqrt{n}}\sum_{i=1}^{n}(m(w_{i})-m_{0}(w_{i}))'V^{-1}(\beta)\varepsilon_{i}\right| > \tilde{D}2^{2j-1}\sqrt{n}\delta_{n}^{2}\right) \\
   &\quad \leq \sum_{j=0}^{\infty}\frac{1}{\tilde{D}2^{2j-1}\sqrt{n}\delta_{n}^{2}}E_{P_{0,YX|W}^{(n)}}\left[\sup_{\beta \in \mathcal{B}}\sup_{m\in B_{n,\mathcal{M}}(m_{0},D2^{j+1}\delta_{n})}\left|\frac{1}{\sqrt{n}}\sum_{i=1}^{n}(m(w_{i})-m_{0}(w_{i}))'V^{-1}(\beta)\varepsilon_{i}\right|\right] \\
   &\quad \leq \sum_{j=0}^{\infty}\frac{\omega_{n}(D2^{j+1}\delta_{n})}{\tilde{D}2^{2j-1}\sqrt{n}\delta_{n}^{2}} \\
   &\quad \lesssim D^{\upsilon-2}\frac{\omega_{n}(\delta_{n})}{\sqrt{n}\delta_{n}^{2}}\sum_{j=0}^{\infty}2^{(\upsilon-2)j}.
\end{align*}
Since it is assumed that $\omega_{n}(\delta_{n}) \leq \sqrt{n}\delta_{n}^{2}$, I am then able to conclude that
\begin{align*}
P_{0,YX|W}^{(n)}\left(\ \exists\ j \ s.t. \sup_{\beta \in \mathcal{B}}\sup_{m \in B_{n,\mathcal{M}}(m_{0},D2^{j+1}\delta_{n})}\left|\frac{1}{\sqrt{n}}\sum_{i=1}^{n}(m(w_{i})-m_{0}(w_{i}))'V^{-1}(\beta)\varepsilon_{i}\right| > \tilde{D}2^{2j-1}\sqrt{n}\delta_{n}^{2}\right)= o(1)
\end{align*}
for $P_{0,W}^{\infty}$-almost every fixed realization $\{w_{i}\}_{i \geq 1}$ of $\{W_{i}\}_{i \geq 1}$ as $n\rightarrow \infty$ provided that $D \rightarrow \infty$. Hence, by setting $D > 0$ large (i.e., $D = D_{n}$ with $D_{n}\rightarrow \infty$ arbitrarily slowly as $n\rightarrow \infty$), I verify (\ref{eq:condition}).
\end{proof}
\begin{proof}[Proof of Proposition \ref{prop:suffgeneralsieve}]
Since $\Pi(\mathcal{M}_{n}^{c}|\{(Y_{i},X_{i},w_{i}')'\}_{i=1}^{n}) \rightarrow 0$ in $P_{0,YX|W}^{(n)}$-probability for $P_{0,W}^{\infty}$-almost every fixed realization $\{w_{i}\}_{i \geq 1}$ of $\{W_{i}\}_{i \geq 1}$ as $n\rightarrow \infty$, the law of total probability reveals that
\begin{align*}
\Pi\left(m \in B_{n,\mathcal{M}}(m_{0},D\delta_{n})^{c}\middle |\{(Y_{i},X_{i},w_{i}')'\}_{i =1}^{n}\right)  = \Pi\left(m \in  \mathcal{M}_{n} \cap B_{n,\mathcal{M}}(m_{0},D\delta_{n})^{c}\middle |\{(Y_{i},X_{i},w_{i}')'\}_{i =1}^{n}\right)  + o_{P_{0,YX}^{(n)}}(1)  
\end{align*}
for $P_{0,W}^{\infty}$ almost-every fixed realization $\{w_{i}\}_{i \geq 1}$ of $\{W_{i}\}_{i \geq 1}$ as $n\rightarrow \infty$. Consequently, the claim of the proposition holds by repeating the steps of Proposition \ref{prop:suffgeneral} with the only modification being that Step 3 uses $\mathcal{M}_{n} \cap B_{n,\mathcal{M}}(m_{0},D\delta_{n})^{c}$ instead of $B_{n,\mathcal{M}}(m_{0},D\delta_{n})^{c}$.
\end{proof}
\subsection{Proof of Proposition \ref{prop:uniform}}
\begin{proof}[Proof of Proposition \ref{prop:uniform}]
It suffices to verify Assumptions \ref{as:consistency} and \ref{as:emp_process}. Step 1 and Step 2 verify the conditions underpinning Proposition \ref{prop:suffgeneral} so that Assumption \ref{as:consistency} is verified. Step 3 verifies Assumption \ref{as:emp_process}.

\textit{Step 1.} I first verify Part 1 of Proposition \ref{prop:suffgeneral} for the sequence $\delta_{n} = \max\{\delta_{n,1},\delta_{n,2}\}$, where 
\begin{align*}
    \delta_{n,j} = \left(\frac{\log n}{n}\right)^{\frac{\alpha_{0,j}}{2\alpha_{0,j}+1}}
\end{align*}
for each $j \in \{1,2\}$. To that end, it suffices to show that, for each $j \in \{1,2\}$, there is a constant $C_{1,j}> 0$ such that
\begin{align}\label{eq:uniformsmallball}
\Pi_{\mathcal{M}_{j}}(||m_{j}-m_{0,j}||_{\infty}< \delta_{n,j}) \geq \exp(-C_{1,j}n\delta_{n,j}^{2}).    
\end{align} 
Indeed, $||\cdot||_{n,2} \leq ||\cdot||_{\infty}$ and $\Pi_{\mathcal{M}}= \Pi_{\mathcal{M}_{1}}\otimes \Pi_{\mathcal{M}_{2}}$ implies that if (\ref{eq:uniformsmallball}) holds, then, for $C_{1}=C_{1,1}+C_{1,2}$,
\begin{align*}
    \Pi_{\mathcal{M}}(m \in B_{n,\mathcal{M}}(m_{0},\delta_{n})) \geq \exp\left(-C_{1}n\delta_{n}^{2}\right)
\end{align*}
for $P_{0,W}^{\infty}$-almost every fixed realization $\{w_{i}\}_{i \geq 1}$ of $\{W_{i}\}_{i \geq 1}$, which is Part 1 of Proposition \ref{prop:suffgeneral}. To verify (\ref{eq:uniformsmallball}), I follow the argument of Proposition 1 in \cite{gine2011rates} with appropriate modifications. Let $m_{0,lk,j}$ be the coefficients in the wavelet series representation of $m_{0,j}$. Since the supremum norm of the difference $m_{0,j}-m_{j}$ satisfies $||m_{0,j}-m_{j}||_{\infty} \leq \sum_{l = 0}^{\infty}2^{l/2}\max_{0 \leq k \leq 2^{l}-1}|m_{0,lk,j}- 2^{-l(\alpha_{0,j}+1/2)}\nu_{lk,j}|$, the following inequality holds
\begin{align*}
    \Pi_{\mathcal{M}_{j}}(||m_{j}-m_{0,j}||_{\infty} < \delta_{n,j}) \geq \Pi_{\mathcal{M}_{j}}\left(\sum_{l=0}^{\infty}2^{l/2}\max_{0 \leq k \leq 2^{l}-1}|2^{-l(\alpha_{0,j}+1/2)}\nu_{lk,j}-m_{0,lk,j}| < \delta_{n,j}\right)
\end{align*}
Next, I write $\tilde{m}_{0,lk,j} = 2^{l(\alpha_{0j}+1/2)}m_{0,lk,j}$ and note that $|\tilde{m}_{0,lk,j}|\leq ||m_{0,j}||_{\infty,\infty,\alpha_{0,j}}  \leq M$ to obtain that
\begin{align*}
    &\sum_{l = 0}^{\infty}2^{l/2}\max_{0 \leq k \leq 2^{l}-1}|2^{-l(\alpha_{0,j}+1/2)}\nu_{lk,j}-m_{0,lk,j}| \\
    &\quad = \sum_{l=0}^{L}\max_{0\leq k \leq 2^{l}-1}2^{-l\alpha_{0,j}}|\nu_{lk,j}-\tilde{m}_{0,lk,j}|+\sum_{l=L+1}^{\infty}2^{-l\alpha_{0,j}}\max_{0 \leq k \leq 2^{l}-1}|\nu_{lk,j}-\tilde{m}_{0,lk,j}| \\
    &\quad \leq\sum_{l=0}^{L}\max_{0\leq k \leq 2^{l}-1}2^{-l\alpha_{0,j}}|\nu_{lk,j}-\tilde{m}_{0,lk,j}| + 2M\sum_{l=L+1}^{\infty}2^{-l\alpha_{0,j}}\\
    &\quad \leq \sum_{l=0}^{L}\max_{0\leq k \leq 2^{l}-1}2^{-l\alpha_{0,j}}|\nu_{lk,j}-\tilde{m}_{0,lk,j}| + 2^{-L\alpha_{0,j}}c_{1,j}
\end{align*}
where $c_{1,j}> 0$ is a constant that depends on $M$ and $\alpha_{0,j}$. Consequently,
\begin{align*}
 \sum_{l=0}^{L}\max_{0 \leq k \leq 2^{l}-1}2^{-l\alpha_{0,j}}|\nu_{lk,j}-\tilde{m}_{0,lk,j}| < \delta_{n,j}-2^{-L\alpha_{0,j}}c_{1,j} \implies   ||m_{j}-m_{0,j}||_{\infty} < \delta_{n,j}. 
\end{align*}
Now, if $L$ is chosen such that $2^{L}$ is of the order $(n/\log n)^{1/(2\alpha_{0,j}+1)}$, then I can use the above to conclude that, for $n$ large, there exists constants $c_{2,j},c_{3,j} > 0$ such that 
\begin{align*}
\Pi_{\mathcal{M}_{j}}\left(||m_{j}-m_{0,j}||_{\infty}< \delta_{n,j}\right)  &\geq \Pi_{\mathcal{M}_{j}}\left(\sum_{l=0}^{L}\max_{0 \leq k \leq 2^{l}-1}2^{-l\alpha_{0,j}}|\nu_{lk,j}-\tilde{m}_{0,lk,j}|< c_{2,j}\delta_{n,j}\right) \\
&\geq \Pi_{\mathcal{M}_{j}}\left(\max_{0 \leq l \leq L}\max_{0 \leq k \leq 2^{l}-1}|\nu_{lk,j}-\tilde{m}_{0,lk,j}|< c_{3,j}\delta_{n,j}\right),
\end{align*}
Let $S(L) = \sum_{l=0}^{L}\sum_{k=0}^{2^{l}-1}1 \leq 2 \cdot 2^{L}$. Since $\nu_{lk,j}$ are i.i.d $U(-M,M)$ random variables and $2^{L}$ is of the order $\delta_{n,j}^{-(1/\alpha_{0,j})}$, it follows that
\begin{align*}
    \Pi_{\mathcal{M}_{j}}\left(\max_{0 \leq l \leq L}\max_{0 \leq k \leq 2^{l}-1}|\nu_{lk,j}-\tilde{m}_{0,lk,j}|< c_{3,j}\delta_{n,j}\right) &= \left(\frac{c_{3,j}\delta_{n,j}}{2M}\right)^{S(L)} \\
    &\geq \exp(-c_{4,j}\log(1/\delta_{n,j})\delta_{n,j}^{-\frac{1}{\alpha_{0,j}}}) \\
    &\geq \exp(-C_{1,j}n\delta_{n,j}^{2}) 
\end{align*}
where $c_{4,j}, C_{1,j} > 0$ are constants. This verifies Part 1 of Assumption \ref{prop:suffgeneral}. Since $\alpha_{0,1},\alpha_{0,2} > 1/2$, $\delta_{n} = o(n^{-1/4})$.

\textit{Step 2.} I verify Part 2 of Proposition \ref{prop:suffgeneral}. Since $\mathcal{B}$ is compact and $\beta \mapsto V^{-1}(\beta)$ is continuous, an application of the triangle inequality reveals that
\begin{align*}
&E_{P_{0,YX|W}^{(n)}}\left[\sup_{(\beta,m) \in \mathcal{B} \times B_{\mathcal{M}}(m_{0},\delta)}\left|\frac{1}{\sqrt{n}}\sum_{i=1}^{n}(m(w_{i})-m_{0}(w_{i}))'V^{-1}(\beta)\varepsilon_{i}\right|\right] \\
&\quad \lesssim 
\sum_{(j_{1},j_{2}) \in \{1,2\}^{2}}E_{P_{0,YX|W}^{(n)}}\left[\sup_{m_{j_{1}} \in \mathcal{M}_{j_{1}}: ||m_{j_{1}}-m_{0,j_{1}}||_{n,2} \leq \delta}\left|\frac{1}{\sqrt{n}}\sum_{i=1}^{n}(m_{j_{1}}(w_{i})-m_{0,j_{1}}(w_{i}))\varepsilon_{i,j_{2}}\right|\right]
\end{align*}
Since $\varepsilon_{i,j_{2}}$ is subgaussian conditional on $W$ for all $j_{2} \in \{1,2\}$, Corollary 2.2.9 of \cite{vaartwellner96book} yields the following bound
\begin{align*}
  \sum_{(j_{1},j_{2}) \in \{1,2\}^{2}}  E_{P_{0,YX|W}^{(n)}}\left[\sup_{m_{j_{1}}\in \mathcal{M}_{j_{1}}: ||m_{j_{1}}-m_{0,j_{1}}||_{n,2} \leq \delta}\left|\frac{1}{\sqrt{n}}\sum_{i=1}^{n}(m_{j_{1}}(w_{i})-m_{0,j_{1}}(w_{i}))\varepsilon_{i,j_{2}}\right|\right] \lesssim \omega_{n}(\delta),
\end{align*}
for $P_{0,W}^{\infty}$-almost every fixed realization $\{w_{i}\}_{i \geq 1}$ of $\{W_{i}\}_{i \geq 1}$, where
\begin{align*}
\omega_{n}(\delta) = \int_{0}^{\delta}\sqrt{\log N(\tau,\mathcal{M}_{1},||\cdot||_{n,2})} d\tau   + \int_{0}^{\delta}\sqrt{\log N(\tau,\mathcal{M}_{2},||\cdot||_{n,2})} d\tau. 
\end{align*}
Since $\mathcal{M}_{1}$ and $\mathcal{M}_{2}$ are balls in the Besov spaces $B_{\infty,\infty}^{\alpha_{0,1}}([0,1])$ and $B_{\infty,\infty}^{\alpha_{0,2}}([0,1])$, respectively, and $||\cdot||_{n,2} \leq ||\cdot||_{\infty}$, I apply Theorem 4.3.36 of \cite{Gine_Nickl_2015} to conclude that 
\begin{align*}
  \int_{0}^{\delta}  \sqrt{\log N(\tau,\mathcal{M}_{j},||\cdot||_{n,2})} d\tau \leq \delta^{1-\frac{1}{2\alpha_{0,j}}}.
\end{align*}
for each $j \in \{1,2\}$. Consequently, any sequence $\{\bar{\delta}_{n}\}_{n \geq 1}$ for which $\bar{\delta}_{n} \geq n^{-\frac{\alpha_{0}}{(2\alpha_{0}+1)}}$ with $\alpha_{0} = \min\{\alpha_{0,1},\alpha_{0,2}\}$ satisfies Part 2 of Proposition \ref{prop:suffgeneral}. Consequently, $\delta_{n} = (n/\log n)^{-\alpha_{0}/(2\alpha_{0}+1)}$ satisfies Part 2 of Proposition \ref{prop:suffgeneral}.

\textit{Step 3.} The argument of Step 2 establishes
\begin{align*}
    E_{P_{0,YX|W}^{(n)}}\left[\sup_{m_{j_{1}}\in \mathcal{M}_{j_{1}}: ||m_{j_{1}}-m_{0,j_{1}}||_{n,2} \leq \delta}\left|\frac{1}{\sqrt{n}}\sum_{i=1}^{n}(m_{j_{1}}(w_{i})-m_{0,j_{1}}(w_{i}))\varepsilon_{i,j_{2}}\right|\right] \leq \sqrt{n}\delta_{n}^{2} \ \forall \ j_{1},j_{2} \in \{1,2\}
\end{align*}
conditionally given $P_{0,W}^{\infty}$-almost every realization $\{w_{i}\}_{i \geq 1}$ of $\{W_{i}\}_{i \geq 1}$. Consequently, Assumption \ref{as:emp_process} is verified because $U = \varepsilon_{1} - \beta_{0}\varepsilon_{2}$ is a linear combination of $\varepsilon_{1}$ and $\varepsilon_{2}$ and $\delta_{n} = o(n^{-1/4})$ implies $\sqrt{n}\delta_{n}^{2} \rightarrow 0$ as $n\rightarrow \infty$.
\end{proof}
\subsection{Proof of Propositions \ref{prop:materngp} and \ref{prop:materninformationloss}}
I start with a lemma that is used to verify the second condition in Proposition \ref{prop:suffgeneralsieve}. It is the only point where the normality assumption in Proposition \ref{prop:materngp} is used (in particular, it is used in \textit{Step 2} of the proof).
\begin{lemma}\label{lem:priorsieve}
Suppose that the following conditions hold
\begin{enumerate}
    \item The distribution of the data satisfies $(Y,X)|W \sim \mathcal{N}(m_{0}(W),V(\beta_{0}))$
    \item There exists a constant $C_{1}>0$ such that the prior $\Pi_{\mathcal{M}}$ satisfies $\Pi_{\mathcal{M}}(m \in B_{n,\mathcal{M}}(m_{0},\delta_{n})) \geq \exp(-C_{1}n\delta_{n}^{2})$ for $P_{0,W}^{\infty}$-almost every fixed realization $\{w_{i}\}_{i \geq 1}$ of $\{W_{i}\}_{i \geq 1}$.
    \item There exists sets $\{\mathcal{M}_{n}\}_{n \geq 1}$ such that $\Pi_{\mathcal{M}}(\mathcal{M}_{n}^{c}) \leq \exp(-Mn\delta_{n}^{2})$ for some $M >0$ large.
    \item The sequence $\{\delta_{n}\}_{n \geq 1}$ satisfies $n\delta_{n}^{2} \rightarrow \infty$ and $\delta_{n}\rightarrow 0$ as $n\rightarrow \infty$.
\end{enumerate}
Then, for prior $\Pi = \Pi_{\mathcal{B}} \otimes \Pi_{\mathcal{M}}$ with $\Pi_{\mathcal{B}}$ being the uniform distribution on $[-B,B]$ and $\Pi_{\mathcal{M}}$ being the prior described above, the posterior satisfies
\begin{align*}
 \Pi(\mathcal{M}_{n}^{c}|\{(Y_{i},X_{i},w_{i}')'\}_{i=1}^{n}) \overset{P_{0,YX|W}^{(n)}}{\rightarrow} 0    
\end{align*}
for $P_{0,W}^{\infty}$-almost every fixed realization $\{w_{i}\}_{i \geq 1}$ of $\{W_{i}\}_{i \geq 1}$ as $n\rightarrow \infty$.
\end{lemma}
\begin{proof}
The proof has two steps. The first step uses Condition 2 to show that the normalizing constant of the posterior is bounded from below by $\exp(-Cn\delta_{n}^{2})$ with $P_{0,YX|W}^{(n)}$-probability approaching one as $n\rightarrow \infty$. The second step uses the first step and Conditions 1,3, and 4 to conclude the result. Note that Step 2 is the only part of the argument in which correct specification (i.e., Condition 1) is used.

\textit{Step 1.} I start with the normalizing constant in the posterior. Since
\begin{align*}
    \frac{L_{n}(\beta,m)}{L_{n}(\beta_{0},m_{0})} = \frac{L_{n}(\beta,m)}{L_{n}(\beta,m_{0})} \cdot \frac{L_{n}(\beta,m_{0})}{L_{n}(\beta_{0},m_{0})},
\end{align*}
I know that
\begin{align*}
    \frac{1}{2B}\int_{[-B,B]}\int_{\mathcal{M}}\frac{L_{n}(\beta,m)}{L_{n}(\beta_{0},m_{0})}d\Pi_{\mathcal{M}}(m)d \beta &\geq \inf_{\beta \in [-B,B]}\int_{\mathcal{M}}\frac{L_{n}(\beta,m)}{L_{n}(\beta,m_{0})}d\Pi_{\mathcal{M}}(m)\\
    &\quad \times \frac{1}{2B}\int_{[-B,B]}\frac{L_{n}(\beta,m_{0})}{L_{n}(\beta_{0},m_{0})}d \beta
\end{align*}
Step 2 of the proof of Proposition \ref{prop:suffgeneral} yields that there exists a constant $C_{1}>0$ such that
\begin{align*}
  P_{0,YX|W}^{(n)}\left(  \inf_{\beta \in [-B,B]}\int_{\mathcal{M}}\frac{L_{n}(\beta,m)}{L_{n}(\beta,m_{0})}d\Pi_{\mathcal{M}}(m)  \geq \exp(-C_{1}n\delta_{n}^{2})\right) \longrightarrow 1.
\end{align*}
for $P_{0,W}^{\infty}$-almost every fixed sequence $\{w_{i}\}_{i \geq 1}$ of $\{W_{i}\}_{i \geq 1}$ as $n\rightarrow \infty$. Now, I show that there exists a constant $C_{2} > 0$ such that
\begin{align*}
  P_{0,YX|W}^{(n)}\left(   \frac{1}{2B}\int_{[-B,B]}\frac{L_{n}(\beta,m_{0})}{L_{n}(\beta_{0},m_{0})}d \beta \geq \exp(-C_{2}n\delta_{n}^{2}) \right) \longrightarrow 1
\end{align*}
for $P_{0,W}^{\infty}$-almost every fixed sequence $\{w_{i}\}_{i \geq 1}$ of $\{W_{i}\}_{i \geq 1}$ as $n\rightarrow \infty$. By the nonnegativity of the integrand and the assumption that $\delta_{n} = o(1)$,
\begin{align*}
\int_{[-B,B]}\frac{L_{n}(\beta,m_{0})}{L_{n}(\beta_{0},m_{0})}d \beta &\geq \int_{\{\beta \in \mathbb{R}: |\beta-\beta_{0}|< \delta_{n}\}}\frac{L_{n}(\beta,m_{0})}{L_{n}(\beta_{0},m_{0})}d \beta 
\end{align*}
for large $n$. Using the expression for $\ell_{n}(\beta,m_{0})$, I know that
\begin{align*}
    \ell_{n}(\beta,m_{0})-\ell_{n}(\beta_{0},m_{0}) &= \sqrt{n}(\beta-\beta_{0})\frac{1}{\sqrt{n}}\tilde{\ell}_{n}(\beta_{0},m_{0}) - \frac{1}{2}n(\beta-\beta_{0})^{2}\tilde{I}_{n}(m_{0}) \\
    &\geq \sqrt{n}(\beta-\beta_{0})\frac{1}{\sqrt{n}}\tilde{\ell}_{n}(\beta_{0},m_{0})-\frac{n\delta_{n}^{2}}{2}\tilde{I}_{n}(m_{0})
\end{align*}
By Part 3 of Lemma \ref{lem:technical}, I know that
\begin{align*}
  P_{0,YX|W}^{(n)}\left( \int_{[-B,B]}\frac{L_{n}(\beta,m_{0})}{L_{n}(\beta_{0},m_{0})}d \beta
\geq \exp\left(-\frac{\overline{c}_{0}}{2}n\delta_{n}^{2}\right)  \int_{\{\beta \in \mathbb{R}:|\beta-\beta_{0}| < \delta_{n}\}}\exp\left(\sqrt{n}(\beta-\beta_{0})\frac{1}{\sqrt{n}}\tilde{\ell}_{n}(\beta_{0},m_{0})\right) d \beta\right)\rightarrow 1
\end{align*}
for $P_{0,W}^{\infty}$-almost every fixed realization $\{w_{i}\}_{i \geq 1}$ of $\{W_{i}\}_{i \geq 1}$ as $n\rightarrow \infty$. Now, letting $\lambda(\cdot)$ denote the Lebesgue measure on $\mathbb{R}$, I know that
\begin{align*}
    &\int_{\{\beta \in \mathbb{R}:|\beta-\beta_{0}| < \delta_{n}\}}\exp\left(\sqrt{n}(\beta-\beta_{0})\frac{1}{\sqrt{n}}\tilde{\ell}_{n}(\beta_{0},m_{0})\right) d \beta\\
    &\quad = \lambda(\{\beta \in \mathbb{R}:|\beta-\beta_{0}|< \delta_{n})\int\exp\left(\sqrt{n}(\beta-\beta_{0})\frac{1}{\sqrt{n}}\tilde{\ell}_{n}(\beta_{0},m_{0})\right)\pi_{\delta_{n}}(\beta) d \beta,
\end{align*}
where $\pi_{\delta_{n}}(\beta)$ is the uniform distribution on $\{\beta \in \mathbb{R}: | \beta-\beta_{0} | < \delta_{n}\}$. An application of Chebyshev's inequality and Jensen's inequality reveals that
\begin{align*}
    P_{0,YX|W}^{(n)}\left(\int\exp\left(\sqrt{n}(\beta-\beta_{0})\frac{1}{\sqrt{n}}\tilde{\ell}_{n}(\beta_{0},m_{0})\right)\pi_{\delta_{n}}(\beta)  d \beta  < \exp\left(-\frac{\overline{c}_{0}}{2}n\delta_{n}^{2}\right)\right) =o (1)
\end{align*}
for $P_{0,W}^{\infty}$-almost every fixed realization $\{w_{i}\}_{i \geq 1}$ of $\{W_{i}\}_{i \geq 1}$. Consequently, the following holds
\begin{align*}
    P_{0,YX|W}^{(n)}\left(\int_{[-B,B]}\frac{L_{n}(\beta,m_{0})}{L_{n}(\beta_{0},m_{0})}d \beta \geq \exp\left(-\overline{c}_{0}n\delta_{n}^{2}\right)\lambda(\{\beta \in \mathbb{R}: | \beta-\beta_{0}| < \delta_{n}\})\right) = 1 + o(1)
\end{align*}
for $P_{0,W}^{\infty}$-almost every fixed realization $\{w_{i}\}_{i \geq 1}$ of $\{W_{i}\}_{i \geq 1}$ as $n\rightarrow \infty$. Since $\lambda(\{\beta \in \mathbb{R}: | \beta-\beta_{0}| < \delta_{n}\}) = 2\delta_{n}$ and $n\delta_{n}^{2} \rightarrow \infty$, it follows that
\begin{align*}
   P_{0,YX|W}^{(n)}\left(\int_{[-B,B]}\frac{L_{n}(\beta,m_{0})}{L_{n}(\beta_{0},m_{0})}d \beta \geq \exp\left(-(\overline{c}_{0}+1)n\delta_{n}^{2}\right))\right) \longrightarrow 1
\end{align*}
conditionally given $P_{0,W}^{\infty}$-almost every fixed realization $\{w_{i}\}_{i \geq 1}$ of $\{W_{i}\}_{i \geq 1}$ as $n\rightarrow \infty$. Define $C_{2} = \overline{c}_{0}+1$ to obtain the desired inequality.

\textit{Step 2.} By Step 1, I can condition on the event
\begin{align*}
  A_{n} =\left\{  \frac{1}{2B}\int_{[-B,B]}\int_{\mathcal{M}}\frac{L_{n}(\beta,m)}{L_{n}(\beta_{0},m_{0})}d \Pi_{\mathcal{M}}(m) d \beta \geq \exp(-Cn\delta_{n}^{2})\right\} ,
\end{align*}
where $C = C_{1}+C_{2}$.
Consequently, I deduce that
\begin{align*}
&E_{P_{0,YX|W}^{(n)}}\left[\Pi_{\mathcal{M}}(\mathcal{M}_{n}^{c}|\{(Y_{i},X_{i},w_{i}')'\}_{i=1}^{n})\right]\\
&\quad \leq E_{P_{0,YX|W}^{(n)}}\left[\Pi_{\mathcal{M}}(\mathcal{M}_{n}^{c}|\{(Y_{i},X_{i},w_{i}')'\}_{i=1}^{n})\mathbf{1}\{A_{n}\}\right] + o(1) \\
&\quad \leq \exp(Cn\delta_{n}^{2})\int_{[-B,B]}\int_{\mathcal{M}_{n}^{c}}E_{P_{0,YX|W}^{(n)}}\frac{L_{n}(\beta,m)}{L_{n}(\beta_{0},m_{0})} d\Pi_{\mathcal{M}}(m) d \beta + o(1)\\
&\quad = \exp(Cn\delta_{n}^{2})\Pi_{\mathcal{M}}(\mathcal{M}_{n}^{c}) + o(1),
\end{align*}
for $P_{0,W}^{\infty}$-almost every sequence $\{w_{i}\}_{i \geq 1}$ of $\{W_{i}\}_{i \geq 1}$, where the first inequality uses the law of total probability and the conclusion of Step 1 (i.e., $P_{0,YX|W}^{(n)}(A_{n}^{c}) = o(1)$), the second inequality uses $A_{n}$ and Fubini's theorem, and the equality uses the assumption of correct specification (i.e., $P_{0,YX|W} = \mathcal{N}(m_{0}(W),V(\beta_{0}))$). Since $\Pi_{\mathcal{M}}$ satisfies $\Pi_{\mathcal{M}}(\mathcal{M}_{n}^{c}) \lesssim \exp(-Mn\delta_{n}^{2})$ for $M > 0$, I can find a constant $\tilde{C} > 0$ such that the following holds
\begin{align*}
    E_{P_{0,YX|W}^{(n)}}\left[\Pi_{\mathcal{M}}(\mathcal{M}_{n}^{c}|\{(Y_{i},X_{i},w_{i}')'\}_{i=1}^{n})\right]  \lesssim \exp(-\tilde{C}n\delta_{n}^{2}) + o(1).
\end{align*}
for $P_{0,W}^{\infty}$-almost every sequence $\{w_{i}\}_{i \geq 1}$ of $\{W_{i}\}_{i \geq 1}$. Since $n\delta_{n}^{2} \rightarrow \infty$ as $n\rightarrow \infty$ and convergence in mean implies convergence in probability, the proof is complete.
\end{proof}
\begin{proof}[Proof of Proposition \ref{prop:materngp}]
The proof has four steps. Step 1 proposes a candidate sequence $\{\delta_{n}\}_{n \geq 1}$ that satisfies Parts 1 and 4 of Proposition \ref{prop:suffgeneralsieve}. Steps 2 and 3 then verify Parts 2 and 3, respectively, to verify Assumption \ref{as:consistency}. Step 4 verifies Assumption \ref{as:emp_process}.

\textit{Step 1.} Applying Lemma 5.3 of \cite{van2008reproducing}, the marginal prior $\Pi_{\mathcal{M}_{j}}$ satisfies the inequality $\Pi_{\mathcal{M}_{j}}(||m_{j}-m_{0j}||_{\infty}< \delta) \geq \exp(-\varphi_{m_{0,j},j}(\delta/2))$ for each $j \in \{1,2\}$ and any $\delta > 0$, where $\delta \mapsto \varphi_{m_{0,j},j}(\delta)$ for $\delta> 0$ is the supremum norm concentration function of the Mat\'{e}rn Gaussian process at $m_{0,j}$. That is, $\varphi_{m_{0,j},j}(\delta) = \inf_{h_{j} \in \mathbb{H}_{j}: ||h_{j}-m_{0,j}||_{\infty} <  \delta}\frac{1}{2}||h_{j}||_{\mathbb{H}_{j}} - \log \Pi_{\mathcal{M}_{j}}(||m_{j}||_{\infty} < \delta)$, for each $\delta > 0$, where $(\mathbb{H}_{j},||\cdot||_{\mathbb{H}_{j}})$ is the reproducing kernel Hilbert space (RKHS) associated with the Matern Gaussian process $m_{j}$. Applying Lemma 3 and 4 of \cite{van2011information}, the minimal sequence $\tilde{\delta}_{n,j}$ for which $\varphi_{m_{0,j},j}(\tilde{\delta}_{n,j}) \leq n\tilde{\delta}_{n,j}^{2}$ is $\tilde{\delta}_{n,j} = n^{-\min\{\alpha_{j},\alpha_{0,j}\}/(2\alpha_{j}+d_{w})}$, however, since I only require $\delta_{n} = o(n^{-1/4})$, it is without loss of generality to set $\delta_{n,j} = \max\{\tilde{\delta}_{n,j},n^{-\alpha_{j}/(2\alpha_{j}+d_{w})}\}$ because $\alpha_{j} > d_{w}/2$ guarantees that $n^{-\alpha_{j}/(2\alpha_{j}+d_{w})} = o(n^{-1/4})$ (in addition to $\alpha_{0,j} > \alpha_{j}/2+ d_{w}/4$ guaranteeing $\tilde{\delta}_{n,j} = o(n^{-1/4})$ too). Hence, a candidate is $\delta_{n} = \max\{\delta_{n,1},\delta_{n,2}\}$ and it meets the requirement that $\delta_{n} = o(n^{-1/4})$ and Part 1 of Proposition \ref{prop:suffgeneralsieve} because $||\cdot||_{n,2} \leq ||\cdot||_{\infty}$ and $\Pi_{\mathcal{M}} = \Pi_{\mathcal{M}_{1}}\otimes \Pi_{\mathcal{M}_{2}}$. Notice that $n \delta_{n}^{2}\rightarrow \infty$ as $n\rightarrow \infty$, so Part 4 of Proposition \ref{prop:suffgeneralsieve} holds too. It remains to verify Parts 2 and 3 of Proposition \ref{prop:suffgeneralsieve} in order to check Assumption \ref{as:consistency}. This is the task of the next two steps. 

\textit{Step 2.} I construct sets $\{\mathcal{M}_{n}\}_{n \geq 1}$ such that Part 2 of Proposition \ref{prop:suffgeneralsieve} holds. These sets are similar to those encountered in Theorem 3 of \cite{dejonge2013semiparametric}. For each $n \geq 1$, let $\mathcal{M}_{n} = \mathcal{M}_{n1} \times \mathcal{M}_{n2}$ with $\mathcal{M}_{nj} = M_{j} \sqrt{n}\delta_{n} \mathbb{H}_{j,1}+ \gamma_{n,j}C_{1}^{a_{j}}([0,1]^{d_{w}})$ for each $j \in \{1,2\}$, where $\mathbb{H}_{j,1} = \{f \in \mathbb{H}_{j}: ||f||_{\mathbb{H}_{j}} \leq 1\}$ is the unit ball in $(\mathbb{H}_{j},||\cdot||_{\mathbb{H}_{j}})$, $C_{1}^{a_{j}}([0,1]^{d_{w}}) = \{f \in C^{a_{j}}([0,1]): ||f||_{a_{j}} \leq 1\}$ is the unit ball in $(C^{a_{j}}([0,1]^{d_{w}}),||\cdot||_{a_{j}})$ for some $a_{j} \in (d_{w}/2,\alpha_{j})$, $\gamma_{n,j} = n^{-(\alpha_{j}-a_{j})/(2\alpha_{j}+d_{w})}$, and $M_{j}>0$ is a sufficiently large constant for each $j \in \{1,2\}$. Since $m_{j}$ takes values in $(C^{a_{j}}([0,1]^{d_{w}}),||\cdot||_{a_{j}})$ for any $ a_{j}  < \alpha_{j}$, I can view $m_{j}$ as a Gaussian element in the Banach space $(C^{a_{j}}([0,1]^{d_{w}}),||\cdot||_{a_{j}})$ for $a_{j} \in (d_{w}/2, \alpha_{j})$. Such an $a_{j}$ exists because $\alpha_{j}>d_{w}/2$ by the assumption of the proposition. The RKHS is unchanged because $C([0,1]^{d_{w}})$ is the completion of $C^{a_{j}}([0,1]^{d_{w}})$ with respect to the uniform norm $||\cdot||_{\infty}$ and $||\cdot||_{\infty} \leq ||\cdot||_{a_{j}}$ (see Lemma 8.1 in \cite{van2008reproducing}). I then apply Theorem 5.1 in \cite{van2008reproducing} to conclude that the law $\Pi_{\mathcal{M}_{j}}$ satisfies the inequality $\Pi_{\mathcal{M}_{j}}(\mathcal{M}_{nj}^{c}) \leq 1- \Phi\left( \Phi^{-1}\left(\Pi_{\mathcal{M}_{j}}(m_{j} \in \gamma_{n,j} C_{1}^{a_{j}}([0,1]^{d_{w}}))\right)+M_{j}\sqrt{n}\delta_{n}\right)$, where $\Phi(\cdot)$ denotes the standard normal cumulative distribution function. Since the unit ball $\mathbb{H}_{1,j}$ in $(\mathbb{H}_{j},||\cdot||_{\mathbb{H}_{j}})$ satisfies the metric entropy condition $\log N(\tau,\mathbb{H}_{1,j},||\cdot||_{a_{j}}) \lesssim \tau^{-2d_{w}/(d_{w}+2(\alpha_{j}-a_{j}))}$ for any $\tau>0$ (see Section 4.3 of \cite{van2011information}), I can invoke the relationship between the metric entropy of the unit ball in RKHS and small ball probabilities established in \cite{li1999approximation} to find a constant $L_{j}>0$ such that $\Pi_{\mathcal{M}_{j}}(||m_{j}||_{a_{j}}<\gamma_{n,j}) \geq \exp(-L_{j}\gamma_{n,j}^{-d_{w}/(\alpha_{j}-a_{j})})$ for each $j \in \{1,2\}$. Since $\gamma_{n,j} = n^{-(\alpha_{j}-a_{j})/(2\alpha_{j}+d_{w})}$ satisfies $\gamma_{n,j}^{-d_{w}/(\alpha_{j}-a_{j})} = n^{d_{w}/(2\alpha_{j}+d_{w})} \leq n \delta_{n}^{2}$ because of the lower bound $\delta_{n}\geq n^{-\alpha_{j}/(2\alpha_{j}+d_{w})}$, it follows that $\Pi_{\mathcal{M}_{j}}(||m_{j}||_{a_{j}}<\gamma_{n,j}) \geq \exp(-L_{j}n\delta_{n}^{2})$. Consequently, $\Pi_{\mathcal{M}_{j}}$ satisfies $\Pi_{\mathcal{M}_{j}}(\mathcal{M}_{nj}^{c}) \leq 1- \Phi\left(\Phi^{-1}\left(\exp(-L_{j}n\delta_{n}^{2})\right)+ M_{j}\sqrt{n}\delta_{n}\right)$ for each $j \in \{1,2\}$. Using the inequality $\Phi^{-1}(p) \geq -\sqrt{2\log (1/p)}$ for any $p \in (0,1)$, I know $\Phi^{-1}\left(\exp(-L_{j}n\delta_{n}^{2})\right) \geq -\sqrt{2L_{j}n\delta_{n}^{2}} = -\tilde{L}_{j}\sqrt{n}\delta_{n}$     
for $\tilde{L}_{j} = \sqrt{2L_{j}}$. Consequently, $\Pi_{\mathcal{M}_{j}}(\mathcal{M}_{nj}^{c}) \lesssim 1 - \Phi((M_{j}-\tilde{L}_{j})\sqrt{n}\delta_{n})$. By choosing $M_{j}>\tilde{L}_{j}$ arbitrarily large and using the inequality $\Phi(z) \geq 1-\exp(-z^{2}/2)$, it follows that $\Pi_{\mathcal{M}_{j}}(\mathcal{M}_{nj}^{c}) \lesssim \exp(-\frac{1}{2}(M_{j}-\tilde{L}_{j})^{2}n \delta_{n}^{2})$ for each $j \in \{1,2\}$. Since $\Pi_{\mathcal{M}}(\mathcal{M}_{n}^{c}) \lesssim \max\{\Pi_{\mathcal{M}_{1}}(\mathcal{M}_{n1}^{c}),\Pi_{\mathcal{M}_{2}}(\mathcal{M}_{n2}^{c})\}$, it follows that there exists a large constant $M>0$ such that $\Pi_{\mathcal{M}}(\mathcal{M}_{n}^{c})\lesssim \exp(-Mn\delta_{n}^{2})$. Applying Lemma \ref{lem:priorsieve}, I then conclude that the posterior satisfies $\Pi_{\mathcal{M}}(\mathcal{M}_{n}^{c}|\{(Y_{i},X_{i},w_{i}')'\}_{i=1}^{n}) \rightarrow 0$ in $P_{0,YX|W}^{(n)}$-probability for $P_{0,W}^{\infty}$-almost every fixed realization $\{w_{i}\}_{i \geq 1}$ of $\{W_{i}\}_{i \geq 1}$ as $n\rightarrow \infty$. This verifies Part 2 of Proposition \ref{prop:suffgeneralsieve}.

\textit{Step 3.} I verify Part 3 of Proposition \ref{prop:suffgeneralsieve}. Since $\mathcal{B}$ is compact and $\beta \mapsto V^{-1}(\beta)$ is continuous, an application of the triangle inequality reveals that
\begin{align*}
&E_{P_{0,YX|W}^{(n)}}\left[\sup_{(\beta,m) \in \mathcal{B} \times (\mathcal{M}_{n} \cap B_{\mathcal{M}}(m_{0},\delta))}\left|\frac{1}{\sqrt{n}}\sum_{i=1}^{n}(m(w_{i})-m_{0}(w_{i}))'V^{-1}(\beta)\varepsilon_{i}\right|\right] \\
&\quad \lesssim 
\sum_{(j_{1},j_{2}) \in \{1,2\}^{2}}E_{P_{0,YX|W}^{(n)}}\left[\sup_{m_{j_{1}} \in \mathcal{M}_{n,j_{1}}: ||m_{j_{1}}-m_{0,j_{1}}||_{n,2} \leq \delta}\left|\frac{1}{\sqrt{n}}\sum_{i=1}^{n}(m_{j_{1}}(w_{i})-m_{0,j_{1}}(w_{i}))\varepsilon_{i,j_{2}}\right|\right]
\end{align*}
conditionally given $P_{0,W}^{\infty}$-almost every realization $\{w_{i}\}_{i \geq 1}$ of the sequence $\{W_{i}\}_{i \geq 1}$. Since $\varepsilon_{j_{2}}$ is subgaussian conditional on $W$ for all $j_{2} \in \{1,2\}$, Corollary 2.2.9 of \cite{vaartwellner96book} implies
\begin{align*}
  \sum_{(j_{1},j_{2}) \in \{1,2\}^{2}}  E_{P_{0,YX|W}^{(n)}}\left[\sup_{m_{j_{1}}\in \mathcal{M}_{n,j_{1}}: ||m_{j_{1}}-m_{0,j_{1}}||_{n,2} \leq \delta}\left|\frac{1}{\sqrt{n}}\sum_{i=1}^{n}(m_{j_{1}}(w_{i})-m_{0,j_{1}}(w_{i}))\varepsilon_{i,j_{2}}\right|\right] \lesssim \omega_{n}(\delta),
\end{align*}
conditionally given $P_{0,W}^{\infty}$-almost every realization $\{w_{i}\}_{i \geq 1}$ of the sequence $\{W_{i}\}_{i \geq 1}$, where
\begin{align*}
\omega_{n}(\delta) = \int_{0}^{\delta}\sqrt{\log N(\tau,\mathcal{M}_{n,1},||\cdot||_{n,2})} d\tau   + \int_{0}^{\delta}\sqrt{\log N(\tau,\mathcal{M}_{n,2},||\cdot||_{n,2})} d\tau 
\end{align*}
I show that $\omega_{n}(\delta_{n}) \lesssim \sqrt{n}\delta_{n}^{2}$ for the choice of $\delta_{n}$ proposed in Step 1. Using the definition of $\mathcal{M}_{nj}$ and $||\cdot||_{n,2} \leq ||\cdot||_{\infty}$, I deduce that
    \begin{align*}
    \log N(\tau, \mathcal{M}_{n,j},||\cdot||_{n,2}) \leq \log N(\tau/2, M_{j}\sqrt{n}\delta_{n}\mathbb{H}_{j,1},||\cdot||_{\infty})+ \log N(\tau/2, \gamma_{n,j}C_{1}^{a_{j}}([0,1]),||\cdot||_{\infty}),
    \end{align*}
    and, as a result, I can use the metric entropy bound of the RKHS unit ball $\mathbb{H}_{j,1}$ stated in Step 2 and Theorem 2.7.1 of \cite{vaartwellner96book} to conclude that
    \begin{align*}
    \log N(\tau, M_{j}\sqrt{n}\delta_{n}\mathbb{H}_{j,1},||\cdot||_{\infty}) \lesssim \left(\frac{\sqrt{n}\delta_{n}}{\tau} \right)^{\frac{2d_{w}}{2\alpha_{j}+d_{w}}}
    \end{align*}
    and
    \begin{align*}
     \log N(\tau, \gamma_{n,j}C_{1}^{a_{j}}([0,1]),||\cdot||_{\infty}) \lesssim \left(\frac{\gamma_{n,j}}{\tau}\right)^{\frac{d_{w}}{a_{j}}}.  
    \end{align*}
    This implies that the metric entropy integral satisfies
    \begin{align*}
        \int_{0}^{\delta_{n}}\sqrt{\log N(\tau,\mathcal{M}_{n,j},||\cdot||_{\infty})} d \tau \lesssim n^{\frac{d_{w}}{4\alpha_{j}+2d_{w}}}\delta_{n} + \gamma_{n,j}^{\frac{d_{w}}{2a_{j}}}\delta_{n}^{\frac{2a_{j} - d_{w}}{2a_{j}}}.
    \end{align*}
    Calculations reveal that  $n^{\frac{d_{w}}{4\alpha_{j}+2d_{w}}}\delta_{n} + \gamma_{n,j}^{\frac{d_{w}}{2a_{j}}}\delta_{n}^{\frac{2a_{j} - d_{w}}{2a_{j}}} \lesssim \sqrt{n}\delta_{n}^{2} \ \forall \ j \in \{1,2\}$, so $\delta_{n}$ satisfies Part 3 of Proposition {\ref{prop:suffgeneralsieve}}.

\textit{Step 4.} I now verify Assumption \ref{as:emp_process} for $\mathcal{M}_{n} \cap B(m_{0},D\delta_{n})$, where $\{\mathcal{M}_{n}\}_{n \geq 1}$ is the sequence introduced in Step 2. Since $U= \varepsilon_{1}-\beta_{0}\varepsilon_{2}$ and $\sqrt{n}\delta_{n}^{2} = o(1)$, the maximal inequalities from Step 3 also verify Assumption \ref{as:emp_process} through an application of the triangle inequality. This completes the proof.
\end{proof}
\begin{proof}[Proof of Proposition \ref{prop:materninformationloss}]
It suffices to verify the conditions of Theorem 12.9 of \cite{ghosal2017fundamentals}. This involves four steps. Step 1 proposes a subsets of the parameter space $[-B,B] \times \mathcal{H}$ for which the posterior concentrates as $n\rightarrow \infty$. Step 2 verifies a LAN condition along the sets from Step 1 to check condition (12.13) of \cite{ghosal2017fundamentals}. Step 3 uses the subsets from Step 1 to verify condition (12.14) of \cite{ghosal2017fundamentals}. Step 4 establishes the claim of the proposition by applying Theorem 12.9 of \cite{ghosal2017fundamentals}. Since Theorem 12.9 in \cite{ghosal2017fundamentals} is an extension of Theorem 2 in \cite{castillo2012semiparametric}, these steps also correspond to verifying the conditions of his theorem and I will make this connection throughout the argument.

\textit{Step 1.} In order to verify Theorem 12.9 of \cite{ghosal2017fundamentals}, I must first construct sequences $\{\zeta_{n}\}_{n \geq 1}$ and $\{\mathcal{H}_{n}\}_{n \geq 1}$ such that $\mathcal{H}_{n} \subseteq \mathcal{H}$, $\zeta_{n} \rightarrow 0$ as $n\rightarrow \infty$, $n \zeta_{n}^{2}\rightarrow \infty$ as $n\rightarrow \infty$, and
\begin{align}\label{eq:maincontraction}
\Pi\left((\beta,\eta) \in [-B,B] \times \mathcal{H}_{n}:\sqrt{|\beta-\beta_{0}|^{2}+||\eta-\eta_{0}||_{n,2}^{2}}< D\zeta_{n}\middle |\{(Y_{i},X_{i},w_{i}')'\}_{i=1}^{n}\right)\overset{P_{0,YX|W}^{(n)}}\longrightarrow 1    
\end{align}
conditionally given $P_{0,W}^{\infty}$-almost every realization $\{w_{i}\}_{i \geq 1}$ of $\{W_{i}\}_{i \geq 1}$ as $n\rightarrow \infty$, where $D > 0$ is a large constant. In relation to \cite{castillo2012semiparametric}, this corresponds to verifying the first part of condition (C'). To establish this result, I will first show that it holds conditional on the sequence $\{(X_{i},W_{i}')'\}_{i \geq 1}$ and then apply the Bounded Convergence Theorem to obtain (\ref{eq:maincontraction}). To this end, let $p_{0,Y|XW}$ be the true density of $Y$ given $X$ and $W$, let $P_{0,Y|XW}^{(n)} = \bigotimes_{i=1}^{n}P_{0,Y|x_{i},w_{i}}$ with $P_{0,Y|x_{i},w_{i}}$ denoting the probability distribution associated with conditional density $p_{0,Y|X,W}(\cdot|x_{i},w_{i})$, and let $P_{0,WX}^{\infty}$ be the joint law of the i.i.d sequence $\{(X_{i},W_{i}')'\}_{i \geq 1}$. I show that there exists a sequence $\{\zeta_{n}\}_{n \geq 1}$ such that $\zeta_{n} \rightarrow 0$ as $n\rightarrow \infty$, $n\zeta_{n}^{2}\rightarrow \infty$ as $n\rightarrow \infty$, and the following holds
\begin{align}\label{eq:betaetacontraction}
\Pi\left(\sqrt{|\beta-\beta_{0}|^{2}+||\eta-\eta_{0}||_{n,2}^{2}}\geq D\zeta_{n}\middle |\{(Y_{i},x_{i},w_{i}')'\}_{i=1}^{n}\right)\overset{P_{0,Y|XW}^{(n)}}\longrightarrow 0    
\end{align} 
conditionally given $P_{0,XW}^{\infty}$-almost every realization $\{(x_{i},w_{i}')'\}_{i\geq 1}$ of $\{(X_{i},W_{i}')'\}_{i \geq 1}$ as $n\rightarrow \infty$, where $D>0$ is a large constant. Since $\liminf_{n\rightarrow \infty}\frac{1}{n}\sum_{i=1}^{n}x_{i}^{2}> 0$ for $P_{0,XW}^{\infty}$-almost every fixed realization $\{(x_{i},w_{i}')'\}_{i\geq 1}$ of $\{(X_{i},W_{i}')'\}_{i \geq 1}$ by the strong law of large numbers, a sufficient condition for (\ref{eq:betaetacontraction}) is
\begin{align}\label{eq:fixeddesign}
\Pi\left(\sqrt{||f(\cdot,\beta,\eta)-f(\cdot,\beta_{0},\eta_{0})||_{n,2}^{2}}\geq \tilde{D}\zeta_{n}\middle |\{(Y_{i},x_{i},w_{i}')'\}_{i=1}^{n}\right)\overset{P_{0,Y|XW}^{(n)}}\longrightarrow 0
\end{align}
conditionally given $P_{0,XW}^{\infty}$-almost every realization $\{(x_{i},w_{i}')'\}_{i\geq 1}$ of $\{(X_{i},W_{i}')'\}_{i \geq 1}$ as $n\rightarrow \infty$, where $\tilde{D}>0$ is a constant, $f(x,w,\beta,\eta) = x\beta+ \eta(w)$ for each $(\beta,\eta) \in [-B,B] \times \mathcal{H}$, and, with some abuse of notation, the norm $||\cdot||_{n,2}$ is the empirical $L^{2}$ norm over the design points $\{(x_{1},w_{1}),...,(x_{n},w_{n})\}$. Display (\ref{eq:fixeddesign}) is asking for a posterior contraction rate in a fixed design regression model with independent Gaussian errors. An application Theorem 4 of \cite{ghosal2007convergence} verifies that $\zeta_{n} = n^{-\min\{\alpha_{\eta},\alpha_{0,\eta}\}/(2\alpha_{\eta}+d_{w})}$ satisfies (\ref{eq:fixeddesign}). Indeed, Lemmas 3 and 4 of \cite{van2011information} imply that the supremum norm concentration function $\varphi_{\eta_{0}}$ of the Mat\'{e}rn Gaussian process at $\eta_{0}$ satisfies $\varphi_{\eta_{0}}(\zeta_{n}) \leq n\zeta_{n}^{2}$, and, as a result, Theorem 4 of \cite{ghosal2007convergence} can be verified for $d_{n} = ||\cdot||_{n,2}$, $\varepsilon_{n} = \zeta_{n}$, and $\Theta_{n} = [-B,B] \times \mathcal{H}_{n}^{*}$ with $\{\mathcal{H}_{n}^{*}\}_{n \geq 1}$ being the sieves from Theorem 2.1 of \cite{vaart2008rates} (i.e., $\mathcal{H}_{n}^{*} = \zeta_{n}C_{1}([0,1]) + M_{n}\mathbb{H}_{\eta,1}$, where $C_{1}([0,1]) = \{f \in C([0,1]):||f||_{\infty} \leq 1\}$ is the unit ball in space of continuous functions $C([0,1])$ equipped with the supremum norm $||\cdot||_{\infty}$, $\mathbb{H}_{\eta,1}$ is the unit ball in the RKHS $(\mathbb{H}_{\eta},||\cdot||_{\mathbb{H}_{\eta}})$ associated with the Gaussian process $\eta$, and $M_{n}$ is a positive constant). This completes the construction of the candidate sequence $\{\zeta_{n}\}_{n \geq 1}$. I now propose a sequence of sets $\{\mathcal{H}_{n}\}_{n \geq 1}$ such that
\begin{align}\label{eq:sievesfixeddesign}
\Pi(\eta \in \mathcal{H}_{n}^{c}|\{(Y_{i},x_{i},w_{i}')'\}_{i=1}^{n}) \overset{P_{0,Y|XW}^{(n)}}{\longrightarrow} 0    
\end{align}
for $P_{0,XW}^{\infty}$-almost every fixed realization $\{(x_{i},w_{i}')'\}_{i \geq 1}$ of $\{(X_{i},W_{i}')'\}_{i \geq 1}$ as $n\rightarrow \infty$. For each $n \geq 1$,  let $\mathcal{H}_{n} = \tilde{\mathcal{H}}_{n} \cap \left\{\eta \in \mathcal{H}: |G(m_{n,2},\eta)| \leq 2 \sqrt{n}\zeta_{n}||m_{n,2}||_{\mathbb{H}_{\eta}} \right\}$, where $\tilde{\mathcal{H}}_{n} = M \sqrt{n}\zeta_{n} \mathbb{H}_{\eta,1}+\gamma_{n}C_{1}^{a_{\eta}}([0,1]^{d_{w}})$ with $C_{1}^{a_{\eta}}([0,1]^{d_{w}})$ being the unit ball in the H\"{o}lder space $(C^{a_{\eta}}([0,1]^{d_{w}}),||\cdot||_{a_{\eta}})$ for $a_{\eta} \in (d_{w}/2,\alpha_{\eta})$, $\mathbb{H}_{\eta,1}$ being the unit ball in the RKHS $(\mathbb{H}_{\eta},||\cdot||_{\mathbb{H}_{\eta}})$, $\gamma_{n} = n^{-(\alpha_{\eta}-a_{\eta})/(2\alpha_{\eta}+d_{w})}$, $M>0$ being a sufficiently large positive constant, $G(m_{n,2},\eta)$ being a centered Gaussian random variable with variance $||m_{n,2}||_{\mathbb{H}_{\eta}}^{2}$ under $\Pi_{\mathcal{H}}$, and $\{m_{n,2}\}_{n \geq 1}$ is a sequence in $(\mathbb{H}_{\eta},||\cdot||_{\mathbb{H}_{\eta}})$ such that $||m_{n,2}||_{\mathbb{H}_{\eta}} \leq 2 \sqrt{n}\rho_{n}$ and $||m_{n,2}-m_{02}||_{\infty} \leq \rho_{n}$ for $\rho_{n} = n^{-\min\{\alpha_{\eta},\alpha_{0,2}\}/(2\alpha_{\eta}+d_{w})}$. The sequences $\{m_{n,2}\}_{n \geq 1}$ and $\{\rho_{n}\}_{n \geq 1}$ exist because $\inf_{h \in \mathbb{H}_{\eta}:||h-m_{02}||_{\infty}<\rho_{n}}\frac{1}{2}||h||_{\mathbb{H}_{\eta}}^{2} \leq n\rho_{n}^{2}$ by Lemma 3 of \cite{van2011information}. Applying the same argument as Step 2 of Proposition \ref{prop:materngp}, I can show that $\Pi_{\mathcal{H}}(\tilde{\mathcal{H}}_{n}^{c}) \lesssim \exp(-M n \zeta_{n}^{2})$. Moreover, the Gaussian tail bound $P(Z > z) \leq \exp(-z^{2}/2)$ for $Z \sim \mathcal{N}(0,1)$ implies $\Pi_{\mathcal{H}}(|G(m_{n,2},\eta)| > 2 \sqrt{n}\zeta_{n}||m_{n,2}||_{\mathbb{H}_{\eta}}) \leq \exp(-n\zeta_{n}^{2})$. Consequently, (\ref{eq:sievesfixeddesign}) holds by an application of Lemma 1 in \cite{ghosal2007convergence}. Lemma 1 of \cite{ghosal2007convergence} is valid because Kullback-Leibler balls of the form (3.5) in their paper are contained in empirical $L^{2}$ balls $\{(\beta,\eta):||f(\cdot,\beta,\eta)-f(\cdot,\beta_{0},\eta_{0})||_{n,2} < \zeta_{n}\}$ and $\Pi(||f(\cdot,\beta,\eta)-f(\cdot,\beta_{0},\eta_{0})||_{n,2} < \zeta_{n}) \geq \exp(-n\zeta_{n}^{2})$ holds when $\Pi_{\mathcal{B}} = U([-B,B])$, $\Pi_{\mathcal{H}}$ is the law of a centered Mat\'{e}rn gaussian process, and $\varphi_{\eta_{0}}(\zeta_{n}) \leq n\zeta_{n}^{2}$. Applying the law of total probability, (\ref{eq:betaetacontraction}) and (\ref{eq:sievesfixeddesign}) combine to establish
\begin{align*}
\Pi\left((\beta,\eta) \in [-B,B] \times \mathcal{H}_{n}:\sqrt{|\beta-\beta_{0}|^{2}+||\eta-\eta_{0}||_{n,2}^{2}}< D\zeta_{n}\middle |\{(Y_{i},x_{i},w_{i}')'\}_{i=1}^{n}\right)\overset{P_{0,Y|XW}^{(n)}}\longrightarrow 1       
\end{align*}
for $P_{0,XW}^{\infty}$-almost every fixed realization $\{(x_{i},w_{i}')'\}_{i \geq 1}$ of $\{(X_{i},W_{i}')'\}_{i\geq 1}$ as $n\rightarrow \infty$. This implies (\ref{eq:maincontraction}) holds by an application of the Bounded Convergence Theorem.

\textit{Step 2.} I verify (12.13) of \cite{ghosal2017fundamentals}, a result that synthesizes the local shape condition (N') and the RKHS approximation condition (E) in \cite{castillo2012semiparametric}. In my setting, this requires showing that
\begin{align*}
  \sup_{(\beta,\eta) \in \mathcal{B} \times \mathcal{H}_{n}: \max\{|\beta-\beta_{0}|,||\eta-\eta_{0}||_{n,2}\}<D\zeta_{n}}\frac{|\bar{R}_{n}(\beta,\eta)|}{1+ n|\beta-\beta_{0}|^{2}} \overset{P_{0,YX|W}^{(n)}}{\longrightarrow} 0
\end{align*}
for $P_{0,W}^{\infty}$-almost every fixed realization $\{w_{i}\}_{i \geq 1}$ of $\{W_{i}\}_{i \geq 1}$ as $n\rightarrow \infty$, where
\begin{align*}
\bar{R}_{n}(\beta,\eta) = \bar{\ell}_{n}(\beta,\eta)-\bar{\ell}_{n}(\beta_{0},\tilde{\eta}_{n}(\beta,\eta))-\sqrt{n}(\beta-\beta_{0})\frac{1}{\sqrt{n}}\tilde{\ell}_{n}(\beta_{0},m_{0})+\frac{1}{2}n(\beta-\beta_{0})^{2}\tilde{I}_{n}(m_{0})    
\end{align*}
with
\begin{align*}
    \bar{\ell}_{n}(\beta,\eta) = -\frac{n}{2}\log 2 \pi -\frac{n}{2}\log \sigma_{01}^{2}-\frac{1}{2\sigma_{01}^{2}}\sum_{i=1}^{n}(Y_{i}-X_{i}\beta-\eta(w_{i}))^{2}
\end{align*}
denoting the log-likelihood under the $(\beta,\eta)$-parametrization and 
\begin{align*}
\tilde{\eta}_{n}(\beta,\eta) = \eta + (\beta-\beta_{0})m_{n,2}    
\end{align*}
for each $(\beta,\eta) \in [-B,B] \times \mathcal{H}$. Letting $t= \beta-\beta_{0}$, a second-order Taylor expansion of $\ell_{n}(\beta_{0}+t,\eta)-\ell_{n}(\beta_{0},\eta-tm_{n,2})$ around $t = 0$ reveals that
\begin{align*}
    \bar{\ell}_{n}(\beta_{0}+t,\eta) - \bar{\ell}_{n}(\beta_{0},\eta + t m_{n,2}) &=t\frac{1}{\sigma_{01}^{2}}\sum_{i=1}^{n}(U_{i}-(\eta(w_{i})-\eta_{0}(w_{i})))(X_{i}-m_{n,2}(w_{i}))\\
    &\quad  -\frac{t^{2}}{2}\frac{1}{\sigma_{01}^{2}}\left\{\sum_{i=1}^{n}(X_{i}-m_{n2}(w_{i}))^{2}+2\sum_{i=1}^{n}(X_{i}-m_{n,2}(w_{i}))m_{n,2}(w_{i})\right\} \\
    &= \sqrt{n}(\beta-\beta_{0})T_{n,1}-\frac{1}{2}n(\beta-\beta_{0})^{2}T_{n,2},
\end{align*}
where
\begin{align*}
    T_{n,1} = \frac{1}{\sqrt{n}}\frac{1}{\sigma_{01}^{2}}\sum_{i=1}^{n}(U_{i}-(\eta(w_{i})-\eta_{0}(w_{i})))(X_{i}-m_{n,2}(w_{i}))
\end{align*}
and
\begin{align*}
       T_{n,2} = \frac{1}{\sigma_{01}^{2}n}\sum_{i=1}^{n}(X_{i}-m_{n2}(w_{i}))^{2}+\frac{2}{n\sigma_{01}^{2}}\sum_{i=1}^{n}(X_{i}-m_{n,2}(w_{i}))m_{n,2}(w_{i}).
\end{align*}
Consequently, to verify (12.13) of \cite{ghosal2017fundamentals}, it suffices to show that
\begin{align}\label{eq:betaetaLAN1}
\sup_{(\beta,\eta) \in \mathcal{B} \times \mathcal{H}_{n}: \max\{|\beta-\beta_{0}|,||\eta-\eta_{0}||_{n,2}\}<D\zeta_{n}}\left\{\frac{\left|\sqrt{n}(\beta-\beta_{0})\right|}{1+n|\beta-\beta_{0}|^{2}}\left|T_{n,1} -   \frac{1}{\sqrt{n}}\tilde{\ell}_{n}(\beta_{0},m_{0})  \right|\right\} \overset{P_{0,YX|W}^{(n)}}{\longrightarrow} 0 
\end{align}
and
\begin{align}\label{eq:betaetaLAN2}
\sup_{(\beta,\eta) \in \mathcal{B} \times \mathcal{H}_{n}: \max\{|\beta-\beta_{0}|,||\eta-\eta_{0}||_{n,2}\}<D\zeta_{n}}\left\{\frac{n\left|\beta-\beta_{0}\right|^{2}}{1+n|\beta-\beta_{0}|^{2}}\left|T_{n,2} -  \tilde{I}_{n}(m_{0})  \right|\right\} \overset{P_{0,YX|W}^{(n)}}{\longrightarrow} 0
\end{align}
conditionally given $P_{0,W}^{\infty}$-almost every realization $\{w_{i}\}_{i \geq 1}$ of $\{W_{i}\}_{i \geq 1}$ as $n\rightarrow \infty$. I first show condition (\ref{eq:betaetaLAN1}). Since $X_{i} -m_{n,2}(w_{i}) = X_{i}-m_{02}(w_{i})+m_{02}(w_{i})-m_{n,2}(w_{i})$ for $i=1,...,n$, I know that
\begin{align*}
\left| T_{n,1}- \frac{1}{\sqrt{n}}\tilde{\ell}_{n}(\beta_{0},m_{0})\right|&\leq \left |\frac{1}{\sigma_{01}^{2}\sqrt{n}}\sum_{i=1}^{n}U_{i}(m_{02}(w_{i})-m_{n,2}(w_{i}))\right|+\left|\frac{1}{\sigma_{01}^{2}\sqrt{n}}\sum_{i=1}^{n}(\eta(w_{i})-\eta_{0}(w_{i}))\varepsilon_{i2}\right|\\
 &\quad +\left|\frac{1}{\sigma_{01}^{2}\sqrt{n}}\sum_{i=1}^{n}(\eta(w_{i})-\eta_{0}(w_{i}))(m_{n,2}(w_{i})-m_{02}(w_{i}))   \right|
\end{align*}
By Chebyshev's inequality, the fact that $Var_{P_{0}}[U|X,W]=\sigma_{01}^{2}$, and $||m_{n,2}-m_{02}||_{\infty}\rightarrow 0$ as $n\rightarrow \infty$,
\begin{align*}
    \left|\frac{1}{\sqrt{n}}\sum_{i=1}^{n}U_{i}(m_{02}(w_{i})-m_{n,2}(w_{i}))\right| = o_{P_{0,YX|W}^{(n)}}(1)
\end{align*}
for $P_{0,W}^{\infty}$-almost every fixed realization $\{w_{i}\}_{i \geq 1}$ of $\{W_{i}\}_{i \geq 1}$. Since a symmetric argument to that encountered in Proposition \ref{prop:materngp} implies that
\begin{align*}
    \int_{0}^{D\zeta_{n}}\sqrt{\log N(\tau,\mathcal{H}_{n},||\cdot||_{n,2})}d\tau \rightarrow 0,
\end{align*}
for $\zeta_{n} = n^{-\min\{\alpha_{\eta},\alpha_{0,\eta}\}/(2\alpha_{\eta}+d_{w})}$ with $\alpha_{\eta}>d_{w}/2$ and $\alpha_{0,\eta} > \alpha_{\eta}/2 + d_{w}/4$, it follows by Corollary 2.2.9 of \cite{vaartwellner96book} that
\begin{align*}
    \sup_{\eta \in \mathcal{H}_{n}:||\eta-\eta_{0}||_{n,2} < D\zeta_{n}}\left| \frac{1}{\sqrt{n}}\sum_{i=1}^{n}\varepsilon_{i2}(\eta(w_{i})-\eta_{0}(w_{i}))\right| =o_{P_{0,YX|W}^{(n)}}(1)
\end{align*}
for $P_{0,W}^{\infty}$-almost every fixed realization $\{w_{i}\}_{i \geq 1}$ of $\{W_{i}\}_{i \geq 1}$ as $n\rightarrow \infty$. Finally, an application of Cauchy-Schwarz and the fact that $||\cdot||_{n,2} \leq ||\cdot||_{\infty}$ reveals that
\begin{align*}
\sup_{\eta \in \mathcal{H}_{n}:||\eta-\eta_{0}||_{n,2} < D\zeta_{n}} \left|    \frac{1}{\sqrt{n}}\sum_{i=1}^{n}(\eta(w_{i})-\eta_{0}(w_{i}))(m_{n,2}(w_{i})-m_{02}(w_{i})) \right|
&\leq \sqrt{n}\zeta_{n}\rho_{n} =o(1)
\end{align*}
for $P_{0,W}^{\infty}$-almost every fixed realization $\{w_{i}\}_{i \geq 1}$ of $\{W_{i}\}_{i \geq 1}$ as $n\rightarrow \infty$, where convergence holds because $\alpha_{\eta} > d_{w}/2$ and $\alpha_{0,\eta} > \alpha_{\eta}/2 + d_{w}/4$ implies $\sqrt{n}\zeta_{n} = o(n^{1/4})$ and $\alpha_{0,2} > \alpha_{\eta}/2+ d_{w}/4$ implies $\rho_{n} = o(n^{-1/4})$. This verifies (\ref{eq:betaetaLAN1}). For (\ref{eq:betaetaLAN2}), $X_{i}-m_{n,2}(w_{i}) = X_{i}-m_{02}(w_{i})+m_{02}(w_{i})-m_{n,2}(w_{i})$ implies that $T_{n,2}-\tilde{I}_{n}(m_{0})$ satisfies
\begin{align*}
    \left|T_{n,2}-\tilde{I}_{n}(m_{0})\right| &\lesssim \left|\frac{1}{n}\sum_{i=1}^{n}\varepsilon_{i2}(m_{n,2}(w_{i})-m_{02}(w_{i}))\right| +\frac{1}{n}\sum_{i=1}^{n}(m_{n,2}(w_{i})-m_{02}(w_{i}))^{2} \\
    &\quad + \left|\frac{1}{n}\sum_{i=1}^{n}\varepsilon_{i2}m_{02}(w_{i})\right|+\left|\frac{1}{n}\sum_{i=1}^{n}(m_{n,2}(w_{i})-m_{02}(w_{i}))m_{02}(w_{i})\right|
\end{align*}
Since $||m_{n,2}-m_{02}||_{\infty} \rightarrow 0$ as $n\rightarrow \infty$ and $\sup_{w \in [0,1]^{d_{w}}}Var_{P_{0}}(\varepsilon_{2}|W=w) < \infty$, an application of Chebyshev's inequality reveals that
\begin{align*}
\left|\frac{1}{n}\sum_{i=1}^{n}\varepsilon_{i2}(m_{n,2}(w_{i})-m_{02}(w_{i}))\right| \overset{P_{0,YX|W}^{(n)}}{\longrightarrow} 0 \quad \text{and} \quad \left|\frac{1}{n}\sum_{i=1}^{n}\varepsilon_{i2}m_{02}(w_{i})\right| \overset{P_{0,YX|W}^{(n)}}{\longrightarrow} 0 
\end{align*}
conditionally given $P_{0,W}^{\infty}$-almost every realization $\{w_{i}\}_{i \geq 1}$ of $\{W_{i}\}_{i \geq 1}$ as $n\rightarrow \infty$. Moreover, the convergence $||m_{n,2}-m_{02}||_{\infty} \rightarrow 0$ as $n\rightarrow \infty$ implies that
\begin{align*}
\frac{1}{n}\sum_{i=1}^{n}(m_{n,2}(w_{i})-m_{02}(w_{i}))^{2} \longrightarrow 0 \quad \text{and} \quad \left|\frac{1}{n}\sum_{i=1}^{n}(m_{n,2}(w_{i})-m_{02}(w_{i}))m_{02}(w_{i})\right| \longrightarrow 0
\end{align*}
conditionally given $P_{0,W}^{\infty}$-almost every realization $\{w_{i}\}_{i \geq 1}$ of $\{W_{i}\}_{i \geq 1}$ as $n\rightarrow \infty$. This verifies (\ref{eq:betaetaLAN2}).

\textit{Step 3.} The next step is to verify (12.14) of \cite{ghosal2017fundamentals}. In my setting, this requires
\begin{align*}
    \sup_{(\beta,\eta) \in \mathcal{B} \times \mathcal{H}_{n}: \max\{|\beta-\beta_{0}|,||\eta-\eta_{0}||_{n,2}\}<D\zeta_{n}}\frac{|\log (d\Pi_{n,\beta,\mathcal{H}}(\eta)/d\Pi_{\mathcal{H}}(\eta))|}{1+n|\beta-\beta_{0}|^{2}} \longrightarrow 0,
\end{align*}
where $\Pi_{n,\beta,\mathcal{H}}$ is the law of $\tilde{\eta}_{n}(\beta,\eta)$ under $\Pi_{\mathcal{H}}$ for $\beta$ fixed. Since $\{m_{n,2}\}_{n \geq 1}$ is a sequence in $(\mathbb{H}_{\eta}, ||\cdot||_{\mathbb{H}_{\eta}})$, I can apply the Cameron-Martin theorem to obtain
\begin{align*}
    \log \frac{d\Pi_{n,\beta,\mathcal{H}}(\eta)}{d\Pi_{\mathcal{H}}(\eta)} = (\beta-\beta_{0})G(m_{n,2},\eta)-\frac{1}{2}|\beta-\beta_{0}|^{2}||m_{n,2}||_{\mathbb{H}_{\eta}}^{2},
\end{align*}
where $G(m_{n,2},\eta) \sim \mathcal{N}(0,||m_{n,2}||_{\mathbb{H}_{\eta}}^{2})$ for each $n \geq 1$ under $\Pi_{\mathcal{H}}$. Since $\{\eta: |G(m_{n,2},\eta)| \leq 2 \sqrt{n}\zeta_{n}||m_{n,2}||_{\mathbb{H}_{\eta}}\} \subseteq \mathcal{H}_{n}$ and $||m_{n,2}||_{\mathbb{H}_{\eta}} \leq 2 \sqrt{n}\rho_{n}$, I know that 
\begin{align*}
  \sup_{(\beta,\eta) \in \mathcal{B} \times \mathcal{H}_{n}: \max\{|\beta-\beta_{0}|,||\eta-\eta_{0}||_{n,2}\}<D\zeta_{n}}\frac{|\log (d\Pi_{n,\beta,\mathcal{H}}(\eta)/d\Pi_{\mathcal{H}}(\eta))|}{1+n|\beta-\beta_{0}|^{2}} \leq 2 \zeta_{n}\rho_{n} + \frac{\rho_{n}^{2}}{n}= o(1)
\end{align*}
for $P_{0,W}^{\infty}$-almost every fixed realization $\{w_{i}\}_{i \geq 1}$ of $\{W_{i}\}_{i \geq 1}$ as $n\rightarrow \infty$ because $\rho_{n}\rightarrow 0$ and $\zeta_{n} \rightarrow 0$ as $n\rightarrow \infty$.

\textit{Step 4.} The results from Steps 1--3 permit an application of Theorem 12.9 of \cite{ghosal2017fundamentals} provided that I can show that the posterior satisfies
\begin{align*}
\sup_{\beta \in [-B,B]:|\beta-\beta_{0}| < D \zeta_{n}}\Pi\left(\tilde{\eta}_{n}(\beta,\{\eta \in \mathcal{H}_{n}:||\eta-\eta_{0}||_{n,2}< D\zeta_{n}\})^{c}\middle |\{(Y_{i},X_{i},w_{i})'\}_{i = 1}^{n},\beta = \beta_{0}\right) \overset{P_{0,YX|W}^{(n)}}{\longrightarrow } 0    
\end{align*} 
for $P_{0,W}^{\infty}$-almost every fixed realization $\{w_{i}\}_{i \geq 1}$ of $\{W_{i}\}_{i \geq 1}$ as $n\rightarrow \infty$. By definition of $\tilde{\eta}_{n}(\beta,\eta)$, I know that
\begin{align*}
\tilde{\eta}_{n}(\beta,\{\eta \in \mathcal{H}_{n}:||\eta-\eta_{0}||_{n,2}< D\zeta_{n}\}) &= \{\eta \in \tilde{\mathcal{H}}_{n}(\beta) : ||\eta-\eta_{0}||_{n,2} < D \zeta_{n}\} \\
&\quad \cap  \{\eta: |G(m_{n,2},\eta - (\beta-\beta_{0})m_{n,2})| \leq 2 \sqrt{n}\zeta_{n}||m_{n,2}||_{\mathbb{H}_{\eta}}\},
\end{align*}
where $\tilde{\mathcal{H}}_{n}(\beta) = \tilde{\mathcal{H}}_{n} + (\beta-\beta_{0})m_{n,2}$. Hence, the union bound implies that it suffices to show
\begin{align}\label{eq:step41}
\sup_{|\beta-\beta_{0}| < D \zeta_{n}}\Pi\left(\{\eta \in \tilde{\mathcal{H}}_{n}(\beta) : ||\eta-\eta_{0}||_{n,2} < D \zeta_{n}\}^{c}\middle |\{(Y_{i},X_{i},w_{i})'\}_{i = 1}^{n},\beta = \beta_{0}\right) \overset{P_{0,YX|W}^{(n)}}{\longrightarrow } 0    
\end{align} 
and
\begin{align}\label{eq:step42}
\sup_{|\beta-\beta_{0}| < D \zeta_{n}}\Pi\left(\{ |G(m_{n,2},\eta - (\beta-\beta_{0})m_{n,2})| \leq 2 \sqrt{n}\zeta_{n}||m_{n,2}||_{\mathbb{H}_{\eta}}\}^{c}\middle |\{(Y_{i},X_{i},w_{i})'\}_{i = 1}^{n},\beta = \beta_{0}\right) \overset{P_{0,YX|W}^{(n)}}{\longrightarrow } 0     
\end{align}
for $P_{0,W}^{\infty}$-almost every fixed realization $\{w_{i}\}_{i \geq 1}$ of $\{W_{i}\}_{i \geq 1}$ as $n\rightarrow \infty$. For (\ref{eq:step41}) (which corresponds to the second part of (C') in \cite{castillo2012semiparametric}), the sets $\tilde{\mathcal{H}}_{n}(\beta)$, where $|\beta-\beta_{0}| < D \zeta_{n}$, are, up to constants, the same as $\tilde{\mathcal{H}}_{n}$ because $\sup_{|\beta-\beta_{0}|<D\zeta_{n}}||(\beta-\beta_{0})m_{n,2}||_{\mathbb{H}_{\eta}} \leq 2D\sqrt{n}\zeta_{n}\rho_{n} = o(1)$. Consequently, a similar argument to Step 1 can be used to check (\ref{eq:step41}). For (\ref{eq:step42}), I apply Lemma 13 of \cite{ray2020semiparametric} to conclude that $G(m_{n,2},\eta - (\beta-\beta_{0})m_{n,2}) \sim \mathcal{N}((\beta-\beta_{0})||m_{n,2}||_{\mathbb{H}_{\eta}}^{2},||m_{n,2}||_{\mathbb{H}_{\eta}}^{2})$ for $\eta \sim \Pi_{\mathcal{H}}$ and $\beta$ fixed, and, as a result, $G(m_{n,2},\eta - (\beta-\beta_{0})m_{n,2})=G(m_{n,2},\eta)+(\beta-\beta_{0})||m_{n,2}||_{\mathbb{H}_{\eta}}^{2}$ for $\beta $ fixed. Since the inequality $\sup_{|\beta-\beta_{0}|<D\zeta_{n}}|\beta-\beta_{0}|\cdot ||m_{n,2}||_{\mathbb{H}_{\eta}} \leq 2D\sqrt{n}\zeta_{n}\rho_{n} \leq \sqrt{n}\zeta_{n}$ holds for large $n$ (as $\sqrt{n}\zeta_{n}\rho_{n} \rightarrow 0$ as $n\rightarrow \infty$ and $n\zeta_{n}^{2} \rightarrow \infty$ as $n\rightarrow \infty$), it follows that, for $n$ large,
\begin{align*}
    &\sup_{|\beta-\beta_{0}| < D \zeta_{n}}\Pi\left(\upsilon: |G(m_{n,2},\upsilon + (\beta-\beta_{0})m_{n,2})| > 2 \sqrt{n}\zeta_{n}||m_{n,2}||_{\mathbb{H}_{\eta}}\middle |\{(Y_{i},X_{i},w_{i})'\}_{i = 1}^{n},\beta = \beta_{0}\right) \\
    &\leq \Pi\left(\upsilon: |G(m_{n,2},\upsilon)| >  \sqrt{n}\zeta_{n}||m_{n,2}||_{n,2}\middle |\{(Y_{i},X_{i},w_{i})'\}_{i = 1}^{n},\beta = \beta_{0}\right) \overset{P_{0,YX|W}^{(n)}}{\longrightarrow} 0 
\end{align*}
for $P_{0,W}^{\infty}$-almost every fixed realization $\{w_{i}\}_{i \geq 1}$ of $\{W_{i}\}_{i \geq 1}$ as $n\rightarrow \infty$, where the convergence in probability holds by a similar argument to Step 1. This verifies all the conditions for Theorem 12.9 in \cite{ghosal2017fundamentals}, and, as a result, the proof is complete.
\end{proof}
\section{Proof of Theorem \ref{thm:KL}}
\begin{proof}[Proof of Theorem \ref{thm:KL}] The proof has two steps. \textit{Step 1} proves Part 1 by relating the objective function to least squares. \textit{Step 2} uses the results from Step 1 to prove the second statement of the theorem.

\textit{Step 1.} Let $Z = (Y,X)'$. I use the definition of Kullback Leibler divergence and the expression for $p_{\beta,m}$ to conclude that
\begin{align*}
&E_{P_{0}}\left[KL\left(p_{0,YX|W}(\cdot|W),p_{\beta,m}(\cdot|W)\right)\right] \\
&=E_{P_{0}}[\log p_{0,YX|W}(Y,X|W)]-E_{P_{0}}[\log p_{\beta,m}(Y,X|W)] \\
&=E_{P_{0}}[\log p_{0,YX|W}(Y,X|W)]+\frac{1}{2}\log 4\pi^{2} + \frac{1}{2} \log \det V(\beta)+ \frac{1}{2}E_{P_{0}}[(Z-m(W))'V^{-1}(\beta)(Z-m(W))].
\end{align*}
Note that the expectation is finite for all $(\beta,m) \in \mathcal{B} \times \mathcal{M}$ by Assumption \ref{as:dgp}, $\mathcal{M} \subseteq L^{2}(\mathcal{W})\times L^{2}(\mathcal{W})$, and the fact that $\det V(\beta) = \sigma_{01}^{2}\sigma_{02}^{2} > 0$. Next, I write $Z = m_{0}(W)+\varepsilon$ and note that $\varepsilon = Z-m_{0}(W)$ satisfies $E_{P_{0}}[\varepsilon|W] = 0$. Consequently, the law of iterated expectations implies
\begin{align*}
    &\frac{1}{2}E_{P_{0}}\left[(Z-m(W))'V(\beta)^{-1}(Z-m(W))\right]\\ 
    &\quad = \frac{1}{2}E_{P_{0}}\left[(\varepsilon+m_{0}(W)-m(W))'V(\beta)^{-1}(\varepsilon+ m_{0}(W)-m(W))\right] \\
    &\quad =\frac{1}{2}E_{P_{0}}[(m(W)-m_{0}(W))'V(\beta)^{-1}(m(W)-m_{0}(W))] +\frac{1}{2}E_{P_{0}}[\varepsilon'V(\beta)^{-1}\varepsilon].
\end{align*}
Assumption \ref{as:dgp} and $V(\beta)^{-1}$ being symmetric and positive-definite guarantees $E_{P_{0}}[\varepsilon'V(\beta)^{-1}\varepsilon]  \in (0,\infty)$. Positive-definiteness of $V(\beta)$ follows from Sylvester's criterion because $\sigma_{01}^{2}+\beta^{2}\sigma_{01}^{2}>0$ and $\det V(\beta) = \sigma_{01}^{2}\sigma_{02}^{2}>0$. Hence, I can combine this with the above to conclude that there is a constant $C(\beta) \in (-\infty,\infty)$ such that
\begin{align*}
&E_{P_{0}}\left[KL\left(p_{0}(\cdot|W),p_{\beta,m}(\cdot|W)\right)\right] = C(\beta) + \frac{1}{2}E_{P_{0}}[(m(W)-m_{0}(W))'V(\beta)^{-1}(m(W)-m_{0}(W))].  
\end{align*}
Consequently, it suffices to solve 
\begin{align}\label{eq:WLS}
    \min_{m \in \mathcal{M}}E_{P_{0}}\left[(m(W)-m_{0}(W))'V(\beta)^{-1}(m(W)-m_{0}(W))\right].
\end{align}
I establish that $m_{0}$ uniquely solves (\ref{eq:WLS}). First, $m_{0}$ minimizes the objective function because 
\begin{align*}
E_{P_{0}}\left[(m(W)-m_{0}(W))'V(\beta)^{-1}(m(W)-m_{0}(W))\right]\geq 0    
\end{align*}
for all $m \in \mathcal{M}$ with equality at $m = m_{0}$ and the nuisance parameter space is chosen so that $m_{0} \in \mathcal{M}$ (i.e., $m_{0}$ is feasible). For uniqueness, I note that the positive-definiteness of $V(\beta)^{-1}$ implies
\begin{align}\label{eq:mproblem2}
    &E_{P_{0}}\left[(m(W)-m_{0}(W))'V(\beta)^{-1}(m(W)-m_{0}(W))\right] \nonumber \\
    &\quad \geq \lambda_{min}(V(\beta)^{-1}) \left\{E_{P_{0}}[(m_{1}(W)-m_{01}(W))^{2}] + E_{P_{0}}[(m_{2}(W)-m_{02}(W))^{2}] \right\}
\end{align}
and
\begin{align}\label{eq:mproblem1}
&E_{P_{0}}\left[(m(W)-m_{0}(W))'V(\beta)^{-1}(m(W)-m_{0}(W))\right] \nonumber\\
&\quad \leq \lambda_{max}(V(\beta)^{-1}) \left\{E_{P_{0}}[(m_{1}(W)-m_{01}(W))^{2}] + E_{P_{0}}[(m_{2}(W)-m_{02}(W))^{2}] \right\},   
\end{align}
Using the definition of the $L^{2}$-norm, it follows that 
\begin{align*}
    E_{P_{0}}[(m_{1}(W)-m_{01}(W))^{2}] + E_{P_{0}}[(m_{2}(W)-m_{02}(W))^{2}] >0
\end{align*}
if and only if there exists $j \in \{1,2\}$ such that
\begin{align*}
P_{0}(m_{j}(W) \neq m_{0j}(W))>0.     
\end{align*}
From this, I conclude that $m_{0}$ is the unique solution to (\ref{eq:WLS}) in an almost-sure sense.

\textit{Step 2.} Based on the results of \textit{Step 1}, it is sufficient to solve
\begin{align*}
\min_{\beta \in \mathcal{B}}E_{P_{0}}[KL(p_{0}(\cdot|W),p_{\beta,m_{0}}(\cdot|W))].
\end{align*}
to obtain the second assertion of the theorem. The factorization of a bivariate normal distribution into its conditional and marginal distributions offers further simplification because it reduces the problem to solving
\begin{align*}
\min_{\beta \in \mathcal{B}}E_{P_{0}}\left[\frac{1}{2\sigma_{01}^{2}}\left(Y-m_{01}(W)-(X-m_{02}(W))\beta\right)^{2} \right].   
\end{align*}
Adding and subtracting $(X-m_{02}(W))\beta_{0}$ and invoking Part 4 of Assumption \ref{as:dgp}, I can find a constant $C \in (-\infty,\infty)$ that does not depend on $\beta$ such that
\begin{align*}
E_{P_{0}}\left[\frac{1}{2\sigma_{01}^{2}}\left(Y-m_{01}(W)-(X-m_{02}(W))\beta\right)^{2} \right] = C + \frac{1}{2\sigma_{01}^{2}}(\beta-\beta_{0})^{2}E_{P_{0}}[(X-m_{02}(W))^{2}],
\end{align*}
and, as a result, it suffices to solve the least squares problem
\begin{align}\label{eq:betaprob}
 \min_{\beta\in \mathcal{B}}\left\{\frac{1}{2\sigma_{01}^{2}}(\beta-\beta_{0})^{2}E_{P_{0}}[(X-m_{02}(W))^{2}]\right\}.   
\end{align}
Since the objective function is nonnegative and equal to zero when $\beta = \beta_{0}$ (feasible because $\beta_{0} \in \mathcal{B}$), it follows that $\beta_{0}$ is a solution to (\ref{eq:betaprob}). Uniqueness follows from Part 2 of Assumption \ref{as:dgp} because the restriction $E_{P_{0}}[(X-m_{02}(W))^{2}] \in [\underline{c},\overline{c}]$ implies that the objective function is equal to $0$ if and only if $\beta = \beta_{0}$.
\end{proof}
\section{Multivariate Extension}\label{ap:multivariate}
The main text assumes that $\sigma_{01}^{2}$ is known and treats $X$ as a scalar random variable, however, this was largely for expositional reasons and can be relaxed with minor modifications to the assumptions. This section presents an extension of Theorem \ref{thm:BVM} to allow for unknown $\sigma_{01}^{2}$ and multivariate $X$.
\subsection{DGP, Bayesian Model, and Assumptions about Posterior Concentration}
\subsubsection{Data-Generating Process}
Assumption \ref{as:dgpmulti} extends Assumption \ref{as:dgp} to allow for multivariate $X$ and unknown $\sigma_{01}^{2}$. The existence of fourth moments for the projection errors $Y-E_{P_{0}}[Y|X]$ (i.e., the first condition in Part 4) relates to the estimation of $\sigma_{01}^{2}$. For notation, the symbol $||\cdot||_{2}$ refers to the Euclidean norm.
\begin{assumption}\label{as:dgpmulti} 
The distribution of the data $P_{0}$ satisfies the following restrictions:
\begin{enumerate}
\item The distribution of $Y$ given $X$ and $W$ satisfies $E_{P_{0}}[Y|X,W]= X'\beta_{0}+\eta_{0}(W)$ and $Var_{P_{0}}[Y|X,W] = \sigma_{01}^{2}$ for some $(\beta_{0},\eta_{0}, \sigma_{01}^{2}) \in \mathcal{B} \times \mathcal{H} \times \mathbb{R}_{++}$, where $\mathcal{B} \subseteq \mathbb{R}^{d_{x}}$, $\mathcal{H} \subseteq L^{2}(\mathcal{W})$. 
\item There exists constants $0<\underline{c}< \overline{c}<\infty$ such that $\underline{c}<\lambda_{min}(E_{P_{0}}[(X-E_{P_{0}}[X|W])(X-E_{P_{0}}[X|W])'])$ and $\lambda_{max}(E_{P_{0}}[(X-E_{P_{0}}[X|W])(X-E_{P_{0}}[X|W])']) \leq \overline{c}$.
\item The conditional distribution of $(Y,X')'$ given $W$ has a density $p_{0,YX|W}$ with respect to the Lebesgue measure and $E_{P_{0}}|\log p_{0,YX|W}(Y,X|W)| \in (0,\infty)$
\item The projection errors $Y-E_{P_{0}}[Y|W]$ and $X- E_{P_{0}}[X|W]$ are such that $E_{P_{0}}[(Y-E_{P_{0}}[Y|W])^{4}] \in (0,\infty)$ and $E_{P_{0}}[||X-E_{P_{0}}[X|W]||_{2}^{4}] \in (0,\infty)$.
\end{enumerate} 
\end{assumption}
\subsubsection{Bayesian Model}
I extend the Bayesian model from Section \ref{sec:DGP} to accommodate multivariate $X$ and unknown $\sigma_{01}^{2}$. I parametrize the sampling model in terms of precision $\xi_{0} = 1/\sigma_{01}^{2}$, which means that the extension of the sampling model is
\begin{align}
    &Y_{i}|\{X_{i},W_{i}\}_{i=1}^{n},\beta,m,\xi \overset{ind}{\sim} \mathcal{N}(m_{1}(W_{i})+ (X_{i}-m_{2}(W_{i}))'\beta,\xi^{-1}) \label{eq:samplingYmulti}\\
    &X_{i}|\{W_{i}\}_{i=1}^{n},\beta,m,\xi\overset{ind}{\sim} \mathcal{N}(m_{2}(W_{i}),\Sigma_{02}) \label{eq:samplingXmulti}
\end{align}
for $i=1,...,n$, where $\beta \in \mathcal{B} \subseteq \mathbb{R}^{d_{x}}$, $m = (m_{1},m_{2}) \in \mathcal{M}_{1} \times \prod_{k=1}^{d_{x}}\mathcal{M}_{2,k} \subseteq L^{2}(\mathcal{W})^{d_{x}+1}$, $\xi \in \Xi$ with $\Xi = [\overline{\sigma}^{-2},\underline{\sigma}^{-2}]$ for some $0 < \underline{\sigma} < \overline{\sigma} < \infty$, and $\Sigma_{02}$ is a known positive-definite matrix. The first display imposes the multivariate version of the \cite{robinson1988root} transformation. The second display is included solely for identifiability of $m_{02}$, and, as a result, it is without loss of generality to treat $\Sigma_{02}$ as known (e.g., $\Sigma_{02} = I_{d_{x}}$) as it has no intrinsic meaning (like $\sigma_{02}^{2}$ in the main text). Let $L_{n}(\beta,\xi,m)$ be the (quasi-)likelihood implied by (\ref{eq:samplingYmulti})--(\ref{eq:samplingXmulti}) and let $\ell_{n}(\beta,\xi,m) = \log L_{n}(\beta,\xi,m)$. For the prior, I assume that $(\beta,\xi) \sim \Pi_{\mathcal{B}\times \Xi}$, $m \sim \Pi_{\mathcal{M}}$, and $(\beta,\xi) \independent m$. Consequently, the joint prior is $\Pi = \Pi_{\mathcal{B}\times \Xi} \otimes \Pi_{\mathcal{M}}$. The following extends Assumption \ref{as:prior}.
\begin{assumption}\label{as:priormultivariate}
The following conditions hold:
\begin{enumerate}
\item $\int_{\mathcal{B}\times \mathcal{M}}L_{n}(\beta,\xi,m)d \Pi(\beta,\xi,m) \in (0,\infty)$ a.e. $[P_{0}^{\infty}]$.
\item $\Pi_{\mathcal{B}\times \Xi}$ has a Lebesgue density $\pi_{\mathcal{B}\times \Xi}$ that is continuous and positive over a neighborhood of $(\beta_{0},\xi_{0})$.
\end{enumerate}
\end{assumption}
Assumption \ref{as:priormultivariate} implies that the posterior $\Pi((\beta,\xi,m) \in \cdot | \{(Y_{i},X_{i}',W_{i}')'\}_{i=1}^{n})$ is obtained via Bayes rule,
\begin{align*}
\Pi((\beta,\xi,m) \in A|\{(Y_{i},X_{i},W_{i}')'\}_{i=1}^{n}) = \frac{\int_{A} L_{n}(\beta,\xi,m) d \Pi(\beta,\xi,m)}{\int_{\mathcal{B}\times \mathcal{M}}L_{n}(\beta,\xi,m)d\Pi(\beta,\xi,m)},   
\end{align*}
where $A$ is a measurable set.
\subsubsection{Assumptions on Nuisance Posterior}
The next two assumptions extend Assumptions \ref{as:consistency} and \ref{as:emp_process} to accommodate multivariate $X$ and unknown $\sigma_{01}^{2}$. The latter explains why Assumption \ref{as:emp_processmultivariate} is stated slightly differently to Assumption \ref{as:emp_process} because estimating the scale parameter $\sigma_{01}^{2}$ introduces an additional multiplier empirical process that involves the interaction between $U$ and $m_{1}-m_{01}$. However, the equicontinuity conditions stated in Assumption \ref{as:emp_processmultivariate} nests Assumption \ref{as:emp_process} because $U = \varepsilon_{1} - \varepsilon_{2}'\beta_{0}$. For notation, $B_{n,\mathcal{M}}(m_{0},\delta) = \{m \in \mathcal{M}: \max\{||m_{1}-m_{01}||_{n,2},||m_{2,1}-m_{02,1}||_{n,2},...,||m_{2,d_{x}}-m_{02,d_{x}}||_{n,2}\}< \delta\}$ with $m_{2,k}$, $k \in \{1,...,d_{x}\}$, being components of $m_{2}$, and $||\cdot||_{op}$ is the operator norm of a matrix induced by the Euclidean norm $||\cdot||_{2}$ (i.e., the spectral norm).
\begin{assumption}\label{as:consistencymultivariate}
There exists a sequence $\{\delta_{n}\}_{n \geq 1}$ such that 
\begin{align*}
\delta_{n} = o(n^{-1/4})    
\end{align*}
and
\begin{align*}
\Pi(B_{n,\mathcal{M}}(m_{0},D\delta_{n})|\{(Y_{i},X_{i},w_{i}')'\}_{i=1}^{n}) \overset{P_{0,YX|W}^{(n)}}{\longrightarrow} 1     
\end{align*}
for $P_{0,W}^{\infty}$-almost every fixed realization $\{w_{i}\}_{i \geq 1}$ of $\{W_{i}\}_{i \geq 1}$ as $n\rightarrow \infty$, where $D>0$ is a large constant.
\end{assumption}
\begin{assumption}\label{as:emp_processmultivariate}
There are sets $\{\mathcal{M}_{n}\}_{n \geq 1}$ such that 
\begin{align*}
&\Pi(\mathcal{M}_{n}|\{(Y_{i},X_{i},w_{i}')'\}_{i=1}^{n}) \overset{P_{0,YX|W}^{(n)}}{\longrightarrow} 1 \\    
\sup_{m \in \mathcal{M}_{n}\cap B_{n,\mathcal{M}}(m_{0},D\delta_{n})}&\left|\frac{1}{\sqrt{n}}\sum_{i=1}^{n}\varepsilon_{i1}(m_{1}(w_{i})-m_{01}(w_{i}))\right| \overset{P_{0,YX|W}^{(n)}}{\longrightarrow} 0 \\
\sup_{m \in \mathcal{M}_{n}\cap B_{n,\mathcal{M}}(m_{0},D\delta_{n})}&\left|\left|\frac{1}{\sqrt{n}}\sum_{i=1}^{n}\varepsilon_{i2}(m_{1}(w_{i})-m_{01}(w_{i}))\right| \right|_{2}\overset{P_{0,YX|W}^{(n)}}{\longrightarrow} 0 \\
\sup_{m \in \mathcal{M}_{n}\cap B_{n,\mathcal{M}}(m_{0},D\delta_{n})}&\left| \left|\frac{1}{\sqrt{n}}\sum_{i=1}^{n}\varepsilon_{i1}(m_{2}(w_{i})-m_{02}(w_{i}))\right| \right|_{2} \overset{P_{0,YX|W}^{(n)}}{\longrightarrow} 0  \\
\sup_{m \in \mathcal{M}_{n}\cap B_{n,\mathcal{M}}(m_{0},D\delta_{n})}&\left| \left|\frac{1}{\sqrt{n}}\sum_{i=1}^{n}\varepsilon_{i2}(m_{2}(w_{i})-m_{02}(w_{i}))'\right| \right|_{op} \overset{P_{0,YX|W}^{(n)}}{\longrightarrow} 0  
\end{align*}
for $P_{0,W}^{\infty}$-almost every fixed realization $\{w_{i}\}_{i \geq 1}$ of $\{W_{i}\}_{i \geq 1}$ as $n\rightarrow \infty$.
\end{assumption}
\subsection{Extensions of Lemma \ref{lem:technical} and \ref{lem:rootn} to Multivariate $X$ and Unknown $\sigma_{01}^{2}$}
This section extends Lemmas \ref{lem:technical} and \ref{lem:rootn} to allow for unknown $\sigma_{01}^{2}$ and multivariate $X$. I start with some notation. Let 
\begin{align*}
\tilde{\ell}_{n}(\beta,\xi,m) = \begin{pmatrix} \tilde{\ell}_{n,1}(\beta,\xi,m) \\
\tilde{\ell}_{n,2}(\beta,\xi,m),
\end{pmatrix}
\end{align*}
where
\begin{align*}
    \tilde{\ell}_{n,1}(\beta,\xi,m) = \frac{\partial}{\partial \beta'}\ell_{n}(\beta,\xi,m) = \xi\sum_{i=1}^{n}(Y_{i}-m_{1}(w_{i})-(X_{i}-m_{2}(w_{i}))'\beta)(X_{i}-m_{2}(w_{i}))
\end{align*}
and
\begin{align*}
  \tilde{\ell}_{n,2}(\beta,\xi,m) =  \frac{\partial}{\partial \xi}\ell_{n}(\beta,\xi,m) = \frac{n}{2\xi}-\frac{1}{2}\sum_{i=1}^{n}(Y_{i}-m_{1}(w_{i})-(X_{i}-m_{2}(w_{i}))'\beta)^{2}.
\end{align*}
Let
\begin{align*}
    \tilde{I}_{n}(\beta,\xi,m) = \begin{pmatrix} \tilde{I}_{n,11}(\beta,\xi,m) & \tilde{I}_{n,12}(\beta,\xi,m) \\
    \tilde{I}_{n,21}(\beta,\xi,m) & \tilde{I}_{n,22}(\beta,\xi,m) \end{pmatrix},
\end{align*}
be a $(d_{x}+1)\times (d_{x}+1)$ symmetric matrix, where
\begin{align*}
    \tilde{I}_{n,11}(\beta,\xi,m) &= -\frac{\partial^{2}}{\partial \beta \partial \beta'}\frac{1}{n}\ell_{n}(\beta,\xi,m) = \xi \frac{1}{n}\sum_{i=1}^{n}(X_{i}-m_{2}(w_{i}))(X_{i}-m_{2}(w_{i}))' \\
   \tilde{I}_{n,12}(\beta,\xi,m) &= -\frac{\partial^{2}}{\partial \beta \partial \xi}\frac{1}{n}\ell_{n}(\beta,\xi,m) = -\frac{1}{n}\sum_{i=1}^{n}(Y_{i}-m_{1}(w_{i})-(X_{i}-m_{2}(w_{i}))'\beta)(X_{i}-m_{2}(w_{i})) \\
   \tilde{I}_{n,22}(\beta,\xi,m) &=-\frac{\partial^{2}}{\partial \xi^{2}}\frac{1}{n}\ell_{n}(\beta,\xi,m) = \frac{1}{2\xi^{2}}.
\end{align*}
The first result extends Lemma \ref{lem:technical} to accommodate multivariate $X$ and unknown $\sigma_{01}^{2}$.
\begin{lemma}\label{lem:technicalmultivariate}
Let $\{\mathcal{M}_{n} \cap B_{n,\mathcal{M}}(m_{0},D\delta_{n})\}_{n=1}^{\infty}$ be the sets defined in Assumptions \ref{as:consistencymultivariate} and \ref{as:emp_processmultivariate}. If Assumption \ref{as:dgpmulti} holds, then the following holds:
\begin{enumerate}
\item The collection of functions $\{\tilde{\ell}_{n}(\beta_{0},\xi_{0},m): m \in \mathcal{M}\}$ satisfies
\begin{align*}
  \sup_{m \in B_{n,\mathcal{M}}(m_{0},D\delta_{n})\cap \mathcal{M}_{n}}\left|\left| \frac{1}{\sqrt{n}}\tilde{\ell}_{n}(\beta_{0},\xi_{0},m) - \frac{1}{\sqrt{n}}\tilde{\ell}_{n}(\beta_{0},\xi_{0},m_{0}) \right| \right|_{2}  \overset{P_{0,YX|W}^{(n)}}{\longrightarrow} 0
\end{align*}
for $P_{0,W}^{\infty}$-almost every fixed realization $\{w_{i}\}_{i \geq 1}$ of $\{W_{i}\}_{i \geq 1}$ as $n\rightarrow \infty$.
\item The collection of functions $\{\tilde{I}_{n,11}(\beta_{0},\xi_{0},m): m \in \mathcal{M}\}$ satisfies
\begin{align*}
 \sup_{m \in B_{n,\mathcal{M}}(m_{0},D\delta_{n})\cap \mathcal{M}_{n}}\left| \left| \tilde{I}_{n,11}(\beta_{0},\xi_{0},m)-\tilde{I}_{n,11}(\beta_{0},\xi_{0},m_{0}) \right| \right|_{op} \overset{P_{0,YX|W}^{(n)}}{\longrightarrow} 0   
\end{align*}
for $P_{0,W}^{\infty}$-almost every fixed realization $\{w_{i}\}_{i \geq 1}$ of $\{W_{i}\}_{i \geq 1}$ as $n\rightarrow \infty$.
\item Let $\underline{c}_{0}=\underline{c}/\sigma_{01}^{2}$ and $\overline{c}_{0} = \overline{c}/\sigma_{01}^{2}$, where $\overline{c}$ and $\underline{c}$ are the constants Part 2 of Assumption \ref{as:dgpmulti}. The following statements hold:
\begin{align*}
   P_{0,YX|W}^{(n)}\left( \inf_{m \in \mathcal{M}_{n} \cap B_{\mathcal{M},\infty}(m_{0},D\delta_{n})}\lambda_{min}(\tilde{I}_{n,11}(\beta_{0},\xi_{0},m))\geq \underline{c}_{0} \right) \rightarrow 1
   \end{align*}
   and
   \begin{align*}
    P_{0,YX|W}^{(n)}\left(\sup_{m \in \mathcal{M}_{n} \cap B_{\mathcal{M},\infty}(m_{0},D\delta_{n})}\lambda_{max}(\tilde{I}_{n,11}(\beta_{0},\xi_{0},m))\leq \overline{c}_{0}\right) \rightarrow 
\end{align*}
for $P_{0,W}^{\infty}$-almost every fixed realization $\{w_{i}\}_{i \geq 1}$ of $\{W_{i}\}_{i \geq 1}$ as $n\rightarrow \infty$.
\item Let $\theta = (\beta',\xi)'$ and let $R_{n}(\theta,m)$ be given by
\begin{align*}
R_{n}(\theta,m) = \ell_{n}(\theta,m)-\ell_{n}(\theta_{0},m)- \frac{1}{\sqrt{n}}\tilde{\ell}_{n}(\theta_{0},m_{0})'\sqrt{n}(\theta-\theta_{0})+\frac{n}{2}(\theta-\theta_{0})'\tilde{I}_{n}(\theta_{0},m_{0})(\theta-\theta_{0}).
\end{align*}
If $M_{n}\rightarrow \infty$ arbitrarily slowly, then
\begin{align*}
\sup_{(\theta,m) \in B_{\mathcal{B}\times \Xi}(\theta_{0},M_{n}/\sqrt{n})\times (B_{\mathcal{M},\infty}(m_{0},D\delta_{n})\cap \mathcal{M}_{n})}\left | R_{n}(\theta,m) \right| \overset{P_{0,YX|W}^{(n)}}{\longrightarrow} 0
\end{align*}
for $P_{0,W}^{\infty}$-almost every fixed realization $\{w_{i}\}_{i \geq 1}$ of $\{W_{i}\}_{i \geq 1}$ as $n\rightarrow \infty$.
\end{enumerate}
\end{lemma}
\begin{proof}
The proof has four steps with each step corresponding to the appropriate part of the lemma.

\textit{Step 1.} Applying the triangle inequality, I know that
\begin{align*}
    \left | \left | \frac{1}{\sqrt{n}}\tilde{\ell}_{n}(\beta_{0},\xi_{0},m) - \frac{1}{\sqrt{n}}\tilde{\ell}_{n}(\beta_{0},\xi_{0},m_{0}) \right | \right|_{2} &\leq      \left | \left | \frac{1}{\sqrt{n}}\tilde{\ell}_{n,1}(\beta_{0},\xi_{0},m) - \frac{1}{\sqrt{n}}\tilde{\ell}_{n,1}(\beta_{0},\xi_{0},m_{0}) \right | \right|_{2} \\
    &\quad + \left | \frac{1}{\sqrt{n}}\tilde{\ell}_{n,2}(\beta_{0},\xi_{0},m) - \frac{1}{\sqrt{n}}\tilde{\ell}_{n,2}(\beta_{0},\xi_{0},m_{0})\right|
\end{align*}
Consequently, it suffices to show that the following holds 
\begin{align}\label{eq:score1}
  \sup_{m \in B_{n,\mathcal{M}}(m_{0},D\delta_{n})\cap \mathcal{M}_{n}}\left | \left | \frac{1}{\sqrt{n}}\tilde{\ell}_{n,1}(\beta_{0},\xi_{0},m) - \frac{1}{\sqrt{n}}\tilde{\ell}_{n,1}(\beta_{0},\xi_{0},m_{0}) \right | \right|_{2} \overset{P_{0,YX|W}^{(n)}}{\longrightarrow} 0  
\end{align}
and
\begin{align}\label{eq:score2}
  \sup_{m \in B_{n,\mathcal{M}}(m_{0},D\delta_{n})\cap \mathcal{M}_{n}}\left | \frac{1}{\sqrt{n}}\tilde{\ell}_{n,2}(\beta_{0},\xi_{0},m) - \frac{1}{\sqrt{n}}\tilde{\ell}_{n,2}(\beta_{0},\xi_{0},m_{0})\right| \overset{P_{0,YX|W}^{(n)}}{\longrightarrow} 0  
\end{align}
for $P_{0,W}^{\infty}$-almost every fixed realization $\{w_{i}\}_{i \geq 1}$ of $\{W_{i}\}_{i \geq 1}$ as $n\rightarrow \infty$. For (\ref{eq:score1}), I directly expand $\tilde{\ell}_{n,1}(\beta_{0},\xi_{0},m)-\tilde{\ell}_{n,1}(\beta_{0},\xi_{0},m_{0})$ and apply the triangle inequality to conclude
\begin{align*}
   &\sup_{m \in B_{n,\mathcal{M}}(m_{0},D\delta_{n})\cap \mathcal{M}_{n}}\left | \left | \frac{1}{\sqrt{n}}\tilde{\ell}_{n,1}(\beta_{0},\xi_{0},m) - \frac{1}{\sqrt{n}}\tilde{\ell}_{n,1}(\beta_{0},\xi_{0},m_{0}) \right | \right|_{2} \\
   &\quad \leq    \sup_{m \in B_{n,\mathcal{M}}(m_{0},D\delta_{n})\cap \mathcal{M}_{n}} \left | \left | \frac{1}{\sqrt{n}}\sum_{i=1}^{n}(m_{1}(w_{i})-m_{01}(w_{i}))\varepsilon_{i2}\right| \right|_{2} \\
   &\quad \quad + \sup_{m \in B_{n,\mathcal{M}}(m_{0},D\delta_{n})\cap \mathcal{M}_{n}} \left | \left | \frac{1}{\sqrt{n}}\sum_{i=1}^{n}(m_{2}(w_{i})-m_{02}(w_{i}))(U_{i}-\varepsilon_{i2}'\beta_{0})\right| \right|_{2} \\
   &\quad \quad +\sup_{m \in B_{n,\mathcal{M}}(m_{0},D\delta_{n})\cap \mathcal{M}_{n}} \left | \left | \frac{1}{\sqrt{n}}\sum_{i=1}^{n}(m_{1}(w_{i})-m_{01}(w_{i}))(m_{2}(w_{i})-m_{02}(w_{i}))\right| \right|_{2} \\
   &\quad \quad+ \sup_{m \in B_{n,\mathcal{M}}(m_{0},D\delta_{n})\cap \mathcal{M}_{n}} \left | \left | \frac{1}{\sqrt{n}}\sum_{i=1}^{n}(m_{2}(w_{i})-m_{02}(w_{i}))(m_{2}(w_{i})-m_{02}(w_{i}))'\beta_{0}\right| \right|_{2} 
\end{align*}
Assumptions \ref{as:consistencymultivariate} and \ref{as:emp_processmultivariate} then imply that
\begin{align*}
\sup_{m \in B_{n,\mathcal{M}}(m_{0},D\delta_{n})\cap \mathcal{M}_{n}}\left | \left | \frac{1}{\sqrt{n}}\tilde{\ell}_{n,1}(\beta_{0},\xi_{0},m) - \frac{1}{\sqrt{n}}\tilde{\ell}_{n,1}(\beta_{0},\xi_{0},m_{0}) \right | \right|_{2} \overset{P_{0,YX|W}^{(n)}}{\longrightarrow} 0
\end{align*}
for $P_{0,W}^{\infty}$-almost every fixed realization $\{w_{i}\}_{i \geq 1}$ of $\{W_{i}\}_{i \geq 1}$. For (\ref{eq:score2}), let
\begin{align*}
U_{i}(m) = Y_{i}-m_{1}(w_{i})-(X_{i}-m_{2}(w_{i}))'\beta_{0}    
\end{align*}
and observe that
\begin{align*}
\frac{1}{\sqrt{n}}\tilde{\ell}_{n,2}(\beta_{0},\xi_{0},m) &=\frac{\sqrt{n}}{2\xi_{0}}- \frac{1}{2\sqrt{n}}\sum_{i=1}^{n}U_{i}^{2}  - \frac{1}{2\sqrt{n}}\sum_{i=1}^{n}\left((m_{1}(w_{i})-m_{01}(w_{i}))+(m_{2}(w_{i})-m_{02}(w_{i}))'\beta_{0}\right)^{2} \\
&\quad +\frac{2}{\sqrt{n}}\sum_{i=1}^{n}U_{i}[(m_{1}(w_{i})-m_{01}(w_{i}))+(m_{2}(w_{i})-m_{02}(w_{i}))'\beta_{0}] \\
&=\frac{1}{\sqrt{n}}\tilde{\ell}_{n,2}(\beta_{0},\xi_{0},m_{0}) - \frac{1}{2\sqrt{n}}\sum_{i=1}^{n}\left((m_{1}(w_{i})-m_{01}(w_{i}))+(m_{2}(w_{i})-m_{02}(w_{i}))'\beta_{0}\right)^{2} \\
&\quad +\frac{2}{\sqrt{n}}\sum_{i=1}^{n}U_{i}[(m_{1}(w_{i})-m_{01}(w_{i}))+(m_{2}(w_{i})-m_{02}(w_{i}))'\beta_{0}],
\end{align*}
where I have used that $U_{i}(m_{0}) = U_{i}$. Assumptions \ref{as:consistencymultivariate} and \ref{as:emp_processmultivariate} then imply that
\begin{align*}
   \sup_{m \in B_{n,\mathcal{M}}(m_{0},D\delta_{n})\cap \mathcal{M}_{n}}\left | \frac{1}{\sqrt{n}}\tilde{\ell}_{n,2}(\beta_{0},\xi_{0},m) - \frac{1}{\sqrt{n}}\tilde{\ell}_{n,2}(\beta_{0},\xi_{0},m_{0})\right| \overset{P_{0,YX|W}^{(n)}}{\longrightarrow} 0     
\end{align*}
for $P_{0,W}^{\infty}$-almost every fixed realization $\{w_{i}\}_{i \geq 1}$ of $\{W_{i}\}_{i \geq 1}$ as $n\rightarrow \infty$. This verifies (\ref{eq:score2}).

\textit{Step 2.} Adding and subtracting $m_{02}(w_{i})$, the block $\tilde{I}_{n,11}(\beta_{0},\xi_{0},m)$ admits the decomposition:
\begin{align*}
 \tilde{I}_{n,11}(\beta_{0},\xi_{0},m) &= \tilde{I}_{n}(\beta_{0},\xi_{0},m_{0}) + \frac{\xi_{0}}{n}\sum_{i=1}^{n}(m_{2}(w_{i})-m_{02}(w_{i}))(m_{2}(w_{i})-m_{02}(w_{i}))' \\
 &\quad -\frac{\xi_{0}}{n}\sum_{i=1}^{n}\varepsilon_{i2}(m_{2}(w_{i})-m_{02}(w_{i}))'-\frac{\xi_{0}}{n}\sum_{i=1}^{n}(m_{2}(w_{i})-m_{02}(w_{i}))\varepsilon_{i2}'
\end{align*}
Consequently, the triangle inequality and the fact that $\lambda_{max}(A) = \lambda_{max}(A')$ reveals
\begin{align*}
   \left | \left |  \tilde{I}_{n,11}(\beta_{0},\xi_{0},m)-\tilde{I}_{n,11}(\beta_{0},\xi_{0},m_{0}) \right | \right |_{op} &\lesssim \left | \left | \frac{1}{n}\sum_{i=1}^{n}(m_{2}(w_{i})-m_{02}(w_{i}))(m_{2}(w_{i})-m_{02}(w_{i}))'\right | \right |_{op} \\
   &\quad + \left | \left | \frac{1}{n}\sum_{i=1}^{n}\varepsilon_{i2}(m_{2}(w_{i})-m_{02}(w_{i}))' \right| \right|_{op}.
\end{align*}
Since $\delta_{n} = o(1)$,
\begin{align*}
     \sup_{m \in \mathcal{M}_{n} \cap B_{n,\mathcal{M}}(m_{0},D\delta_{n})}\left | \left | \frac{1}{n}\sum_{i=1}^{n}(m_{2}(w_{i})-m_{02}(w_{i}))(m_{2}(w_{i})-m_{02}(w_{i}))'\right | \right |_{op}  \longrightarrow 0
\end{align*}
for $P_{0,W}^{\infty}$-almost every fixed realization $\{w_{i}\}_{i \geq 1}$ of $\{W_{i}\}_{i \geq 1}$ as $n\rightarrow \infty$. Moreover, Assumption \ref{as:emp_processmultivariate} implies that
\begin{align*}
 \sup_{m \in \mathcal{M}_{n} \cap B_{n,\mathcal{M}}(m_{0},D\delta_{n})}\left | \left | \frac{1}{n}\sum_{i=1}^{n}\varepsilon_{i2}(m_{2}(w_{i})-m_{02}(w_{i}))'\right | \right |_{op}     \overset{P_{0,YX|W}^{(n)}}{\longrightarrow } 0
\end{align*}
for $P_{0,W}^{\infty}$-almost every fixed realization $\{w_{i}\}_{i \geq 1}$ of $\{W_{i}\}_{i \geq 1}$ as $n\rightarrow \infty$. This establishes Part 2.

\textit{Step 3.} I first show the result for the minimum eigenvalue. Since 
\begin{align*}
\lambda_{min}(\tilde{I}_{n,11}(\beta_{0},\xi_{0},m)) = ||\tilde{I}_{n,11}^{-1}(\beta_{0},\xi_{0},m)||_{op},    
\end{align*} 
it suffices to show that
\begin{align}\label{eq:convergenceofinverses}
\sup_{m \in \mathcal{M}_{n} \cap B_{n,\mathcal{M}}(m_{0},D\delta_{n})} \left| \left | \tilde{I}_{n,11}^{-1}(\beta_{0},\xi_{0},m) - \tilde{I}_{n,11}^{-1}(\beta_{0},\xi_{0},m_{0}) \right| \right|_{op} \overset{P_{0,YX|W}^{(n)}}{\longrightarrow } 0   
\end{align}
conditionally given $P_{0,W}^{\infty}$-almost every realization $\{w_{i}\}_{i \geq 1}$ of $\{W_{i}\}_{i \geq 1}$. Indeed, Parts 2 and 4 of Assumption \ref{as:dgpmulti}, Chebyshev's inequality, and the strong law of large numbers implies that
\begin{align*}
    \left| \left| \tilde{I}_{n,11}^{-1}(\beta_{0},\xi_{0},m_{0}) \right| \right|_{op} \geq \underline{c}_{0}
\end{align*}
with $P_{0,YX|W}^{(n)}$-probability approaching one, conditionally given $P_{0,W}^{\infty}$-almost every realization $\{w_{i}\}_{i \geq 1}$ of $\{W_{i}\}_{i \geq 1}$. The convergence (\ref{eq:convergenceofinverses}) holds by Part 2 of the lemma and the continuous mapping theorem. Next, I show the result for the maximum eigenvalue. Part 4 of Assumption \ref{as:dgpmulti} and Chebyshev's inequality implies 
\begin{align*}
||\tilde{I}_{n,11}(\beta_{0},\xi_{0},m_{0})||_{op} \leq \overline{c}_{0}    
\end{align*}
with $P_{0,YX|W}^{(n)}$-probability approaching one for $P_{0,W}^{\infty}$-almost every fixed realization $\{w_{i}\}_{i \geq 1}$ of $\{W_{i}\}_{i \geq 1}$ as $n\rightarrow \infty$. Consequently, it follows by Part 2 of this lemma, the triangle inequality, and the fact that $||A||_{op} = \lambda_{max}(A)$ that with $P_{0,YX|W}^{(n)}$-probability approaching one, $$\sup_{m \in \mathcal{M}_{n} \cap B_{n,\mathcal{M}}(m_{0},D\delta_{n})} \lambda_{max}\left(\tilde{I}_{n,11}(\beta_{0},\xi_{0},m)\right) \leq \overline{c}_{0}$$ conditionally given $P_{0,W}^{\infty}$-almost every realization $\{w_{i}\}_{i \geq 1}$ of $\{W_{i}\}_{i \geq 1}$ as $n\rightarrow \infty$.

\textit{Step 4.} Let $\theta = (\beta',\xi)'$. Applying a second-order Taylor expansion, I know that
\begin{align*}
    \ell_{n}(\theta,m) - \ell_{n}(\theta_{0},m) = \sqrt{n}(\theta-\theta_{0})'\frac{1}{\sqrt{n}}\tilde{\ell}_{n}(\theta_{0},m)-\frac{1}{2}n(\theta-\theta_{0})'\tilde{I}_{n}(\tilde{\theta}_{n},m)(\theta-\theta_{0}),
\end{align*}
where $\tilde{\theta}_{n}$ is on the line segment joining $\theta$ and $\theta_{0}$. Consequently, I want to show that
\begin{align}\label{eq:taylor1}
 \sup_{(\theta,m) \in B_{\mathcal{B}\times \Xi}(\theta_{0},\frac{M_{n}}{\sqrt{n}}) \times (\mathcal{M}_{n} \cap B_{n,\mathcal{M}}(m_{0},D\delta_{n}))}\left|\sqrt{n}(\theta-\theta_{0})'\left(\frac{1}{\sqrt{n}}\tilde{\ell}_{n}(\theta_{0},m) - \frac{1}{\sqrt{n}}\tilde{\ell}_{n}(\theta_{0},m)\right)\right| = o_{P_{0,YX|W}^{(n)}}(1)
\end{align}
and
\begin{align}\label{eq:taylor2}
\sup_{(\theta,m) \in B_{\mathcal{B}\times \Xi}(\theta_{0},\frac{M_{n}}{\sqrt{n}}) \times (\mathcal{M}_{n} \cap B_{n,\mathcal{M}}(m_{0},D\delta_{n}))}\left|\frac{1}{2}n(\theta-\theta_{0})'\left(\tilde{I}_{n}(\tilde{\theta}_{n},m)-\tilde{I}_{n}(\theta_{0},m_{0})\right)(\theta-\theta_{0})\right| = o_{P_{0,YX|W}^{(n)}}(1)
\end{align}
conditionally given $P_{0,W}^{\infty}$-almost every sequence $\{w_{i}\}_{i \geq 1}$ of $\{W_{i}\}_{i \geq 1}$.
For (\ref{eq:taylor1}), Part 1 of the lemma and the fact that $M_{n} \rightarrow \infty$ arbitrarily slowly as $n\rightarrow \infty$ implies that
\begin{align*}
    &\sup_{(\theta,m) \in B_{\mathcal{B}\times \Xi}(\theta_{0},\frac{M_{n}}{\sqrt{n}}) \times (\mathcal{M}_{n} \cap B_{n,\mathcal{M}}(m_{0},D\delta_{n}))}\left|\sqrt{n}(\theta-\theta_{0})'\frac{1}{\sqrt{n}}\tilde{\ell}_{n}(\theta_{0},m) - \sqrt{n}(\theta-\theta_{0})'\frac{1}{\sqrt{n}}\tilde{\ell}_{n}(\theta_{0},m_{0})\right| \\
    &\quad \leq M_{n} \sup_{m\in \mathcal{M}_{n} \cap B_{n,\mathcal{M}}(m_{0},D\delta_{n})}\left| \frac{1}{\sqrt{n}}\tilde{\ell}_{n}(\theta_{0},m)-\frac{1}{\sqrt{n}}\tilde{\ell}_{n}(\theta_{0},m_{0})\right| \\
    &\quad = o_{P_{0,YX|W}^{(n)}}(1)
\end{align*}
for $P_{0,W}^{\infty}$-almost every fixed realization $\{w_{i}\}_{i \geq 1}$ of $\{W_{i}\}_{i \geq 1}$ as $n\rightarrow \infty$. For (\ref{eq:taylor2}), I apply the Cauchy-Schwarz inequality, $||Ax||_{2} \leq ||A||_{op} ||x||_{2}$ for a matrix $A$ and Euclidean vector $x$, and the triangle inequality to conclude that
\begin{align*}
    &\sup_{(\theta,m) \in B_{\mathcal{B}\times \Xi}(\theta_{0},\frac{M_{n}}{\sqrt{n}}) \times (\mathcal{M}_{n} \cap B_{n,\mathcal{M}}(m_{0},D\delta_{n}))}\left|\frac{1}{2}n(\theta-\theta_{0})'\left(\tilde{I}_{n}(\tilde{\theta}_{n},m)-\tilde{I}_{n}(\theta_{0},m_{0})\right)(\theta-\theta_{0})\right| \\
    &\quad \leq M_{n}^{2} \sup_{(\theta,m) \in B_{\mathcal{B}\times \Xi}(\theta_{0},\frac{M_{n}}{\sqrt{n}}) \times (\mathcal{M}_{n} \cap B_{n,\mathcal{M}}(m_{0},D\delta_{n}))}\left|\left|\tilde{I}_{n}(\tilde{\theta}_{n},m)-\tilde{I}_{n}(\theta_{0},m_{0})\right| \right|_{op} \\
    &\quad \leq M_{n}^{2}\sup_{(\theta,m) \in B_{\mathcal{B}\times \Xi}(\theta_{0},\frac{M_{n}}{\sqrt{n}}) \times (\mathcal{M}_{n} \cap B_{n,\mathcal{M}}(m_{0},D\delta_{n}))}\left|\left|\tilde{I}_{n}(\tilde{\theta}_{n},m)-\tilde{I}_{n}(\tilde{\theta}_{n},m_{0})\right| \right|_{op}\\
    &\quad \quad+ M_{n}^{2}\sup_{\theta \in B_{\mathcal{B}\times \Xi}(\theta_{0},\frac{M_{n}}{\sqrt{n}})}\left|\left|\tilde{I}_{n}(\tilde{\theta}_{n},m_{0})-\tilde{I}_{n}(\theta_{0},m_{0})\right| \right|_{op}
\end{align*}
Since $\sup_{\theta \in B_{\mathcal{B}\times \Xi}(\theta_{0},M_{n}/\sqrt{n})}||\tilde{\theta}_{n}||_{2} = O(1)$ and the elements of $\tilde{I}_{n}(\theta,m)-\tilde{I}_{n}(\theta,m_{0})$ are linear in $\theta$, Part 2 of the lemma implies the following convergence
\begin{align*}
    \sup_{(\theta,m) \in B_{\mathcal{B}\times \Xi}(\theta_{0},\frac{M_{n}}{\sqrt{n}}) \times (\mathcal{M}_{n} \cap B_{n,\mathcal{M}}(m_{0},D\delta_{n}))}\left|\left|\tilde{I}_{n}(\tilde{\theta}_{n},m)-\tilde{I}_{n}(\tilde{\theta}_{n},m_{0})\right| \right|_{op} = o_{P_{0,YX|W}^{(n)}}(1)
\end{align*}
for $P_{0,W}^{\infty}$-almost every fixed realization $\{w_{i}\}_{i \geq 1}$ of $\{W_{i}\}_{i \geq 1}$ as $n\rightarrow \infty$. Moreover, since the mapping $\theta \mapsto \tilde{I}_{n}(\theta,m_{0})$ is continuous over $B_{\mathcal{B}\times \Xi}(\theta_{0},M_{n}/\sqrt{n})$,
\begin{align*}
  \sup_{\theta \in B_{\mathcal{B}\times \Xi}(\theta_{0},\frac{M_{n}}{\sqrt{n}}) }  \left|\left|\tilde{I}_{n}(\tilde{\theta}_{n},m_{0})-\tilde{I}_{n}(\theta_{0},m_{0})\right| \right|_{op} = o_{P_{0,YX|W}^{(n)}}(1)
\end{align*}
for $P_{0,W}^{\infty}$-almost every fixed realization $\{w_{i}\}_{i \geq 1}$ of $\{W_{i}\}_{i \geq 1}$ as $n\rightarrow \infty$. Consequently, I conclude that if $M_{n} \rightarrow \infty$ arbitrarily slowly as $n\rightarrow \infty$, then
\begin{align*}
    \sup_{(\theta,m) \in B_{\mathcal{B}\times \Xi}(\theta_{0},\frac{M_{n}}{\sqrt{n}}) \times (\mathcal{M}_{n} \cap B_{n,\mathcal{M}}(m_{0},D\delta_{n}))}\left|\frac{1}{2}n(\theta-\theta_{0})'\left(\tilde{I}_{n}(\tilde{\theta}_{n},m)-\tilde{I}_{n}(\theta_{0},m_{0})\right)(\theta-\theta_{0})\right|  \overset{P_{0,YX|W}^{(n)}}{\longrightarrow } 0
\end{align*}
conditionally given $P_{0,W}^{\infty}$-almost every realization $\{w_{i}\}_{i \geq 1}$ of $\{W_{i}\}_{i \geq 1}$ as $n\rightarrow \infty$. The proof is complete.
\end{proof}
\begin{lemma}\label{lem:rootnmultivariate}
Let $\{\mathcal{M}_{n} \cap B_{n,\mathcal{M}}(m_{0},D\delta_{n})\}_{n=1}^{\infty}$ be the sets defined in Assumptions \ref{as:consistencymultivariate} and \ref{as:emp_processmultivariate}. If Assumptions \ref{as:dgpmulti}, \ref{as:priormultivariate}, \ref{as:consistencymultivariate}, and \ref{as:emp_processmultivariate} hold, then for every $M_{n}\rightarrow \infty$,
\begin{align*}
    \sup_{m \in \mathcal{M}_{n} \cap B_{n,\mathcal{M}}(m_{0},D\delta_{n})}\Pi\left((\beta,\xi) \in B_{\mathcal{B}\times \Xi}((\beta_{0},\xi_{0}),M_{n}/\sqrt{n})^{c}|\{(Y_{i},X_{i},W_{i}')'\}_{i=1}^{n},m\right) \overset{P_{0,YX|W}^{(n)}}{\longrightarrow} 0
\end{align*}
conditionally given $P_{0,W}^{\infty}$-almost every realization $\{w_{i}\}_{i \geq 1}$ of $\{W_{i}\}_{i \geq 1}$ as $n\rightarrow \infty$, where $B_{\mathcal{B}\times \Xi}((\beta_{0},\xi_{0}),M_{n}/\sqrt{n})$ is a Euclidean ball centered at $(\beta_{0},\xi_{0})$ with radius $M_{n}/\sqrt{n}$.
\end{lemma}
\begin{proof}
The proof has two steps. The first step derives a uniform upper bound for the numerator of the conditional posterior. The second step derives a uniform lower bound for the denominator of the conditional posterior that when combined with the first step implies the result. The argument is conceptually similar to the proof of Lemma \ref{lem:rootn}.

\textit{Step 1.} I start with a unifom upper bound for the numerator of the conditional posterior. Specifically, I show that there exists a constant $C>0$ such that
\begin{align*}
   P_{0,YX|W}^{(n)}\left( \sup_{m \in \mathcal{M}_{n} \cap  B_{n,\mathcal{M}}(m_{0},D\delta_{n})}\int_{B_{\mathcal{B}\times \Xi}\left((\beta_{0},\xi_{0}),\frac{M_{n}}{\sqrt{n}}\right)^{c}}\exp\left(\ell_{n}(\beta,\xi,m)-\ell_{n}(\beta_{0},\xi_{0},m)\right)d \Pi_{\mathcal{B}\times \Xi}(\beta,\xi) \leq \exp(-CM_{n}^{2}) \right) \rightarrow 1
\end{align*}
for $P_{0,W}^{\infty}$-almost every fixed realization $\{w_{i}\}_{i \geq 1}$ of $\{W_{i}\}_{i\geq 1}$ as $n\rightarrow \infty$. Adding/subtracting $\ell_{n}(\beta_{0},\xi,m)$ to obtain
\begin{align*}
    \ell_{n}(\beta,\xi,m)-\ell_{n}(\beta_{0},\xi_{0},m) = \ell_{n}(\beta,\xi,m) - \ell_{n}(\beta_{0},\xi,m) + \ell_{n}(\beta_{0},\xi,m)-\ell_{n}(\beta_{0},\xi_{0},m)
\end{align*}
reveals that it suffices to show that there exists constants $C_{1}>0$ and $C_{2} > 0$ such that
\begin{align}\label{eq:firstterm}
P_{0,YX|W}^{(n)}\left(\sup_{m \in \mathcal{M}_{n} \cap B_{n,\mathcal{M}}(m_{0},D\delta_{n})}\int_{B_{\mathcal{B}\times \Xi}\left((\beta_{0},\xi_{0}),\frac{M_{n}}{\sqrt{n}}\right)^{c}}\exp\left(\ell_{n}(\beta,\xi,m)-\ell_{n}(\beta_{0},\xi,m)\right)d \Pi_{\mathcal{B}\times \Xi}(\beta,\xi) \leq \exp(-C_{1}M_{n}^{2})\right) \rightarrow 1
\end{align}
and
\begin{align}\label{eq:secondterm}
P_{0,YX|W}^{(n)}\left(\sup_{m \in \mathcal{M}_{n} \cap B_{n,\mathcal{M}}(m_{0},D\delta_{n})}\int_{B_{\mathcal{B}\times \Xi}\left((\beta_{0},\xi_{0}),\frac{M_{n}}{\sqrt{n}}\right)^{c}}\exp\left(\ell_{n}(\beta_{0},\xi,m)-\ell_{n}(\beta_{0},\xi_{0},m)\right)d \Pi_{\mathcal{B}\times \Xi}(\beta,\xi) \leq \exp(-C_{2}M_{n}^{2}) \right) \rightarrow 1  
\end{align}
conditionally given $P_{0,W}^{\infty}$-almost every realization $\{w_{i}\}_{i \geq 1}$ of $\{W_{i}\}_{i \geq 1}$ as $n\rightarrow \infty$. For the bound (\ref{eq:firstterm}), an expansion of the log-likelihood implies
\begin{align*}
    &\ell_{n}(\beta,\xi,m)-\ell_{n}(\beta_{0},\xi,m) \\
    &\quad= \frac{\xi}{\xi_{0}}\left(\sqrt{n}(\beta-\beta_{0})'\frac{1}{\sqrt{n}}\tilde{\ell}_{n,1}(\beta_{0},\xi_{0},m)- \frac{1}{2}\sqrt{n}(\beta-\beta_{0})'\tilde{I}_{n,11}(\beta_{0},\xi_{0},m)\sqrt{n}(\beta-\beta_{0}) \right) \\
    &\quad =\frac{\xi}{\xi_{0}}\left(\ell_{n}(\beta,\xi_{0},m)-\ell_{n}(\beta_{0},\xi_{0},m)\right)
\end{align*}
Consequently, I can apply a symmetric argument to that of Step 2 of Lemma \ref{lem:rootn} to conclude that
\begin{align*}
  P_{0,YX|W}^{(n)}\left(  \sup_{m \in \mathcal{M}_{n} \cap B_{n,\mathcal{M}}(m_{0},D\delta_{n})}\sup_{\beta \in \mathcal{B}: ||\beta-\beta_{0}||_{2} \geq \frac{M_{n}}{\sqrt{n}}}\exp\left(\ell_{n}(\beta,\xi_{0},m)-\ell_{n}(\beta_{0},\xi_{0},m)\right) \leq \exp\left(-\frac{\underline{c}_{0}}{4}M_{n}^{2}\right)\right)\rightarrow 1
\end{align*}
for $P_{0,W}^{\infty}$-almost every fixed realization $\{w_{i}\}_{i \geq 1}$ of $\{W_{i}\}_{i \geq 1}$ as $n\rightarrow \infty$, where $\underline{c}_{0}$ is the constant defined in Part 3 of Lemma \ref{lem:technicalmultivariate}. Since $\Xi= [\overline{\sigma}^{-2},\underline{\sigma}^{-2}]$, I able to further conclude that there exists a constant $C_{1}>0$ such that
\begin{align*}
    P_{0,YX|W}^{(n)}\left(  \sup_{(\beta,\xi) \in B((\beta_{0},\xi_{0}),M_{n}/\sqrt{n})^{c}}\exp( \ell_{n}(\beta,\xi,m)-\ell_{n}(\beta_{0},\xi,m) ) \lesssim \exp(-C_{1}M_{n}^{2})\right) \rightarrow 1
\end{align*}
 for $P_{0,W}^{\infty}$ almost every fixed realization $\{w_{i}\}_{i \geq 1}$ of $\{W_{i}\}_{i \geq 1}$ as $n\rightarrow \infty$. For (\ref{eq:secondterm}), expanding the log-likelihood reveals that
\begin{align*}
    \ell_{n}(\beta_{0},\xi,m)-\ell_{n}(\beta_{0},\xi_{0},m) = \frac{n}{2}(\log \xi-\log \xi_{0})-\frac{n}{2}\left(\xi-\xi_{0}\right)\frac{1}{n}\sum_{i=1}^{n}\left(Y_{i}-m_{1}(w_{i})-(X_{i}-m_{2}(w_{i}))'\beta_{0}\right)^{2}
\end{align*}
A second order Taylor expansion of $x \mapsto \log x$ reveals that
\begin{align*}
    \frac{n}{2}(\log \xi - \log \xi_{0}) = \frac{n}{2\xi_{0}}(\xi-\xi_{0})-\frac{n}{2\tilde{\xi}^{2}_{n}}(\xi-\xi_{0})^{2},
\end{align*}
where $\tilde{\xi}_{n}$ on the line segment joining $\xi$ and $\xi_{0}$. Since $\tilde{\xi}_{n} \geq \overline{\sigma}^{-2}$ uniformly over $\Xi$, it follows that
\begin{align*}
    \ell_{n}(\beta_{0},\xi,m)-\ell_{n}(\beta_{0},\xi_{01},m) \leq  \frac{\sqrt{n}}{2}\left(\frac{1}{\xi_{0}}-\frac{1}{n}\sum_{i=1}^{n}U_{i}^{2}(m)\right)\sqrt{n}(\xi-\xi_{0})-\frac{1}{2\overline{\sigma}^{2}}(\sqrt{n}(\xi-\xi_{0}))^{2},
\end{align*}
where $U_{i}(m) =Y_{i}-m_{1}(w_{i})-(X_{i}-m_{2}(w_{i}))'\beta_{0}$ for $i=1,...,n$. The proof of Part 1 of Lemma \ref{lem:technicalmultivariate} shows that
\begin{align*}
    \sup_{m \in \mathcal{M}_{n} \cap B_{n,\mathcal{M}}(m_{0},D\delta_{n})}\left|\frac{1}{\sqrt{n}}\sum_{i=1}^{n}U_{i}^{2} - \frac{1}{\sqrt{n}}\sum_{i=1}^{n}U_{i}^{2}(m)\right| \overset{P_{0,YX|W}^{(n)}}{\longrightarrow } 0
\end{align*}
for $P_{0,W}^{\infty}$-almost every fixed realization $\{w_{i}\}_{i \geq 1}$ of $\{W_{i}\}_{i \geq 1}$ as $n\rightarrow \infty$. Moreover, Chebyshev's inequality reveals
\begin{align*}
\frac{1}{n}\sum_{i=1}^{n}U_{i}^{2}-\frac{1}{\xi_{0}} = O_{P_{0,YX|W}^{(n)}}\left(\frac{1}{\sqrt{n}}\right)    
\end{align*}
for $P_{0,W}^{\infty}$-almost every fixed realization $\{w_{i}\}_{i \geq 1}$ of $\{W_{i}\}_{i \geq 1}$ as $n\rightarrow \infty$. Consequently,
\begin{align*}
  P_{0,YX|W}^{(n)}\left(   \sup_{m \in \mathcal{M}_{n} \cap B_{n,\mathcal{M}}(m_{0},D\delta_{n})}\left|\frac{1}{\sqrt{n}}\sum_{i=1}^{n}\left(U_{i}^{2}(m)-\frac{1}{\xi_{0}}\right)\right| \leq \frac{M_{n}}{2\overline{\sigma}^{2}}\right) \longrightarrow 1
\end{align*}
for $P_{0,W}^{\infty}$-almost every fixed realization $\{w_{i}\}_{i \geq 1}$ of $\{W_{i}\}_{i \geq 1}$ as $n\rightarrow \infty$. Consequently, with $P_{0,YX|W}^{(n)}$-probability approaching one, the following holds
\begin{align*}
  \sup_{(\xi,m): |\xi-\xi_{0}| \geq \frac{M_{n}}{\sqrt{n}}, m \in \mathcal{M}_{n} \cap B_{n,\mathcal{M}}(m_{0},D\delta_{n})}\exp\left(\ell_{n}(\beta_{0},\xi,m)-\ell_{n}(\beta_{0},\xi_{0},m)\right) \leq \exp\left(\sup_{t \in \mathbb{R}: t \geq M_{n}}\left\{\frac{M_{n}}{4\overline{\sigma}^{2}}t-\frac{1}{2\overline{\sigma}^{2}}t^{2}\right\}\right),
\end{align*}
where $t = |\sqrt{n}(\xi-\xi_{0})|$. Solving the maximization problem to obtain $t = M_{n}$, I conclude that there is a constant $C_{2}>0$ such that
\begin{align*}
  P_{0,YX|W}^{(n)}\left(  \sup_{m \in \mathcal{M}_{n} \cap B_{n,\mathcal{M}}(m_{0},D\delta_{n})}\int_{B_{\mathcal{B}\times \Xi}\left((\beta_{0},\xi_{0}),\frac{M_{n}}{\sqrt{n}}\right)^{c}}\exp\left(\ell_{n}(\beta_{0},\xi,m)-\ell_{n}(\beta_{0},\xi_{0},m)\right)d \Pi_{\mathcal{B}\times \Xi}(\beta,\xi) \leq \exp(-C_{2}M_{n}^{2})\right) \rightarrow 1
\end{align*}
for $P_{0,W}^{\infty}$-almost every fixed realization $\{w_{i}\}_{i \geq 1}$ of $\{W_{i}\}_{i \geq 1}$ as $n\rightarrow \infty$. Letting $C = C_{1}+C_{2}$, I conclude that
\begin{align*}
   P_{0,YX|W}^{(n)}\left(  \sup_{m \in \mathcal{M}_{n} \cap B_{n,\mathcal{M}}(m_{0},D\delta_{n})}\int_{B_{\mathcal{B}\times \Xi}\left((\beta_{0},\xi_{0}),\frac{M_{n}}{\sqrt{n}}\right)^{c}}\exp\left(\ell_{n}(\beta,\xi,m)-\ell_{n}(\beta_{0},\xi_{0},m)\right)d \Pi_{\mathcal{B}\times \Xi}(\beta,\xi) \leq \exp(-CM_{n}^{2})\right) \rightarrow 1
\end{align*}
 for $P_{0,W}^{\infty}$-almost every fixed realization $\{w_{i}\}_{i \geq 1}$ of $\{W_{i}\}_{i \geq 1}$ as $n\rightarrow \infty$.

\textit{Step 2.} Given the result of Step 1, it is sufficient to show that
\begin{align*}
    P_{0,YX|W}^{(n)}\left(\int_{\mathcal{B} \times \Xi}\exp\left(\ell_{n}(\beta,\xi,m)-\ell_{n}(\beta_{0},\xi_{0},m)\right)d \Pi_{\mathcal{B}\times \Xi}(\beta,\xi) \geq \exp\left(-\frac{C}{2}M_{n}^{2}\right)\right) \longrightarrow 1
\end{align*}
for $P_{0,W}^{\infty}$-almost every fixed realization $\{w_{i}\}_{i \geq 1}$ of $\{W_{i}\}_{i \geq 1}$ as $n\rightarrow \infty$. Let $\theta = (\beta',\xi)'$, let $s = \sqrt{n}(\theta-\theta_{0})$, and let $\mathcal{S}$ be the image of $\mathcal{B} \times \Xi$ under the mapping $\theta \mapsto s$. Applying the change of variables $\theta \mapsto s$, the posterior normalizing constant satisfies
\begin{align*}
    \int_{\mathcal{B} \times \Xi}\exp\left(\ell_{n}(\beta,\xi,m)-\ell_{n}(\beta_{0},\xi_{0},m)\right)d \Pi(\beta,\sigma_{1}^{2})  &=\int_{\mathcal{S}}\exp\left(\ell_{n}^{*}(s,m)-\ell_{n}^{*}(0,m)\right)d\Pi_{\mathcal{S}}(s) \\
    & \geq \int_{\mathcal{S}_{\tilde{C}}}\exp\left(\ell_{n}^{*}(s,m)-\ell_{n}^{*}(0,m)\right)d\Pi_{\mathcal{S}}(s) \\
    &\geq K\int_{\mathbb{R}^{d_{x}+1}}\exp\left(\ell_{n}^{*}(s,m)-\ell_{n}^{*}(0,m)\right)\pi_{\mathcal{S}_{\tilde{C}}}(s)ds,
\end{align*}
where $\ell_{n}^{*}(s,m) = \ell_{n}(\theta_{0} + s/\sqrt{n},m)$, $\mathcal{S}_{\tilde{C}} = \{s \in \mathcal{S}: ||s||_{2} \leq \tilde{C}\}$ for $\tilde{C}>0$, $\pi_{\mathcal{S}_{\tilde{C}}}(s) = \mathbf{1}\{s \in \mathcal{S}_{\tilde{C}}\}/\lambda(\mathcal{S}_{\tilde{C}})$, and $K=\inf_{s \in \mathcal{S}_{\tilde{C}}}\pi_{\mathcal{S}}(s)/\lambda(\mathcal{S}_{\tilde{C}})$. Part 2 of Assumption \ref{as:priormultivariate} guarantees that $K \in (0,\infty)$. Consequently, the following inequality holds
\begin{align*}
    &P_{0,YX|W}^{(n)}\left(\int_{\mathcal{B} \times \Xi}\exp\left(\ell_{n}(\beta,\xi,m)-\ell_{n}(\beta_{0},\xi_{0},m)\right)d \Pi_{\mathcal{B}\times \Xi}(\beta,\xi) \leq \exp\left(-\frac{C}{2}M_{n}^{2}\right)\right) \\
    &\leq P_{0,YX|W}^{(n)}\left(\int_{\mathbb{R}^{d_{x}+1}}\exp\left(\ell_{n}^{*}(s,m)-\ell_{n}^{*}(0,m)\right)\pi_{\mathcal{S}_{\tilde{C}}}(s)ds \leq K^{-1}\exp\left(-\frac{C}{2}M_{n}^{2}\right)\right).
\end{align*}
for $P_{0,W}^{\infty}$-almost every fixed realization $\{w_{i}\}_{i \geq 1}$ of $\{W_{i}\}_{i \geq 1}$. Letting $s = (s_{1}',s_{2})'$ with $s_{1}$ corresponding to $\beta$ and $s_{2}$ corresponding to $\xi$, I know that
\begin{align*}
    \ell_{n}^{*}(s,m)-\ell_{n}^{*}(0,m) = \left(\ell_{n}^{*}(s,m)-\ell_{n}^{*}((0,s_{2}),m)\right)+ \left(\ell_{n}^{*}((0,s_{2}),m)-\ell_{n}^{*}(0,m)\right)
\end{align*}
Taylor expansions of each term reveals that
\begin{align*}
\ell_{n}^{*}(s,m)-\ell_{n}^{*}((0,s_{2}),m) = s_{1}'\frac{1}{\sqrt{n}}\tilde{\ell}_{n,1}\left(\beta_{0},\xi_{0}+\frac{s_{2}}{\sqrt{n}},m\right)-\frac{1}{2}s_{1}'\tilde{I}_{n,11}\left(\beta_{0},\xi_{0}+\frac{s_{2}}{\sqrt{n}},m\right)s_{1}    
\end{align*}
and
\begin{align*}
  \ell_{n}^{*}((0',s_{2})',m)-\ell_{n}^{*}((0,s_{2})',m) &= \frac{n}{2}\left(\log \left(\xi_{0}+ \frac{s_{2}}{\sqrt{n}}\right)-\log \left(\xi_{0}\right)\right)-s_{2}\frac{1}{2\sqrt{n}}\sum_{i=1}^{n}U_{i}^{2}(m) \\
    &=s_{2}\frac{\sqrt{n}}{2}\left(\frac{1}{\xi_{0}}-\frac{1}{n}\sum_{i=1}^{n}U_{i}^{2}(m)\right)-\frac{1}{2\xi_{0}^{2}}s_{2}^{2} + O\left(\frac{s_{2}^{3}}{\sqrt{n}}\right)    
\end{align*}
Using the definition of $\tilde{\ell}_{n,1}(\beta,\xi,m)$ and $\tilde{I}_{n,11}(\beta,\xi,m)$, Lemma \ref{lem:technicalmultivariate}, and that $||s||_{2} \leq \tilde{C}$, I know that the following holds:
\begin{align*}
    \sup_{s \in \mathcal{S}_{\tilde{C}}}\sup_{m \in \mathcal{M}_{n} \cap B(m_{0},D\delta_{n})}\left|s_{1}'\frac{1}{\sqrt{n}}\tilde{\ell}_{n,1}\left(\beta_{0},\xi_{0}+\frac{s_{2}}{\sqrt{n}},m\right)- s_{1}'\frac{1}{\sqrt{n}}\tilde{\ell}_{n,1}(\beta_{0},\xi_{0},m_{0})\right| = o_{P_{0,YX|W}^{(n)}}(1)
\end{align*}
for $P_{0,W}^{\infty}$-almost every fixed realization $\{w_{i}\}_{i \geq 1}$ of $\{W_{i}\}_{i \geq 1}$ as $n\rightarrow \infty$, and, as a result,
\begin{align*}
  P_{0,YX|W}^{(n)}\left(  \inf_{s \in \mathcal{S}_{\tilde{C}}}\inf_{m \in \mathcal{M}_{n} \cap B_{n,\mathcal{M}}(m_{0},D\delta_{n})}s_{1}'\frac{1}{\sqrt{n}}\tilde{\ell}_{n,1}\left(\beta_{0},\xi_{0}+\frac{s_{2}}{\sqrt{n}},m\right) \geq - \frac{C}{16}M_{n}^{2} + s_{1}'\frac{1}{\sqrt{n}}\tilde{\ell}_{n,1}(\beta_{0},\xi_{0},m_{0})\right) \longrightarrow 1
\end{align*}
for $P_{0,W}^{\infty}$-almost every fixed realization $\{w_{i}\}_{i \geq 1}$ of $\{W_{i}\}_{i \geq 1}$ as $n\rightarrow \infty$. Similarly,
\begin{align*}
 \sup_{s \in \mathcal{S}_{\tilde{C}}}\sup_{m \in \mathcal{M}_{n} \cap B(m_{0},D\delta_{n})}\left| \frac{\sqrt{n}}{2}\left(\frac{1}{\xi_{0}}-\frac{1}{n}\sum_{i=1}^{n}U_{i}^{2}(m)\right) - s_{2}\frac{1}{\sqrt{n}}\tilde{\ell}_{n,2}(\beta_{0},\xi_{0},m_{0})\right|=o_{P_{0,YX|W}^{(n)}}(1)
\end{align*}
for $P_{0,W}^{\infty}$-almost every fixed realization $\{w_{i}\}_{i \geq 1}$ of $\{W_{i}\}_{i \geq 1}$ as $n\rightarrow \infty$, and, as a result,
\begin{align*}
 P_{0,YX|W}^{(n)}\left(  \inf_{s \in \mathcal{S}_{\tilde{C}}}\inf_{m \in \mathcal{M}_{n} \cap B_{n,\mathcal{M}}(m_{0},D\delta_{n})}s_{2}\frac{1}{\sqrt{n}}\tilde{\ell}_{n,2}\left(\beta_{0},\xi_{0},m\right) \geq - \frac{C}{16}M_{n}^{2} + s_{2}\frac{1}{\sqrt{n}}\tilde{\ell}_{n,2}(\beta_{0},\xi_{0},m_{0})\right) \longrightarrow 1    
\end{align*}
for $P_{0,W}^{\infty}$-almost every fixed realization $\{w_{i}\}_{i \geq 1}$ of $\{W_{i}\}_{i \geq 1}$ as $n\rightarrow \infty$. Moreover, Lemma \ref{lem:technicalmultivariate} implies that 
\begin{align*}
    P_{0,YX|W}^{(n)}\left(  \sup_{s \in \mathcal{S}_{\tilde{C}}}\sup_{m \in \mathcal{M}_{n} \cap B(m_{0},D\delta_{n})}\frac{1}{2}s_{1}'\tilde{I}_{n,11}\left(\beta_{0},\xi_{0}+\frac{s_{2}}{\sqrt{n}},m\right)s_{1} \leq \frac{C}{16}M_{n}^{2} \right) \longrightarrow 1
\end{align*}
conditionally given $P_{0,W}^{\infty}$-almost every fixed realization $\{w_{i}\}_{i \geq 1}$ of $\{W_{i}\}_{i \geq 1}$ as $n\rightarrow \infty$, and, since $|s_{2}| \leq \tilde{C}$ and $\xi_{01}^{2} \in (0,\infty)$, $\sup_{s:||s||_{2}\leq \tilde{C}}\frac{s_{2}^{2}}{2\xi_{0}^{2}}\leq \frac{C}{16}M_{n}^{2}$ for large $n$. Consequently, I am able to conclude that
\begin{align*}
    &P_{0,YX|W}^{(n)}\left(\inf_{m \in \mathcal{M}_{n} \cap B_{n,\mathcal{M}}(m_{0},D\delta_{n})}\int_{\mathbb{R}^{d_{x}+1}}\exp\left(\ell_{n}^{*}(s,m)-\ell_{n}^{*}(0,m)\right)\pi_{\mathcal{S}_{\tilde{C}}}(s)ds \leq K^{-1}\exp\left(-\frac{C}{2}M_{n}^{2}\right)\right) \\
    &\leq P_{0,YX|W}^{(n)}\left(\left\{\int_{\mathbb{R}^{d_{x}+1}}\exp\left(s'\frac{1}{\sqrt{n}}\tilde{\ell}_{n}(\beta_{0},\xi_{0},m_{0})\right)\pi_{\mathcal{S}_{\tilde{C}}}(s)ds \leq K^{-1}\exp\left(-\frac{C}{4}M_{n}^{2}\right)\right\}\right) + o(1) \\
    &\leq P_{0,YX|W}^{(n)}\left(\frac{1}{\sqrt{n}}\tilde{\ell}_{n}(\beta_{0},\xi_{0},m_{0})'\int_{\mathbb{R}^{d_{x}+1}}s\pi_{\mathcal{S}_{\tilde{C}}}(s)ds  \leq -\frac{C}{8}M_{n}^{2}\right)+ o(1),
\end{align*}
conditionally given $P_{0,W}^{\infty}$-almost every realization $\{w_{i}\}_{i \geq 1}$ of $\{W_{i}\}_{i \geq 1}$, where the last inequality is an application of Jensen's inequality and the fact that $-\log K \leq \frac{C}{8}M_{n}^{2}$ for large $n$. Applying Chebyshev's inequality and the Cauchy-Schwarz inequality, I obtain the following inequalities
\begin{align*}
    &P_{0,YX|W}^{(n)}\left(\frac{1}{\sqrt{n}}\tilde{\ell}_{n}(\beta_{0},\xi_{0},m_{0})'\int_{\mathbb{R}^{d_{x}+1}}s\pi_{\mathcal{S}_{\tilde{C}}}(s)ds  \leq -\frac{C}{8}M_{n}^{2}\right) \\
    &\quad \leq P_{0,YX|W}^{(n)}\left(\left|\frac{1}{\sqrt{n}}\tilde{\ell}_{n}(\beta_{0},\xi_{0},m_{0})'\int_{\mathbb{R}^{d_{x}+1}}s\pi_{\mathcal{S}_{\tilde{C}}}(s)ds\right|  \geq \frac{C}{8}M_{n}^{2}\right) \\
    &\quad \lesssim \frac{1}{M_{n}^{4}}\left| \left| \int_{\mathbb{R}^{d_{x}+1}}s\pi_{\mathcal{S}_{\tilde{C}}}(s)ds\right| \right|_{2}^{2}E_{P_{0,YX|W}^{(n)}}\left|\left| \frac{1}{\sqrt{n}}\tilde{\ell}_{n}(\beta_{0},\sigma_{01}^{2},m_{0})\right| \right|_{2}^{2} \\
    &\quad = O(M_{n}^{-4}),
\end{align*}
for $P_{0,W}^{\infty}$-almost every fixed realization $\{w_{i}\}_{i \geq 1}$ of $\{W_{i}\}_{i \geq 1}$. Since $M_{n} \rightarrow \infty$ as $n\rightarrow \infty$, I can then conclude that
\begin{align*}
P_{0,YX|W}^{(n)}\left(\inf_{m \in \mathcal{M}_{n} \cap B(m_{0},D\delta_{n})}\int_{\mathcal{B} \times \Xi}\exp\left(\ell_{n}(\beta,\xi,m)-\ell_{n}(\beta_{0},\xi_{0},m)\right)d \Pi(\beta,\sigma_{1}^{2}) \leq \exp\left(-\frac{C}{2}M_{n}^{2}\right)\right)\longrightarrow 0
\end{align*}
for $P_{0,W}^{\infty}$-almost every fixed realization $\{w_{i}\}_{i\geq 1}$ of $\{W_{i}\}_{i \geq 1}$ as $n\rightarrow \infty$. This completes the proof.
\end{proof}
\subsection{Extension of Theorem \ref{thm:BVM} to Multivariate $X$ and Unknown $\sigma_{01}^{2}$}
The next result extends Theorem \ref{thm:BVM} to allow for multivariate $X$ and unknown $\sigma_{01}^{2}$. The result is stated in terms of the joint posterior for $(\beta',\xi)'$, however, it implies the same result as Theorem \ref{thm:BVM} for the marginal posterior for $\beta$.
\begin{theorem}\label{thm:BVMmultivariate}
Let $\theta = (\beta',\xi)'$ and let $\tilde{\Delta}_{n,0}= \tilde{I}_{n}^{-1}(\beta_{0},\xi_{0},m_{0})n^{-1/2}\tilde{\ell}_{n}(\beta_{0},\xi_{0},m_{0})$. If Assumptions \ref{as:dgpmulti}, \ref{as:priormultivariate}, \ref{as:consistencymultivariate}, and \ref{as:emp_processmultivariate} hold and $\theta_{0} \in \text{int}(\mathcal{B}\times \Xi)$, then
\begin{align*}
  \left | \left |\Pi\left(\theta \in \cdot |\{(Y_{i},X_{i},w_{i}')'\}_{i=1}^{n}\right) - \mathcal{N}\left(\theta_{0}+\frac{\tilde{\Delta}_{n,0}}{\sqrt{n}},\frac{1}{n}\tilde{I}_{n}^{-1}(\beta_{0},\xi_{0},m_{0}) \right) \right | \right |_{TV} \overset{P_{0,YX|W}^{(n)}}{\longrightarrow} 0
\end{align*}
for $P_{0,W}^{\infty}$-almost every fixed realization $\{w_{i}\}_{i \geq 1}$ of $\{W_{i}\}_{i \geq 1}$ as $n\rightarrow \infty$.
\end{theorem}
\begin{proof}
The argument is very similar to Theorem \ref{thm:BVM} except with appropriate modifications. Following symmetric arguments as Steps 1--3 of the proof of Theorem 1, Assumptions \ref{as:dgpmulti}--\ref{as:emp_processmultivariate} lead me to deduce that
\begin{align*}
    &\Pi\left(\theta \in A\middle | \{(Y_{i},X_{i},w_{i}')'\}_{i=1}^{n}\right) \\
    &\quad= \frac{\int_{A \cap B_{\mathcal{B}\times \Xi}(\theta_{0},\frac{M_{n}}{\sqrt{n}})}\exp\left(\frac{1}{\sqrt{n}}\tilde{\ell}_{n}(\theta_{0},m_{0})'\sqrt{n}(\theta-\theta_{0})-\frac{n}{2}(\theta-\theta_{0})'\tilde{I}_{n}(\theta_{0},m_{0})(\theta-\theta_{0})\right)d \theta}{\int_{B_{\mathcal{B}\times \Xi}(\theta_{0},\frac{M_{n}}{\sqrt{n}})}\exp\left(\frac{1}{\sqrt{n}}\tilde{\ell}_{n}(\theta_{0},m_{0})'\sqrt{n}(\theta-\theta_{0})-\frac{n}{2}(\theta-\theta_{0})'\tilde{I}_{n}(\theta_{0},m_{0})(\theta-\theta_{0})\right)d \theta} + o_{P_{0,YX|W}^{(n)}}(1)
\end{align*}
for all events $A$, conditionally given $P_{0,W}^{\infty}$-almost every realization $\{w_{i}\}_{i \geq 1}$ of $\{W_{i}\}_{i \geq 1}$. Recalling that $\tilde{\Delta}_{n,0} = n^{-\frac{1}{2}}\tilde{I}_{n}^{-1}(\theta_{0},m_{0})\tilde{\ell}_{n}(\theta_{0},m_{0})$, I complete the square to obtain
\begin{align*}
    &\frac{\int_{A \cap B_{\mathcal{B}\times \Xi}(\theta_{0},\frac{M_{n}}{\sqrt{n}})}\exp\left(\frac{1}{\sqrt{n}}\tilde{\ell}_{n}(\theta_{0},m_{0})'\sqrt{n}(\theta-\theta_{0})-\frac{n}{2}(\theta-\theta_{0})'\tilde{I}_{n}(\theta_{0},m_{0})(\theta-\theta_{0})\right)d \theta}{\int_{B_{\mathcal{B}\times \Xi}(\theta_{0},\frac{M_{n}}{\sqrt{n}})}\exp\left(\frac{1}{\sqrt{n}}\tilde{\ell}_{n}(\theta_{0},m_{0})'\sqrt{n}(\theta-\theta_{0})-\frac{n}{2}(\theta-\theta_{0})'\tilde{I}_{n}(\theta_{0},m_{0})(\theta-\theta_{0})\right)d \theta} \\
    &\quad =\frac{\int_{A \cap B_{\mathcal{B}\times \Xi}(\theta_{0},\frac{M_{n}}{\sqrt{n}})}\exp\left(-\frac{1}{2}(\theta-\theta_{0}-n^{-\frac{1}{2}}\tilde{\Delta}_{n,0})'n\tilde{I}_{n}(\theta_{0},m_{0})(\theta-\theta_{0}-n^{-\frac{1}{2}}\tilde{\Delta}_{n,0})\right)d \theta}{\int_{B_{\mathcal{B}\times \Xi}(\theta_{0},\frac{M_{n}}{\sqrt{n}})}\exp\left(-\frac{1}{2}(\theta-\theta_{0}-n^{-\frac{1}{2}}\tilde{\Delta}_{n,0})'n\tilde{I}_{n}(\theta_{0},m_{0})(\theta-\theta_{0}-n^{-\frac{1}{2}}\tilde{\Delta}_{n,0})\right)d \theta} \\
    &\quad = \frac{\int_{A \cap B_{\mathcal{B}\times \Xi}(\theta_{0},\frac{M_{n}}{\sqrt{n}})}\phi_{\theta_{0}+ n^{-\frac{1}{2}}\tilde{\Delta}_{n,0},n^{-1}\tilde{I}_{n}^{-1}(\theta_{0},m_{0})}(\theta)d \theta}{\int_{B_{\mathcal{B}\times \Xi}(\theta_{0},\frac{M_{n}}{\sqrt{n}})}\phi_{\theta_{0}+ n^{-\frac{1}{2}}\tilde{\Delta}_{n,0},n^{-1}\tilde{I}_{n}^{-1}(\theta_{0},m_{0})}(\theta)d \theta},
\end{align*}
where $\phi_{\theta_{0}+ n^{-\frac{1}{2}}\tilde{\Delta}_{n,0},n^{-1}\tilde{I}_{n}^{-1}(\theta_{0},m_{0})}$ is the probability density function of $\mathcal{N}(\theta_{0}+n^{-\frac{1}{2}}\tilde{\Delta}_{n,0},n^{-1}\tilde{I}_{n}^{-1}(\theta_{0}))$. This expression is analogous to that obtained in the final step of the proof of Theorem 1 (cf. equation (\ref{eq:gauss})). Since $\tilde{\Delta}_{n,0}= O_{P_{0,YX|W}^{(n)}}(1)$ and $\tilde{I}_{n}(\theta_{0},m_{0}) = O_{P_{0,YX|W}^{(n)}}(1)$ conditionally given $P_{0,W}^{\infty}$-almost every realization $\{w_{i}\}_{i \geq 1}$ of $\{W_{i}\}_{i \geq 1}$, it can be shown using analogous arguments to the proof of Theorem \ref{thm:BVM} that, if $M_{n}\rightarrow \infty$ arbitrarily slowly, then
\begin{align*}
   \sup_{A}\left| \frac{\int_{A \cap B_{\mathcal{B}\times \Xi}(\theta_{0},\frac{M_{n}}{\sqrt{n}})}\phi_{\theta_{0}+ n^{-\frac{1}{2}}\tilde{\Delta}_{n,0},n^{-1}\tilde{I}_{n}^{-1}(\theta_{0},m_{0})}(\theta)d \theta}{\int_{B_{\mathcal{B}\times \Xi}(\theta_{0},\frac{M_{n}}{\sqrt{n}})}\phi_{\theta_{0}+ n^{-\frac{1}{2}}\tilde{\Delta}_{n,0},n^{-1}\tilde{I}_{n}^{-1}(\theta_{0},m_{0})}(\theta)d \theta}-\int_{A}\phi_{\theta_{0}+ n^{-\frac{1}{2}}\tilde{\Delta}_{n,0},n^{-1}\tilde{I}_{n}^{-1}(\theta_{0},m_{0})}(\theta)d \theta\right| = o_{P_{0,YX|W}^{(n)}}(1)
\end{align*}
for $P_{0,W}^{\infty}$-almost every fixed realization $\{w_{i}\}_{i \geq 1}$ of $\{W_{i}\}_{i \geq 1}$ as $n\rightarrow \infty$. Consequently, an application of the triangle inequality allows me to conclude that
\begin{align*}
    \sup_{A}\left|\Pi\left(\theta \in A\middle | \{(Y_{i},X_{i},w_{i}')'\}_{i=1}^{n}\right) - \int_{A}\phi_{\theta_{0}+ n^{-\frac{1}{2}}\tilde{\Delta}_{n,0},n^{-1}\tilde{I}_{n}^{-1}(\theta_{0},m_{0})}(\theta)d \theta\right| \overset{P_{0,YX|W}^{(n)}}{\longrightarrow} 0
\end{align*}
for $P_{0,W}^{\infty}$-almost every fixed realization $\{w_{i}\}_{i \geq 1}$ of $\{W_{i}\}_{i \geq 1}$ as $n\rightarrow \infty$. This completes the proof.
\end{proof}

\end{document}